\documentclass[12pt, twoside, fleqn, a4paper]{article}

\usepackage{latexsym}  
\usepackage{amscd}    
\usepackage{theorem}  
\usepackage{pifont}  
\usepackage{mathbbol} 
\usepackage{amsfonts}  
\usepackage{xspace}  
\usepackage{amssymb} 
\usepackage{fancyhdr} 
\usepackage{mathrsfs} 
\usepackage{amsmath} 
\usepackage{graphics} 
\usepackage{graphicx} 
\usepackage{euscript}
\usepackage{psfrag}
\usepackage{empheq} 
\usepackage{mathabx} 
\usepackage[parfill]{parskip}     
\usepackage[colorlinks=true, allcolors=blue]{hyperref}
\usepackage{cancel} 
\usepackage{relsize} 

\title{A multiplicative ergodic theorem in \\ random and simultaneous dynamics}
\author{Thirupathi Perumal\footnote{Indian Institute of Science Education and Research Thiruvananthapuram (IISER-TVM), Maruthamala P.O., Vithura, Kerala, India. PIN 695 551. \\ ORCID: 0000 0001 7951 8360\ \ email: \texttt{thirupathip23@iisertvm.ac.in} \\ The first named author thanks the support provided by NBHM through a fellowship with grant No. 0203/5(40)/2024-R\&D-II/11404}\ \ and\ \ Shrihari Sridharan\footnote{Indian Institute of Science Education and Research Thiruvananthapuram (IISER-TVM), Maruthamala P.O., Vithura, Kerala, India. PIN 695 551. \\ ORCID: 0000 0003 2434 4767\ \ email: \texttt{shrihari@iisertvm.ac.in} (Corresponding Author). \\ The second named author thanks the support provided by NBHM through a research grant No. 02011/35/2025/NBHM(R.P)/R\&D II/9832}}

\DeclareFontFamily{OT1}{pzc}{}
\DeclareFontShape{OT1}{pzc}{m}{it}%
              {<-> s * [0.900] pzcmi7t}{}
\DeclareMathAlphabet{\mathpzc}{OT1}{pzc}%
                                 {m}{it}

{\theorembodyfont{\slshape} \newtheorem{theorem}{Theorem}[section]}
{\theorembodyfont{\slshape} \newtheorem{definition}[theorem]{Definition}}
{\theorembodyfont{\slshape} \newtheorem{lemma}[theorem]{Lemma}}
{\theorembodyfont{\slshape} \newtheorem{proposition}[theorem]{Proposition}}
{\theorembodyfont{\slshape} }
{\theorembodyfont{\slshape} \newtheorem{corollary}[theorem]{Corollary}} 
{\theorembodyfont{\slshape} } 
{\theorembodyfont{\slshape} }
{\theorembodyfont{\slshape} }

\numberwithin{equation}{section}

\newenvironment{proof}{\paragraph{Proof:}}{\hfill$\bullet$}

\begin{document}

\maketitle

\begin{abstract} 
\noindent 
In this manuscript, we consider finitely many maps defined on a compact probability space; each of which preserves the probability measure. We consider two aspects of growth of dynamical systems in this scenario; the first setting where we let some infinite lettered word determine the path through which the orbit of points evolve, while the second setting where we consider all possible paths through which the orbit of a set may evolve. Our aim in this paper is to provide a full description of the celebrated multiplicative ergodic theorem, originally due to Oseledets, for such dynamical systems that is point-valued (as in the first case) and set-valued (as in the second case). 
\end{abstract}

\begin{tabular}{l l l} 
{\bf Keywords} & \hspace{+1cm} & Multiplicative ergodic theorem, \\ 
& & Random dynamics, \\ 
& & Simultaneous dynamics, \\ 
& & Lyapunov exponents. \\ 
& & \\
{\bf MSC Subject} & & 37H15, 37H12, 37C35, 37A05. \\ 
{\bf Classifications} & & \\ 
\end{tabular}

\bigskip 

\newpage 

\section{Introduction} 
\label{setting}

The multiplicative ergodic theorem, originally established by Oseledets in \cite{osel:1968}, is one of the cornerstone results in ergodic theory and dynamical systems. Building on an earlier work of Furstenberg and Kesten, as found in \cite{fust:1960}, concerning random products of matrices, the Oseledets' theorem provides a rigorous framework for describing the asymptotic behaviour of products of linear operators along trajectories of dynamical systems. In particular, it establishes the existence of Lyapunov exponents and invariant subspaces, which characterise the exponential growth rates, geometric structure of the underlying dynamics and thereby, gives us vital information concerning the stability of typical orbits. 

Since its introduction, the theorem has undergone several significant extensions and generalisations. Ruelle, in \cite{ruelle:1982} extended the Oseledets' theorem to infinite-dimensional Hilbert spaces, thereby broadening its applicability to a wide range of infinite-dimensional dynamical systems. A geometric interpretation of the theorem was later provided by Kaimanovich, in \cite{kaim:1987}. More recently, Bowen, Hayes and Lin, in \cite{bl:2021}, extended the multiplicative ergodic theorem to cocycles taking values in von Neumann algebras. Further developments and related results can be found in \cite{al:1998, sf:2016} and the references therein. 

In this paper, we study the multiplicative ergodic theorem in the setting of finitely many measure-preserving maps defined on a compact space. We consider two distinct situations. First, we fix an infinite word over a finite alphabet set and study the dynamics generated along the corresponding path. Second, we consider the collection of all possible paths, leading naturally to a set-valued dynamical system, as we shall explain soon. Following the classical approach of Oseledets, we establish analogous multiplicative ergodic results in both settings. The existence of Lyapunov exponents for these two cases has been established by the authors in \cite{ts:pp}, and this work further develops the corresponding ergodic framework and invariant structures. 

The paper is organised as follows: For the remainder of this section, we explain the setting of point-wise dynamics of points that we shall call random dynamical systems and the set-valued dynamics of sets that we shall call simultaneous dynamical systems that we are interested in, throughout this paper. In Section \ref{sec:prelim}, we state a few preliminary results in the field of random dynamical systems and simultaneous dynamical systems; especially focussing on the interesting ones that we state as Theorem \ref{thm:fkomega} and Theorem \ref{thm:fkomegafree}, where the authors obtain an expression for the Lyapunov exponents in both the settings. We state the main results of this paper in Section \ref{main}. As interested readers may find, Theorem \ref{thm:oselomega} and Theorem \ref{thm:oselnoomega} are nothing but the multiplicative ergodic theorem, stated for random and simultaneous dynamical systems respectively. Following up with a few simple lemmas in Section \ref{sec:afsl}, we embark on the proof of Theorem \ref{thm:oselomega} in Section \ref{sec:potomegaab}. The results in the next four sections, namely Section \ref{kingmananalogue}, Section \ref{markov}, Section \ref{invariantmsb} and Section \ref{action} come in handy to complete the proof of Theorem \ref{thm:oselomega} in Section \ref{partcomega}. In the next two sections, namely Section \ref{sec:potnoomegaab} and Section \ref{sec:potnoomegac}, we concern ourselves with the proof of the multiplicative ergodic theorem for the set-valued simultaneous dynamical system, the proof of each statement therein, running in similar spirit to their counterparts in the setting of random dynamical system. 

Consider finitely many maps $T_{1}, T_{2}, \cdots, T_{N}$, all of which act on $\left( X, \mathcal{B}_{X}, \mu \right)$, where $X$ is a compact space, $\mathcal{B}_{X}$ the $\sigma$-algebra of subsets of $X$ and $\mu$ a probability measure supported on $X$, that remains invariant under the action of $T_{i}$ for every $1 \le i \le N$. The dynamics that we investigate in this paper, grows as follows on $X$: Any point $x_{0} \in X$ at time $0$ has an equal probability of being at $x_{n} \in X$ at time $n \in \mathbb{Z}_{+}$, where 
\[ \displaystyle{x_{n} \in \bigg\{ \left( T_{\omega_{n}} \circ \cdots \circ T_{\omega_{1}} \right) (x_{0})\ :\ \omega_{j} \in \left\{ 1, 2, \cdots, N \right\} \bigg\}}. \] 
It is then a simple observation that we have a possibility of $N^{n}$ points, that one may consider as the $n$-th order forward image of $ x_{0}$, thus rendering a richer dynamical system than dealing with a single map on $X$. 

As one may observe, the choice of the combinatorial $n$-lettered word from the set of letters $\{ 1, 2, \cdots, N \}$ dictates the position of the evolution of the forward orbit of $x_{0}$. Hence, defining 
\[ \Sigma_{N}^{n}\ :=\ \left\{ \omega^{n} = \left( \omega^{n}_{1} \omega^{n}_{2} \cdots \omega^{n}_{n} \right)\ :\ \omega^{n}_{j} \in \left\{ 1, 2, \cdots, N \right\} \right\}, \] 
one may consider a specific $n$-lettered word $\omega^{n} \in \Sigma_{N}^{n}$ and study the evolution of the orbit of points in $X$ under the map $T_{\omega^{n}}$ given by $T_{\omega^{n}_{n}} \circ \cdots \circ T_{\omega^{n}_{1}}$ or consider the various possibilities of the point $x_{n} \in X$ pertaining to every $x_{0} \in X$ and $\omega^{n} \in \Sigma_{N}^{n}$. Finally, we let $n \to \infty$ to learn about the eventual evolution of the orbit of $x_{0}$ in $X$, and thus shall also deal with the space of infinitely long words that we denote by $\Sigma_{N}^{+}$ and define as 
\[ \Sigma_{N}^{+}\ :=\ \left\{ \omega = \left( \omega_{1} \omega_{2} \cdots \right)\ :\ \omega_{j} \in \left\{ 1, 2, \cdots, N \right\} \right\}. \] 

For any given $\omega \in \Sigma_{N}^{+}$, let $\pi_{j} \left( \omega \right) = \left( \omega_{1} \omega_{2} \cdots \omega_{j} \right) \in \Sigma_{N}^{j}$, whenever $0 < j < \infty$. One can analogously define $\pi_{j} \left( \omega^{n} \right)$ for any $\omega^{n} \in \Sigma_{N}^{n}$, whenever $0 < j \le n$. Even though $\pi_{0}$ does not make any sense when defined on the space $\Sigma_{N}^{+}$ and $\Sigma_{N}^{n}$, we shall understand $\pi_{0} \left( \cdot \right)$ to be the empty word, that has no letters. It is, of course easy to verify that given any $\omega \in \Sigma_{N}^{+}$, there exists a sequence of $n$-lettered words, $\left\{ \pi_{n} \left( \omega \right) \in \Sigma_{N}^{n} \right\}_{n\, \in\, \mathbb{Z}_{+} \cup \{0\}}$ that grows to become to $\omega$, as $n \to \infty$. Alternatively, considering $\omega^{n} = \pi_{n} \left( \omega \right) \in \Sigma_{N}^{n}$ and denoting the infinite concatenation of $\omega^{n}$ with itself by $\overline{\omega^{n}}$, one can observe that the sequence $\displaystyle{\left\{ \overline{\omega^{n}} \right\}_{n\, \in\, \mathbb{Z}_{+}}}$ converges to the given $\omega$ in $\Sigma_{N}^{+}$. The cylinder sets form the basis for a topology on $\Sigma_{N}^{+}$ and one can consider $\Sigma_{N}^{+}$ to be a measure space by defining the Bernoulli measure $\mathfrak{B}$ dependent on any given probability vector in $\mathbb{R}^{N}$. For more details on $\Sigma_{N}^{+}$, see \cite{kit:98, pw:2000, sb:2020}. 

Let $L : X \longrightarrow {\rm GL}_{d} \left( \mathbb{R} \right)$ be a measurable map. We consider the operator norm on ${\rm GL}_{d} \left( \mathbb{R} \right)$ given by 
\[ \| A \|_{{\rm op}}\ \ :=\ \ \sup_{v\, \in\, \mathbb{R}^{d}\; :\; \|v\|\, =\, 1} \|Av\|,\ \ \ \text{where}\ \ \| \cdot \|\ \text{denotes the Euclidean norm on}\ \mathbb{R}^{d}. \] 
By $\left( L(x) \right)^{\pm 1}$, we denote the matrix $L(x)$ or its inverse, as appropriate. We further assume that $\log^{+} \left\| \left( L(x) \right)^{\pm 1} \right\|_{{\rm op}}$ is an integrable function with respect to the measure $\mu$, where $\log^{+} \left\| A \right\|_{{\rm op}} = \max \left\{ 0, \log \left\| A \right\|_{{\rm op}} \right\}$. 

Define $S : \Sigma_{N}^{+} \times X \times \mathbb{R}^{d} \circlearrowright$ be a map given by 
\begin{equation} 
\label{actionofS} 
S \left( \omega,\, x,\, v \right)\ \ =\ \ \left( \sigma \omega,\, T_{\pi_{1} \left( \omega \right)} x,\, L(x) v \right), 
\end{equation} 
where $\sigma : \Sigma_{N}^{+} \circlearrowright$ is the shift map that shifts the terms in any infinite sequence by one place to the left, dropping the first letter. By virtue of this definition of $S$, for any $n \in \mathbb{Z}_{+}$, one can find the $n$-th iterate of $S$ as 
\begin{eqnarray*} 
S^{n} \left( \omega,\, x,\, v \right) & = & \left( \sigma^{n} \omega,\, T_{\pi_{n} \left( \omega \right)} x,\, \left( L \left( \left( T_{\pi_{n - 1} \left( \omega \right)} \right) x \right) \cdots L \left( \left( T_{\pi_{1} \left( \omega \right)} \right) x \right) L(x) \right) v \right) \\ 
& = & \left( \sigma^{n} \omega,\, T_{\pi_{n} \left( \omega \right)} x,\, \prod_{j\, =\, 0}^{n - 1} \left( L \circ T_{\pi_{j} \left( \omega \right)} (x) \right) v \right), 
\end{eqnarray*} 
where we use the notation $T_{\pi_{j} \left( \omega \right)} = T_{\omega_{j}} \circ \cdots \circ T_{\omega_{1}}$ for $1 \le j \le n$ and $T_{\pi_{0} \left( \omega \right)} (x) \equiv x$. As is customary, we consider 
\begin{equation} 
\label{eqn:s^0} 
S^{0} \left( \omega,\, x,\, v \right)\ \ \equiv\ \ \left( \omega,\, x,\, v \right)\ \ =\ \ \left( \sigma^{0} (\omega),\, T_{\pi_{0} \left( \omega \right)} (x),\, I_{d} (v) \right), 
\end{equation} 
where $I_{d}$ is the identity matrix in ${\rm GL}_{d} \left( \mathbb{R} \right)$. Further, since matrix multiplication is not commutative, we shall always multiply the matrices in the natural order arising from the dynamical system, meaning 
\[ \prod_{j\, =\, 0}^{n - 1} \left( L \circ T_{\pi_{j} \left( \omega \right)} (x) \right)\ \ =\ \ \left( L \circ T_{\pi_{n - 1} \left( \omega \right)} (x) \right) \cdots \left( L \circ T_{\pi_{1} \left( \omega \right)} (x) \right) \left( L \circ T_{\pi_{0} \left( \omega \right)} (x) \right). \] 

Our objective in this paper is to understand the dynamics that arises due to the map $S$ defined on $\Sigma_{N}^{+} \times X \times \mathbb{R}^{d}$, and use the same to investigate the set-valued dynamics due to $R$ that we shall now define on subsets of $X \times \mathbb{R}^{d}$. For any set $E \subseteq X$ and $F \subseteq \mathbb{R}^{d}$, we define 
\begin{equation} 
\label{actionofR} 
R \left( E,\, F \right)\ \ =\ \ \bigcup_{\omega_{1}\, =\, 1}^{N} \bigcup_{x\, \in\, E} \bigcup_{v\, \in\, F} \big\{ \left( T_{\omega_{1}} x,\, \left( L (x) \right) v \right) \big\}, 
\end{equation} 
so that the action of $R^{n}$ on the ordered pair $\left( E,\, F \right)$ can be recursively defined for any $n \in \mathbb{Z}_{+}$ and given by 
\[ R^{n} \left( E,\, F \right)\ \ =\ \ \bigcup_{\omega^{n}\, \in\, \Sigma_{N}^{n}} \bigcup_{x\, \in\, E} \bigcup_{v\, \in\, F} \left\{ \left( T_{\omega^{n}} x,\, \prod_{j\, =\, 0}^{n - 1} \left( L \circ T_{\pi_{j} \left( \omega^{n} \right)} (x) \right) v \right) \right\}. \]

\section{Preliminary results} 
\label{sec:prelim} 

The Birkhoff's ergodic theorem is a useful tool in dynamical systems, that relates the local time average of generic orbits to the global space average by employing a real-valued observable function defined on the domain of dynamics. The theorem states thus: 

\begin{theorem}\cite{pw:2000} 
\label{thm:bet} 
Consider the dynamical system governed by the map $T : X \circlearrowright$ that preserves the probability measure $\mu$ supported on a compact space $X$. For any real-valued observable $f \in \mathscr{L}^{1} (\mu)$, we have 
\[ \lim_{n\, \to\, \infty} \frac{1}{n} \sum_{j\, =\, 0}^{n - 1} \left( f \circ T^{j} x \right)\ \ =\ \ \widetilde{f} (x)\ \ \text{that satisfies}\ \ \int_{X} \widetilde{f} \mathrm{d}\mu\ =\ \int_{X} f \mathrm{d}\mu. \] 
\end{theorem} 

Our primary focus in this paper is to consider finitely maps $T_{i}$ acting on $X$, as has been explained in Section \ref{setting}. For such systems, we initially fix $\omega \in \Sigma_{N}^{+}$ and consider the growth of orbits focusing on the action that takes place in the space $X$ and $\mathbb{R}^{d}$ separately. The analogue of the Kingman ergodic theorem, as proved in \cite{ts:pp}, focuses on the space $X$. 

\begin{theorem}\cite{ts:pp} 
\label{thm:ketomega} 
Consider the maps $T_{i} : X \circlearrowright,\ 1 \le i \le N < \infty$ each of which preserves the probability measure $\mu$ supported on the compact space $X$. Let $f \in \mathscr{L}^{1} (\mu)$. Then, for any fixed $\omega \in \Sigma_{N}^{+}$, the sequence $\displaystyle{\left\{ \dfrac{1}{n} \sum\limits_{j\, =\, 0}^{n - 1} \left( f \circ T_{\pi_{j} \left( \omega \right)} x \right) \right\}_{n\, \in \mathbb{Z}_{+}}}$ converges for $\mu$-almost every $x \in X$ to some measurable function $f_{\omega} : X \longrightarrow [-\infty, \infty)$ that satisfies $f_{\omega} = f_{\omega} \circ T_{\pi_{n} \left( \omega \right)}$ for every $n \in \mathbb{Z}_{+}$ and for $\mu$-almost every $x \in X$. Moreover, 
\[ \int f_{\omega} \mathrm{d}\mu\ \ =\ \ \lim_{n\, \to\, \infty} \frac{1}{n} \int \sum\limits_{j\, =\, 0}^{n - 1} \left( f \circ T_{\pi_{j} \left( \omega \right)} \right) \mathrm{d}\mu\ \ =\ \ \inf_{n\, \ge\, 1} \frac{1}{n} \int \sum\limits_{j\, =\, 0}^{n - 1} \left( f \circ T_{\pi_{j} \left( \omega \right)} \right) \mathrm{d}\mu. \] 
\end{theorem} 

For the same setting as in Theorem \ref{thm:ketomega}, we now look at the effect of the dynamics on $\mathbb{R}^{d}$, as stated in the following theorem that is analogous to the Furstenberg-Kesten theorem. 

\begin{theorem}\cite{ts:pp} 
\label{thm:fkomega} 
Consider the maps $T_{i} : X \circlearrowright,\ 1 \le i \le N < \infty$ each of which preserves the probability measure $\mu$ supported on the compact space $X$. Suppose $L : X \longrightarrow {\rm GL}_{d} \left( \mathbb{R} \right)$ is a measurable map such that $\log^{+} \left\| L(x)^{\pm 1} \right\|_{{\rm op}} \in \mathscr{L}^{1} (\mu)$. Then, for any fixed $\omega \in \Sigma_{N}^{+}$, the extremal Lyapunov exponents given by 
\begin{equation} 
\label{lambdaplmiomega} 
\lambda_{+}^{\omega} (x)\ =\ \lim_{n\, \to\, \infty} \frac{1}{n} \log \left\| \prod_{j\, =\, 0}^{n - 1} \left( L \left( T_{\pi_{j} \left( \omega \right)} x \right) \right) \right\|_{{\rm op}}\ \ \text{and}\ \ \lambda_{-}^{\omega} (x)\ =\ \lim_{n\, \to\, \infty} \frac{1}{n} \log \left\| \prod_{j\, =\, 0}^{n - 1} \left( L \left( T_{\pi_{j} \left( \omega \right)} x \right) \right)^{-1} \right\|_{{\rm op}}^{-1}, 
\end{equation} 
exist for $\mu$-almost every $x \in X$. The functions $\lambda_{\pm}^{\omega}$ satisfy $\lambda_{\pm}^{\omega} = \lambda_{\pm}^{\omega} \circ T_{\pi_{n} \left( \omega \right)}$ for every $n \in \mathbb{Z}_{+}$ and for $\mu$-almost every $x \in X$. Further, 
\begin{eqnarray*} 
\int \lambda_{+}^{\omega} \mathrm{d}\mu & = & \lim_{n\, \to\, \infty} \frac{1}{n} \int \log \left\| \prod_{j\, =\, 0}^{n - 1} \left( L \circ T_{\pi_{j} \left( \omega \right)} \right) \right\|_{{\rm op}} \mathrm{d}\mu\ \ \text{and} \\ 
\int \lambda_{-}^{\omega} \mathrm{d}\mu & = & \lim_{n\, \to\, \infty} \frac{1}{n} \int \log \left\| \prod_{j\, =\, 0}^{n - 1} \left( L \circ T_{\pi_{j} \left( \omega \right)} \right)^{-1} \right\|_{{\rm op}}^{-1} \mathrm{d}\mu. 
\end{eqnarray*} 
\end{theorem} 

The authors in \cite{aswin:2024}, studied the case where all maps $T_{i};\ 1 \le i \le N < \infty$ act on the unit interval $[0, 1)$ simultaneously, and obtained the analogue of Birkhoff's ergodic theorem as thus: 

\begin{theorem}\cite{aswin:2024} 
\label{thm:betm} 
Let $T_{i} : [0, 1) \circlearrowright,\ 1 \le i \le N < \infty$ be a finite collection of interval maps given by $T_{i} (x) = (i + 1) x (\mod 1)$. Then, for any real-valued function $f \in \mathscr{L}^{1} (\ell)$, we have 
\[ \lim_{n\, \to\, \infty} \frac{1}{n} \frac{1}{N^{n}} \sum_{\omega^{n}\, \in\, \Sigma_{N}^{n}}\; \sum_{j\, =\, 0}^{n - 1} \left( f \circ T_{\pi_{j} \left( \omega^{n} \right)} x \right)\ \ =\ \ \int_{[0, 1)} f \mathrm{d}\ell, \] 
for $\ell$-almost every $x \in [0, 1)$, where $\ell$ denotes the Lebesgue measure. 
\end{theorem} 

Even though the results in \cite{aswin:2024} have been proved for simultaneous action of interval maps, one can generalise the same to simultaneous dynamics as we consider in this paper. Moreover, the authors have proved the Kingman ergodic theorem, in such generalised setting, in \cite{ts:pp}, that we now state.

\begin{theorem}\cite{ts:pp} 
\label{kingmanomegafreep} 
Consider the maps $T_{i} : X \longrightarrow X,\ 1 \le i \le N < \infty$ each of which preserves the probability measure $\mu$ supported on the compact space $X$. Let $f \in \mathscr{L}^{1} (\mu)$. Then, the sequence $\displaystyle{\left\{ \dfrac{1}{n} \frac{1}{N^{n}} \sum_{\omega^{n}\, \in\, \Sigma_{N}^{n}}\; \sum_{j\, =\, 0}^{n - 1} \left( f \circ T_{\pi_{j} \left( \omega^{n} \right)} x \right) \right\}_{n\, \in \mathbb{Z}_{+}}}$ converges for $\mu$-almost every $x \in X$ to some measurable function $F : X \longrightarrow [-\infty, \infty)$ that satisfies $F = F \circ T_{i}$ for every $i \in \left\{ 1, 2, \cdots, N \right\}$ and for $\mu$-almost every $x \in X$. Moreover, 
\[ \int F \mathrm{d}\mu\ =\ \lim_{n\, \to\, \infty} \frac{1}{n} \frac{1}{N^{n}} \sum_{\omega^{n}\, \in\, \Sigma_{N}^{n}} \int \sum_{j\, =\, 0}^{n - 1} \left( f \circ T_{\pi_{j} \left( \omega^{n} \right)} \right) \mathrm{d}\mu\ =\ \inf_{n\, \ge\, 1} \frac{1}{n} \frac{1}{N^{n}} \sum_{\omega^{n}\, \in\, \Sigma_{N}^{n}} \int \sum_{j\, =\, 0}^{n - 1} \left( f \circ T_{\pi_{j} \left( \omega^{n} \right)} \right) \mathrm{d}\mu. \]
\end{theorem} 

Finally, we state the Furstenberg-Kesten theorem for such simultaneous actions as was considered in Theorem \ref{kingmanomegafreep}. 

\begin{theorem}\cite{ts:pp} 
\label{thm:fkomegafree} 
Consider the maps $T_{i} : X \circlearrowright,\ 1 \le i \le N < \infty$ each of which preserves the probability measure $\mu$ supported on the compact space $X$. Let $\log^{+} \left\| L(x)^{\pm 1} \right\|_{{\rm op}} \in \mathscr{L}^{1} (\mu)$. Then, the extremal Lyapunov exponents of the simultaneous dynamical system governed by the finite set of maps under consideration, 
\begin{eqnarray*} 
\Lambda_{+} (x) & = & \lim_{n\, \to\, \infty} \frac{1}{n} \frac{1}{N^{n}} \sum_{\omega^{n} \in \Sigma_{N}^{n}} \log \left\| \prod_{j\, =\, 0}^{n - 1} \left( L \left( T_{\pi_{j} \left( \omega^{n} \right)} x \right) \right) \right\|_{{\rm op}}\ \ \ \ \text{and} \\ 
\Lambda_{-} (x) & = & \lim_{n\, \to\, \infty} \frac{1}{n} \frac{1}{N^{n}} \sum_{\omega^{n} \in \Sigma_{N}^{n}} \log \left\| \prod_{j\, =\, 0}^{n - 1} \left( L \left( T_{\pi_{j} \left( \omega^{n} \right)} x \right) \right)^{-1} \right\|_{{\rm op}}^{-1}, 
\end{eqnarray*} 
exist for $\mu$-almost every $x \in X$. Further, the functions $\Lambda_{\pm}$ satisfy $\Lambda_{\pm} = \Lambda_{\pm} \circ T_{i}$ for every $i \in \left\{ 1, 2, \cdots, N \right\}$. Moreover, 
\begin{eqnarray*} 
\int \Lambda_{+} \mathrm{d}\mu & = & \lim_{n\, \to\, \infty} \frac{1}{n} \frac{1}{N^{n}} \int \sum_{\omega^{n} \in \Sigma_{N}^{n}} \log \left\| \prod_{j\, =\, 0}^{n - 1} \left( L \left( T_{\pi_{j} \left( \omega^{n} \right)} x \right) \right) \right\|_{{\rm op}} \mathrm{d}\mu\ \ \ \ \text{and} \\ 
\int \Lambda_{-} \mathrm{d}\mu & = & \lim_{n\, \to\, \infty} \frac{1}{n} \frac{1}{N^{n}} \int \sum_{\omega^{n} \in \Sigma_{N}^{n}} \log \left\| \prod_{j\, =\, 0}^{n - 1} \left( L \left( T_{\pi_{j} \left( \omega^{n} \right)} x \right) \right)^{-1} \right\|_{{\rm op}}^{-1} \mathrm{d}\mu. 
\end{eqnarray*} 
\end{theorem} 

\section{Main theorems} 
\label{main} 

Equipped with the necessary dynamical notions that we are interested in, we state the main results of this manuscript in this section. We begin with an elementary definition, as may be found in \cite{mv:2014}, and follow it up with a theorem due to Oseledets, called the \emph{multiplicative ergodic theorem}, as may be found in \cite{osel:1968, mv:2014}. 

\begin{definition} 
A finite increasing sequence of vector subspaces of $\mathbb{R}^{d}$, say $\left\{ 0 \right\} \subsetneq V_{1} \subsetneq \cdots \subsetneq V_{r}$ is called a \emph{flag}. If $r \equiv d$ and $\dim (V_{k}) = k$ for all $1 \le k \le r$, the flag is said to be \emph{complete}. 
\end{definition} 

\begin{theorem}\cite{osel:1968, mv:2014} 
\label{thm:osel}
Consider the map $T : X \circlearrowright$ that preserves the probability measure $\mu$ supported on the compact space $X$. For $L : X \longrightarrow {\rm GL}_{d} \left( \mathbb{R} \right)$, assume that the function $\log^{+} \| \left( L(x) \right)^{\pm 1} \|_{{\rm op}} \in \mathscr{L}^{1} (\mu)$. Then, for $\mu$-almost every $x \in X$, there exists 
\begin{enumerate} 
\item a positive integer, say $r \equiv r \left( x \right)$; 
\item real numbers ordered as $\lambda_{1} \left( x \right)\ <\ \lambda_{2} \left( x \right)\ <\ \cdots\ <\ \lambda_{r} \left( x \right)$; and 
\item a flag $V_{\left( x \right)}^{0} \equiv \left\{ 0 \right\} \subsetneq V_{\left( x \right)}^{1} \subsetneq V_{\left( x \right)}^{2} \subsetneq \cdots \subsetneq V_{\left( x \right)}^{r} = \mathbb{R}^{d}$, 
\end{enumerate} 
such that for all $1 \le k \le r$, we have the following: 
\begin{enumerate} 
\item[A.] $(a)\ r \left( T x \right) \equiv r \left( x \right);\ \ \ \ (b)\ \lambda_{k} \left( T x \right) \equiv \lambda_{k} \left( x \right)\ \ \text{and}\ \ \ \ (c)\ L(x) V_{\left( x \right)}^{k} \equiv V_{\left( T x \right)}^{k}$, 
\item[B.] the maps $(a)\ x \longmapsto r(x),\ \ \ \ (b)\ x \longmapsto \lambda_{k} (x)\ \ \text{and}\ \ \ \ (c)\ x \longmapsto V_{(x)}^{k}$ are measurable, 
\item[C.] for every $v \in V_{\left( x \right)}^{k} \setminus V_{\left( x \right)}^{k - 1}$, we have 
\[ \lim\limits_{n\, \to\, \infty} \dfrac{1}{n} \log \| L^{n} (x) v \|\ \ =\ \ \lambda_{k} (x),\ \ \ \text{where}\ \ L^{n} (x)\ =\ \prod_{j\, =\, 0}^{n - 1} L \left( T^{j} (x) \right). \]  
\end{enumerate} 
\end{theorem} 

We now state the multiplicative ergodic theorem for the dynamical system arising from $S$ and the set-valued dynamics arising from $R$ defined on $\Sigma_{N}^{+} \times X \times \mathbb{R}^{d}$ and $X \times \mathbb{R}^{d}$ respectively, that we shall prove one after another in this paper. 

\begin{theorem} 
\label{thm:oselomega} 
Consider the maps $T_{i} : X \circlearrowright,\ 1 \le i \le N < \infty$ each of which preserves the probability measure $\mu$ supported on the compact space $X$. Let $S$ be defined on the product space $\Sigma_{N}^{+} \times X \times \mathbb{R}^{d}$, as defined in Equation \eqref{actionofS}. Let $L : X \longrightarrow {\rm GL}_{d} \left( \mathbb{R} \right)$ be such that $\log^{+} \| \left( L(x) \right)^{\pm 1} \|_{{\rm op}} \in \mathscr{L}^{1} (\mu)$. Then, for $\left( \mathfrak{B} \times \mu \right)$-almost every $\left( \omega,\, x \right) \in \Sigma_{N}^{+} \times X$, there exists 
\begin{enumerate} 
\item a positive integer, say $r \equiv r \left( \omega, x \right)$; 
\item real numbers ordered as $\lambda_{1} \left( \omega, x \right)\ <\ \lambda_{2} \left( \omega, x \right)\ <\ \cdots\ <\ \lambda_{r} \left( \omega, x \right)$; and 
\item a flag $V_{\left(\omega,\, x\right)}^{0} \equiv \left\{ 0 \right\} \subsetneq V_{\left(\omega,\, x\right)}^{1} \subsetneq V_{\left(\omega,\, x\right)}^{2} \subsetneq \cdots \subsetneq V_{\left(\omega,\, x\right)}^{r} = \mathbb{R}^{d}$, 
\end{enumerate} 
such that for all $1 \le k \le r$, we have the following: 
\begin{enumerate} 
\item[A.] \begin{enumerate} \item $r \left( \omega, x \right) \equiv r \left( \sigma \omega, T_{\pi_{1} \left( \omega \right)} x \right) \equiv r \left( \sigma^{2} \omega, T_{\pi_{2} \left( \omega \right)} x \right) \equiv \cdots$; 
\item $\lambda_{k} \left( \omega, x \right) \equiv \lambda_{k} \left( \sigma \omega, T_{\pi_{1} \left( \omega \right)} x \right) \equiv \lambda_{k} \left( \sigma^{2} \omega, T_{\pi_{2} \left( \omega \right)} x \right) \equiv \cdots$ and 
\item $L \left( T_{\pi_{j} \left( \omega \right)} x \right) V_{\left( \sigma^{j} \omega,\, T_{\pi_{j} \left( \omega \right)} x \right)}^{k} \equiv V_{\left( \sigma^{j + 1} \omega,\, T_{\pi_{j + 1} \left( \omega \right)} x \right)}^{k}$, for every $j \ge 0$, 
\end{enumerate} 
\item[B.] The maps $(a)\ (\omega, x) \longmapsto r(\omega, x),\ \ \ (b)\ (\omega, x) \longmapsto \lambda_{k} (\omega, x)\ \ \text{and}\ \ \ (c)\ (\omega, x) \longmapsto V_{\left( \omega,\, x \right)}^{k}$ are measurable, 
\item[C.] for every $v \in V_{\left( \omega,\, x \right)}^{k} \setminus V_{\left( \omega,\, x \right)}^{k - 1}$, we have 
\[ \lim\limits_{n\, \to\, \infty} \dfrac{1}{n} \log \| L_{\omega}^{n} (x) v \|\ \ =\ \ \lambda_{k} \left( \omega, x \right),\ \ \text{where}\ \ L_{\omega}^{n} (x) = \prod_{j\, =\, 0}^{n - 1} L \left( T_{\pi_{j} \left( \omega \right)} (x) \right). \] 
\end{enumerate} 
\end{theorem} 

Finally, we state the multiplicative ergodic theorem for simultaneous dynamics that we intend to study under $R$ defined on $X \times \mathbb{R}^{d}$, where we no longer have the luxury provided by $\Sigma_{N}^{+}$, that essentially led us in the direction of the evolution of dynamics. 

\begin{theorem} 
\label{thm:oselnoomega}
Consider the maps $T_{i} : X \circlearrowright,\ 1 \le i \le N < \infty$ each of which preserves the probability measure $\mu$ supported on the compact space $X$. Let $R$ be the set-valued map defined on subsets of the product space $X \times \mathbb{R}^{d}$, as given in Equation \eqref{actionofR}. Let $L : X \longrightarrow {\rm GL}_{d} \left( \mathbb{R} \right)$ be such that $\log^{+} \left\| \left( L(x) \right)^{\pm 1} \right\|_{{\rm op}} \in \mathscr{L}^{1} (\mu)$. Then, for $\mu$-a.e. $x \in X$, defining $E_{0} = \left\{ x \right\}$, there exists 
\begin{enumerate} 
\item a positive integer, say $r \equiv r \left( E_{0} \right)$; 
\item real numbers ordered as $\Lambda_{1} \left( E_{0} \right)\ <\ \Lambda_{2} \left( E_{0} \right)\ <\ \cdots\ <\ \Lambda_{r} \left( E_{0} \right)$; and 
\item a flag $V_{E_{0}}^{0} \equiv \left\{ 0 \right\} \subsetneq V_{E_{0}}^{1} \subsetneq V_{E_{0}}^{2} \subsetneq \cdots \subsetneq V_{E_{0}}^{r} = \mathbb{R}^{d}$, 
\end{enumerate} 
such that for all $1 \le k \le r$, we have the following: Suppose $E_{l} = \bigcup\limits_{\omega^{l}\, \in \Sigma_{N}^{l}} T_{\omega^{l}} \left( x \right)$ for every $l \in \mathbb{Z}_{+}$, then 
\begin{enumerate} 
\item[A.] \begin{enumerate} \item $r \left( E_{0} \right) \equiv r \left( E_{1} \right) \equiv r \left( E_{2} \right) \equiv \cdots$; 
\item $\Lambda_{k} \left( E_{0} \right) \equiv \Lambda_{k} \left( E_{1} \right) \equiv \Lambda_{k} \left( E_{2} \right) \equiv \cdots$;  
\item $\sum\limits_{\omega^{l}\, \in\, \Sigma_{N}^{l}} L \left( T_{\omega^{l}} x \right) V_{E_{l}}^{k} \equiv V_{E_{l + 1}}^{k}$, for every $l \in \mathbb{Z}_{+} \cup \{ 0 \}$, 
\end{enumerate} 
\item[B.] The maps $(a)\ x \longmapsto r \left( E_{0} \right),\ \ \ (b)\ x \longmapsto \Lambda_{k} \left( E_{0} \right)\ \ \text{and}\ \ \ (c)\ x \longmapsto V_{E_{0}}^{k}$ are measurable, 
\item[C.] for every $v \in V_{E_{0}}^{k} \setminus V_{E_{0}}^{k - 1}$, we have 
\[ \lim\limits_{n\, \to\, \infty} \dfrac{1}{n} \dfrac{1}{N^{n}} \sum_{\omega^{n}\, \in\, \Sigma_{N}^{n}} \log \left\| \left( L_{\omega^{n}}^{n} (x) \right) v \right\|\ \ =\ \ \Lambda_{k} \left( x \right),\ \ \text{where}\ \ L_{\omega^{n}}^{n} (x) = \prod_{j\, =\, 0}^{n - 1} L \left( T_{\pi_{j} \left( \omega^{n} \right)} (x) \right). \] 
\end{enumerate} 
\end{theorem} 

\section{A few simple lemmas} 
\label{sec:afsl} 

In this section, we concentrate on the action of $S$ on the product space $\Sigma_{N}^{+} \times X \times \mathbb{R}^{d}$, as described in Equation \eqref{actionofS}. We start by defining the following quantity. For any triple $\left( \omega,\, x,\, v \right) \in \Sigma_{N}^{+} \times X \times \mathbb{R}^{d}$, define 
\begin{equation} 
\label{eqn:loxv}
\lambda \left( \omega,\, x,\, v \right)\ \ =\ \ \limsup_{n\, \to\, \infty} \frac{1}{n} \log \left\| \prod_{j\, =\, 0}^{n - 1} \left( L \circ T_{\pi_{j} \left( \omega \right)} x \right) v \right\|, 
\end{equation} 
where we allow $\lambda \left( \omega,\, x,\, v \right)$ to take the value $- \infty$ when we consider $v \equiv 0 \in \mathbb{R}^{d}$. Thus, with the caveat that $\lambda = - \infty$ for the zero vector only, we shall now be interested in the quantity $\lambda$ when $v \in \mathbb{R}^{d}_{*} = \mathbb{R}^{d} \setminus \{ 0 \}$. Then, for any vector $v \in \mathbb{R}^{d}_{*}$, we have from the definition that 
\begin{equation} 
\label{lamda(cv)} 
\lambda \left( \omega,\, x,\, c v \right)\ =\ \lambda \left( \omega,\, x,\, v \right)\ \forall c \ne 0. 
\end{equation} 
Moreover, suppose $\{ a_{m} \}$ and $\{ b_{m} \}$ are positive sequences, then, 
\[ \limsup_{m\, \to\, \infty} \frac{1}{m} \log (a_{m} + b_{m})\ \ =\ \ \max \left\{ \limsup_{m\, \to\, \infty} \frac{1}{m} \log a_{m},\ \limsup_{m\, \to\, \infty} \frac{1}{m} \log b_{m} \right\}. \] 
Using the above fact, we obtain for any pair of vectors $v_{1}, v_{2} \in \mathbb{R}^{d}_{*}$, 
\begin{equation} 
\label{lamda(v1+v2)} 
\lambda \left( \omega,\, x,\, v_{1} + v_{2} \right)\ \le\ \max \left\{ \lambda \left( \omega,\, x,\, v_{1} \right),\ \lambda \left( \omega,\, x,\, v_{2} \right) \right\}. 
\end{equation} 
Moreover, 
\begin{equation} 
\label{lamdax=tomnx}
\lambda \left( \sigma \omega,\, T_{\pi_{1} \left( \omega \right)} x,\, L(x) v \right)\ \ =\ \ \limsup_{n\, \to\, \infty} \frac{1}{n} \log \left\| \prod_{j\, =\, 0}^{n} \left( L \circ T_{\pi_{j} \left( \omega \right)} x \right) v \right\|\ \ =\ \ \lambda \left( \omega,\, x,\, v \right). 
\end{equation} 

We now write a few lemmas and include short proofs of the same, for the convenience of the readers. 

\begin{lemma} 
\label{lambdapmbound}
Let $\omega \in \Sigma_{N}^{+}$ be fixed. Then, for $\mu$-almost every $x \in X$ and for any vector $v \in \mathbb{R}^{d}_{*}$, we have 
\[ \lambda_{-}^{\omega} \left( x \right)\ \ \le\ \ \lambda \left( \omega,\, x,\, v \right)\ \ \le\ \ \lambda_{+}^{\omega} \left( x \right). \]
\end{lemma} 

\begin{proof} 
For any fixed $\omega \in \Sigma_{N}^{+}$ and $x \in X$, recall the notation $\displaystyle{L_{\omega}^{n} (x) = \prod\limits_{j\, =\, 0}^{n - 1} \left( L \circ T_{\pi_{j} \left( \omega \right)} x \right)}$. Note that $L_{\omega}^{n} (x)$ is an invertible matrix, since each of the matrices in the product is invertible. We now fix some $v \in \mathbb{R}^{d}_{*}$ and observe that $\| v \| \le \left\| \left( L_{\omega}^{n} (x) \right)^{-1} \right\| \left\| L_{\omega}^{n} (x) v \right\|$. Since $\left\| \left( L_{\omega}^{n} (x) \right)^{-1} \right\| \ne 0$, we write this as 
\[ \left\| \left( L_{\omega}^{n} (x) \right)^{-1} \right\|^{-1} \|v \|\ \le\ \left\| \left( L_{\omega}^{n} (x) \right) v \right\|\ \le\ \left\| \left( L_{\omega}^{n} (x) \right) \right\| \| v \|. \] 
Thus, owing to Theorem \ref{thm:fkomega}, we have for $\mu$-almost every $x \in X$, 
\begin{eqnarray*} 
\limsup_{n\, \to\, \infty} \frac{1}{n} \log \left( \left\| \left( L_{\omega}^{n} (x) \right)^{-1} \right\|^{-1} \|v \| \right) & \le & \limsup_{n\, \to\, \infty} \frac{1}{n} \log \left( \left\| \left( L_{\omega}^{n} (x) \right) v \right\| \right) \\ 
& \le & \limsup_{n\, \to\, \infty} \frac{1}{n} \log \left( \left\| \left( L_{\omega}^{n} (x) \right) \right\| \| v \| \right), 
\end{eqnarray*} 
proving the inequality mentioned in the Lemma. 
\end{proof} 

\begin{lemma} 
\label{fonrd}
Suppose the function $g : \mathbb{R}^{d} \longrightarrow \mathbb{R} \cup \{ - \infty \}$ satisfies the following properties for all non-zero vectors $v_{1}, v_{2}$ and $v$ in $\mathbb{R}^{d}$: 
\begin{enumerate} 
\item $g(v_{1} + v_{2}) \le \max \left\{ g(v_{1}),\; g(v_{2}) \right\}$, 
\item $g(c v) = g(v)$ for all $c \ne 0$, 
\item $g(0) = - \infty$. 
\end{enumerate} 
Then, 
\begin{enumerate} 
\item[(a)] if $g(v_{1}) \ne g(v_{2})$, then $g(v_{1} + v_{2}) = \max \left\{ g(v_{1}),\; g(v_{2}) \right\}$, 
\item[(b)] if there exists a collection of $m$ vectors, say $V = \left\{ v_{1}, v_{2}, \cdots, v_{m} \right\}$ in $\mathbb{R}^{d}_{*}$ such that the values $g(v_{1}), g(v_{2}), \cdots, g(v_{m})$ are distinct, then $V$ is a collection of linearly independent vectors, 
\item[(c)] $g$ attains atmost $d$ distinct values in $\mathbb{R}$. 
\end{enumerate} 
\end{lemma} 

\begin{proof} 
\begin{enumerate} 
\item[{\sl (c)}] This follows from the fact that there can exist only atmost $d$ many linearly independent vectors in $\mathbb{R}^{d}$ and the assertion as in statement {\sl (b)}, which we now prove. 

\item[{\sl (b)}] Let $V = \left\{ v_{1}, v_{2}, \cdots, v_{m} \right\}$ be a collection of linearly dependent vectors in $\mathbb{R}^{d}_{*}$ such that the values $g(v_{1}), g(v_{2}), \cdots, g(v_{m})$ are distinct. Then, there exists at least one non-zero constant among the real numbers $\alpha_{1}, \alpha_{2}, \cdots, \alpha_{m}$ satisfying $\alpha_{1} v_{1} + \alpha_{2} v_{2} + \cdots + \alpha_{m} v_{m} = 0$. However, owing to hypothesis $(2)$ and the assertion in {\sl (a)}, we get a contradiction, as explained below. 
\[ - \infty\ =\ g(0)\ =\ g \left( \alpha_{1} v_{1} + \alpha_{2} v_{2} + \cdots + \alpha_{m} v_{m} \right)\ =\ \max_{1\, \le\, j\, \le\, m} \left\{ g (v_{j}) \right\}. \] 
Thus, the collection of vectors in $V$ is necessarily linearly independent. We now prove {\sl (a)}. 

\item[{\sl (a)}] Let $v_{1}, v_{2} \in \mathbb{R}^{d}_{*}$ such that $v_{1} \ne - v_{2}$ and $g(v_{1}) \ne g(v_{2})$. Assume without loss of generality that $g(v_{1}) < g(v_{2})$. Then, our aim is to prove that $g(v_{1} + v_{2}) = g(v_{2})$. 

Consider $g(v_{2}) = g(v_{1} + v_{2} - v_{1}) \le \max \left\{ g(v_{1} + v_{2}), g(v_{1}) \right\}$, where we have used first two enumerated statements in the hypothesis of the Theorem. Suppose the maximum in the above inequality is achieved by the quantity $g(v_{1})$, then it leads to a contradiction to our assumption that $g(v_{1}) < g(v_{2})$. Thus, $g(v_{1} + v_{2}) = g(v_{2})$. 
\end{enumerate} 
\end{proof} 

\section{Proof of Theorem \ref{thm:oselomega}: Parts $A$ and $B$} 
\label{sec:potomegaab}

In this section, we consider the action of the map $S$ on $\Sigma_{N}^{+} \times X \times \mathbb{R}^{d}$ and prove parts $A$ and $B$ of Theorem \ref{thm:oselomega}. In other words, we consider some $\left( \mathfrak{B} \times \mu \right)$-typical point $\left( \omega,\, x \right) \in \Sigma_{N}^{+} \times X$ and complete the proof of the parts $A$ and $B$ of Theorem \ref{thm:oselomega}, one after another, in the ensuing subsections. 

\subsection{Proof of Part $A$} 

For some $\left( \mathfrak{B} \times \mu \right)$-typical point $\left( \omega,\, x \right) \in \Sigma_{N}^{+} \times X$, observe that $\lambda \left( \omega,\, x,\, \cdot \right)$, as defined in Equation \eqref{eqn:loxv} can be viewed as a function from $\mathbb{R}^{d}$ to $\mathbb{R} \cup \{ - \infty \}$ that satisfies the hypothesis of $g$ in Lemma \ref{fonrd}, thus rendering the assertions on $g$ to $\lambda \left( \omega,\, x,\, \cdot \right)$. We make use of this to prove part $A$ of Theorem \ref{thm:oselomega}. 

\begin{proof} \big(of part $A$ of Theorem \ref{thm:oselomega}\big) 
Consider a $\left( \mathfrak{B} \times \mu \right)$-typical point $\left( \omega,\, x \right) \in \Sigma_{N}^{+} \times X$ and define $U_{\left( \omega,\, x \right)} = \left\{ \lambda \left( \omega,\, x,\, v \right)\ :\ v \in \mathbb{R}^{d}_{*} \right\}$. Then, owing to Lemma \ref{fonrd}, the cardinality of the set $U_{\left( \omega,\, x \right)}$ is atmost $d$. We define $r \equiv r \left( \omega,\, x \right) = \# U_{\left( \omega,\, x \right)}$ and arrange the elements in the set $U_{\left( \omega,\, x \right)}$ in increasing order given by 
\begin{equation} 
\label{lambdaarranged} 
\lambda_{1} \left( \omega,\, x \right)\ <\ \lambda_{2} \left( \omega,\, x \right)\ <\ \cdots\ <\ \lambda_{r} \left( \omega,\, x \right). 
\end{equation} 
Now for every $1 \le k \le r$, define 
\[ V^{k}_{\left( \omega,\, x \right)}\ \ =\ \ \left\{ v \in \mathbb{R}^{d}_{*}\ :\ \lambda \left( \omega,\, x,\, v \right) \le \lambda_{k} \left( \omega,\, x \right) \right\} \cup \{ 0 \}. \] 
Then, owing to the properties of $\lambda \left( \omega,\, x,\, v \right)$, as stated in Equations \eqref{lamda(cv)} and \eqref{lamda(v1+v2)}, it is obvious that $V^{k}_{\left( \omega,\, x \right)}$ is a collection of nested vector subspaces of $\mathbb{R}^{d}$ for every $1 \le k \le r$, {\it i.e.}, 
\begin{equation} 
\label{nestedvs} 
V^{1}_{\left( \omega,\, x \right)}\ \subsetneq\ V^{2}_{\left( \omega,\, x \right)}\ \subsetneq\ \cdots\ \subsetneq\ V^{r}_{\left( \omega,\, x \right)}\ \equiv\ \mathbb{R}^{d}. 
\end{equation} 
Define $V^{0}_{\left( \omega,\, x \right)} = \left\{ 0 \right\}$. Then, our definition of the collection of vector subspaces $\left\{ V^{k}_{\left( \omega,\, x \right)} \right\}$ for $0 \le k \le r$ and our ordering of the numbers $\left\{ \lambda_{1} \left( \omega,\, x \right), \cdots, \lambda_{r} \left( \omega,\, x \right) \right\}$ ensures that 
\begin{equation} 
\label{eqn:lam} 
\lambda \left( \omega,\, x,\, v \right) = \lambda_{k} \left( \omega,\, x \right)\ \ \ \ \forall v \in V^{k}_{\left( \omega,\, x \right)} \setminus V^{k -1}_{\left( \omega,\, x \right)}, 
\end{equation} 
as otherwise, it would result in an increase in the cardinality of $U_{\left( \omega,\, x \right)}$. Also, one observes from Equation \eqref{lamdax=tomnx} and the definition of $U_{\left( \omega,\, x \right)}$ that $U_{\left( \omega,\, x \right)} \equiv U_{\left( \sigma \omega,\, T_{\pi_{1} \left( \omega \right)} x \right)}$ that results in 
\[ r \left( \omega,\, x \right)\ \equiv\ r \left( \sigma \omega,\, T_{\pi_{1} \left( \omega \right)} x \right)\ \ \ \text{and}\ \ \ \lambda_{k} \left( \omega,\, x \right)\ \equiv\ \lambda_{k} \left( \sigma \omega,\, T_{\pi_{1} \left( \omega \right)} x \right)\ \ \forall 1 \le k \le r. \] 
Finally, note that 
\begin{eqnarray*} 
L(x) V^{k}_{\left( \omega,\, x \right)} & = & \left\{ L(x) v \in \mathbb{R}^{d}_{*}\ :\ \lambda \left( \omega,\, x,\, v \right) \le \lambda_{k} \left( \omega,\, x \right)\ \text{for}\ v \in \mathbb{R}^{d}_{*} \right\} \cup \{ 0 \} \\ 
& = & \left\{ w \in \mathbb{R}^{d}_{*}\ :\ \lambda \left( \sigma \omega,\, T_{\pi_{1} \left( \omega \right)} x,\, w \right) \le \lambda_{k} \left( \sigma \omega,\, T_{\pi_{1} \left( \omega \right)} x \right) \right\} \cup \{ 0 \} \\ 
& = & V^{k}_{\left( \sigma \omega,\, T_{\pi_{1} \left( \omega \right)} x \right)}, 
\end{eqnarray*} 
making $L \left( T_{\pi_{j} \left( \omega \right)} x \right) V_{\left( \sigma^{j} \omega,\, T_{\pi_{j} \left( \omega \right)} x \right)}^{k} \equiv V_{\left( \sigma^{j + 1} \omega,\, T_{\pi_{j + 1} \left( \omega \right)} x \right)}^{k}$ for $j = 0$. 

Assuming the assertion for $j$, we now prove for $j + 1$. Thus, consider 
\begin{eqnarray*} 
& & L \left( T_{\pi_{j + 1} \left( \omega \right)} x \right) V_{\left( \sigma^{j + 1} \omega,\, T_{\pi_{j + 1} \left( \omega \right)} x \right)}^{k} \\ 
& = & \bigg\{ L \left( T_{\pi_{j + 1} \left( \omega \right)} x \right) v \in \mathbb{R}^{d}_{*}\ :\ \lambda \left( \sigma^{j + 1} \omega,\, T_{\pi_{j + 1} \left( \omega \right)} x,\, v \right) \le \lambda_{k} \left( \sigma^{j + 1} \omega,\, T_{\pi_{j + 1} \left( \omega \right)} x \right)\ \text{for}\ v \in \mathbb{R}^{d}_{*} \bigg\} \\ 
& & \hspace{+13.7cm} \cup \{ 0 \} \\ 
& = & \bigg\{ w' \in \mathbb{R}^{d}_{*}\ :\ \lambda \left( \sigma^{j + 1} \omega,\, T_{\pi_{j + 1} \left( \omega \right)} x,\, w' \right) \le \lambda_{k} \left( \sigma^{j + 1} \omega,\, T_{\pi_{j + 1} \left( \omega \right)} x \right) \bigg\} \cup \{ 0 \} \\ 
& = & V_{\left( \sigma^{j + 2} \omega,\, T_{\pi_{j + 2} \left( \omega \right)} x \right)}^{k}. 
\end{eqnarray*} 

Thus, one proves the assertions in part $A$ of Theorem \ref{thm:oselomega} completely, namely 
\begin{enumerate} 
\item $r \left( \omega,\, x \right) \equiv r \left( \sigma^{n} \omega,\, T_{\pi_{n} \left( \omega \right)} x \right)$ for all $n \in \mathbb{Z}_{+}$; 
\item $\lambda_{k} \left( \omega,\, x \right) \equiv \lambda_{k} \left( \sigma^{n} \omega,\, T_{\pi_{n} \left( \omega \right)} x \right)$ for $n \in \mathbb{Z}_{+}$ and 
\item $L \left( T_{\pi_{n} \left( \omega \right)} x \right) V_{\left( \sigma^{n} \omega,\, T_{\pi_{n} \left( \omega \right)} x \right)}^{k} \equiv V_{\left( \sigma^{n + 1} \omega,\, T_{\pi_{n + 1} \left( \omega \right)} x \right)}^{k}$, for every $n \in \mathbb{Z}_{+}$. 
\end{enumerate} 
\end{proof} 

\subsection{Proof of Part $B$} 

Let $(\mathscr{S},\, d)$ be a complete and separable metric space and let $2^{\mathscr{S}}_{{\rm cpt}}$ denote the set of all non-empty compact subsets of $\mathscr{S}$. Then, 
\begin{eqnarray*} 
d_{H} \left( C_{1},\, C_{2} \right) & = & \inf \Big\{ \epsilon > 0\ :\ \left(C_{1}\right)_{\epsilon} \subseteq C_{2}\ \text{and}\ \left(C_{2}\right)_{\epsilon} \subseteq C_{1} \\ 
& & \hspace{+2cm} \text{where}\ \left(C_{i}\right)_{\epsilon}\; =\; \left\{ s' \in \mathscr{S}\; :\; d \left( s, s' \right) \le \epsilon\ \text{whenever}\ s \in C_{i} \right\} \Big\} 
\end{eqnarray*} 
provides the Hausdorff metric on the set $2^{\mathscr{S}}_{{\rm cpt}}$, that induces a topology and thereby a Borel sigma-algebra denoted by $\mathcal{B} \left( 2^{\mathscr{S}}_{{\rm cpt}} \right)$ on $2^{\mathscr{S}}_{{\rm cpt}}$. Further, let ${\rm Gr}(d)$ denote the Grassmannian of $\mathbb{R}^{d}$, {\it i.e.}, the disjoint union of all Grassmannian manifolds ${\rm Gr} \left(l, d\right)$ for $1 \le l \le d$, where ${\rm Gr} \left(l, d\right)$ is the collection of all $l$-dimensional vector subspaces of $\mathbb{R}^{d}$. We define a metric on ${\rm Gr}(d)$ given by 
\[ d_{{\rm Gr}} \left( V_{1},\, V_{2} \right)\ \ =\ \ d_{H} \left( \mathbb{S}^{d - 1} \cap V_{1},\; \mathbb{S}^{d - 1} \cap V_{2} \right),\ \ \ \text{where}\ \ \ \mathbb{S}^{d - 1} \subsetneq \mathbb{R}^{d}\ \text{denotes the unit sphere}. \] 
One can note that $\left( {\rm Gr}(d),\, d_{{\rm Gr}} \right)$ is a complete metric space. 

Suppose $\left( \mathscr{P},\, \mathcal{B}_{\mathscr{P}} \right)$ is a complete probability measure space. We now define a set-valued map, say $H : \mathscr{P} \longrightarrow 2^{\mathscr{S}}_{{\rm cpt}}$ and make use of the selection theorem due to Kuratowski and Ryll-Nardzewski \cite{krn:1965}, as one may find in \cite{cm:1977, ao:2017} that conveys the map $H$ to points in $\mathscr{S}$ denoted by $h$, as follows. 

\begin{theorem}\cite{cm:1977, ao:2017} 
\label{equivalent} 
For a complete probability measure space $\left( \mathscr{P},\, \mathcal{B}_{\mathscr{P}} \right)$ and a complete and separable metric space $\left( \mathscr{S},\, d \right)$, the following statements are equivalent. 
\begin{enumerate} 
\item The map $H : \mathscr{P} \longrightarrow 2^{\mathscr{S}}_{{\rm cpt}}$ given by $H(p) = K_{p}$ is measurable. 
\item The graph of $H$ given by $\Gamma_{H} = \left\{ (p, q) \in \mathscr{P} \times \mathscr{S} : q \in K_{p} \right\}$ belongs to the product sigma-algebra $\mathcal{B}_{\mathscr{P}} \times \mathcal{B}_{\mathscr{S}}$ of the product space $\mathscr{P} \times \mathscr{S}$. 
\item The set $\left\{ p \in \mathscr{P} : K_{p} \cap U \ne \emptyset \right\}$ belongs to $\mathcal{B}_{\mathscr{P}}$ for any open set $U \subset \mathscr{S}$. 
\end{enumerate} 
Further, assuming one of these equivalent statements also conveys a selection map denoted by $h : \mathscr{P} \longrightarrow \mathscr{S}$ that satisfies $h(p) \in K_{p}$ for every $p \in \mathscr{P}$. 
\end{theorem} 

We now write a theorem, as may be found in \cite{pw:1993} that serves as an interpretation of Theorem \ref{equivalent} for the map $H : \Sigma_{N}^{+} \times X \longrightarrow {\rm Gr}(d)$ for which the analogous set of equivalent statements are true. 

\begin{theorem} 
\label{equivalent:3}
Let $H : \Sigma_{N}^{+} \times X \longrightarrow {\rm Gr}(d)$. Then, the following statements are equivalent. 
\begin{enumerate} 
\item The map $H$ given by $H \left( \omega,\, x \right) = V_{(\omega,\, x)}$ is measurable. 
\item The set $\Gamma_{H} = \left\{ \left( \omega,\, x,\, v \right) \in \Sigma_{N}^{+} \times X \times \mathbb{R}^{d} : v \in H \left( \omega,\, x \right) \right\}$ belongs to the product sigma-algebra $\mathcal{B}_{\Sigma_{N}^{+}} \times \mathcal{B}_{X} \times \mathcal{B}_{\mathbb{R}^{d}}$ of the product space $\Sigma_{N}^{+} \times X \times \mathbb{R}^{d}$. 
\item For every $1 \le l \le d$, the set $X_{l} = \left\{ \left( \omega,\, x \right) \in \Sigma_{N}^{+} \times X : \dim \left( V_{(\omega,\, x)} \right) = l \right\}$ belongs to $\mathcal{B}_{\Sigma_{N}^{+}} \times \mathcal{B}_{X}$. Moreover, there exists measurable vector fields $w_{i} : X_{l} \longrightarrow \mathbb{R}^{d}$ for every $1 \le i \le l$ such that the collection $\left\{ w_{1} \left( \omega,\, x \right), w_{2} \left( \omega,\, x \right), \cdots, w_{l} \left( \omega,\, x \right) \right\}$ is a basis of the vector space $H \left( \omega,\, x \right)$ for every $(\omega,\, x) \in X_{l}$. 
\end{enumerate} 
\end{theorem} 

We now prove that the map $(\omega,\, x) \longmapsto r (\omega,\, x)$ is measurable and so are the maps $(\omega,\, x) \longmapsto \lambda_{k} (\omega,\, x)$ and $(\omega,\, x) \longmapsto V_{(\omega,\, x)}^{k}$ for $1 \le k \le r$, by induction on $k$. 

In order to prove that the map $(\omega,\, x) \longmapsto r (\omega,\, x)$ is measurable, we need to show that $r^{-1} \left( \left[ \left. a,\, \infty \right) \right. \right)$ is measurable for every $a > - \infty$. Note that for $a \le 1$, we have $r^{-1} \left( \left[ \left. a,\, \infty \right) \right. \right) = \Sigma_{N}^{+} \times X$, a measurable set. Moreover, for $\left( \mathfrak{B} \times \mu \right)$-almost every $\left( \omega, x \right) \in \Sigma_{N}^{+} \times X$, we know from Equation \eqref{nestedvs} that $V_{(\omega,\, x)}^{r} = \mathbb{R}^{d}$. For $1 \le l < d$, observe that the set $X_{l} = \left\{ \left( \omega,\, x \right) \in \Sigma_{N}^{+} \times X : \dim \left( V^{r}_{(\omega,\, x)} \right) = l \right\}$ is empty. Hence, we only consider the set $X_{d} = \left\{ \left( \omega,\, x \right) \in \Sigma_{N}^{+} \times X : \dim \left( V^{r}_{(\omega,\, x)} \right) = d \right\}$. Suppose $\left\{ e_{1}, e_{2}, \cdots, e_{d} \right\}$ is an arbitrary collection of basis vectors of $\mathbb{R}^{d}$, we define for every point $(\omega,\, x) \in \Sigma_{N}^{+} \times X_{d}$, a collection of vector fields given by $w_{i} : \Sigma_{N}^{+} \times X_{d} \longrightarrow \mathbb{R}^{d}$ as $w_{i} \left( (\omega, x) \right) = e_{i}$. Thus making use of the equivalence of the three statements, as written in Theorem \ref{equivalent:3}, we obtain the map $(\omega, x) \longmapsto V_{(\omega, x)}^{r}$ to be measurable. 

Given $(\omega,\, x) \in \Sigma_{N}^{+} \times X$, rearrange the above collection of basis vectors of $\mathbb{R}^{d}$, if necessary, ordered also as $\left\{ e_{1}, e_{2}, \cdots, e_{d} \right\}$, by an abuse of notations, so that the vector space $V_{(\omega,\, x)}^{k}$ is spanned by $\left\{ e_{1}, e_{2}, \cdots, e_{k} \right\}$. Let $v \in V_{(\omega,\, x)}^{k} \setminus V_{(\omega,\, x)}^{k - 1}$. Then, by our arrangement of the numbers $\lambda_{k} (\omega,\, x)$, as in Equation \eqref{lambdaarranged}, we have $\lambda_{k - 1} (\omega,\, x) < \lambda (\omega,\, x,\, v) \le \lambda_{k} (\omega,\, x)$. Further, $\lambda (\omega,\, x,\, v) < \lambda_{k} (\omega,\, x)$ increases the cardinality of the set $U_{(\omega,\, x)}$. Hence, $\lambda (\omega,\, x,\, v) \equiv \lambda_{k} (\omega,\, x)$ for every $v \in V_{(\omega,\, x)}^{k} \setminus V_{(\omega,\, x)}^{k - 1}$. This implies 
\[ \max \big\{ \lambda \left( \omega,\, x,\, e_{i} \right)\ :\ 1 \le i \le k \big\}\ \ =\ \ \lambda_{k} (\omega,\, x),\ \forall \ 1 \le k \le r. \] 
Thus, the map $(\omega,\, x) \longmapsto \lambda^{k}(\omega,\, x)$ is measurable for every $1 \le k \le r$. 

\begin{lemma}\cite{cm:1977, mv:2014} 
\label{lem:projmeas}
Consider a complete probability measure space $\mathscr{P}$ and the complete separable metric space $\mathscr{S}$, as written in Theorem \ref{equivalent}. Let $\Pi : \mathscr{P} \times \mathscr{S} \longrightarrow \mathscr{P}$ be the projection map that is defined as $\Pi (p, s) = p$. Then, $\Pi (E) \in \mathcal{B}_{\mathscr{P}}$ whenever $E \in \mathcal{B}_{\mathscr{P}} \times \mathcal{B}_{\mathscr{S}}$. 
\end{lemma} 

Observe that the set $\left\{ (\omega,\, x,\, v) \in \Sigma_{N}^{+} \times X \times \mathbb{R}^{d}\ :\ \lambda (\omega,\, x,\, v) < \lambda_{k} (\omega,\, x) \right\}$ is a measurable subset of $\Sigma_{N}^{+} \times X \times \mathbb{R}^{d}$. Using Lemma \ref{lem:projmeas}, we then obtain that the set 
\[ \left\{ (\omega,\, x) \in \Sigma_{N}^{+} \times X\ :\ \lambda (\omega,\, x,\, v) < \lambda_{r} (\omega,\, x) \right\}\ \ =\ \ \left\{ (\omega,\, x) \in \Sigma_{N}^{+} \times X\ :\ r (\omega,\, x) \ge 2 \right\} \] 
is a measurable subset of $\Sigma_{N}^{+} \times X$. Further, 
\begin{eqnarray*} 
& & \left\{ (\omega,\, x,\, v) \in \Sigma_{N}^{+} \times X \times \mathbb{R}^{d}\ :\ v \in V_{(\omega,\, x)}^{r - 1} \right\} \\ 
& = & \left\{ (\omega,\, x,\, v) \in \Sigma_{N}^{+} \times X \times \mathbb{R}^{d}\ :\ \lambda (\omega,\, x,\, v) \le \lambda_{r - 1} (\omega,\, x) \right\} \cup \left( \Sigma_{N}^{+} \times X \times \{ 0 \} \right) 
\end{eqnarray*} 
is a measurable subset of $\Sigma_{N}^{+} \times X \times \mathbb{R}^{d}$. Then, making use of the three equivalent statements in Theorem \ref{equivalent}, we obtain that the map $(\omega,\, x) \longmapsto V_{(\omega,\, x)}^{r - 1}$ is measurable. 

Proceeding by induction on $r$ where $r$ decreases on the set of positive integers, the proof of part $B$ of Theorem \ref{thm:oselomega} is complete. 

\section{Analogue of Kingman ergodic theorem} 
\label{kingmananalogue} 

Before we embark on the proof of part $C$ of Theorem \ref{thm:oselomega}, we write a few necessary lemmas and technical constructions in this section. 

Suppose $\mathcal{C} (\mathscr{K})$ is the Banach space of real-valued continuous functions defined on a compact metric space $\mathscr{K}$ that comes equipped with the supremum norm denoted by $\| \cdot \|_{\infty}$. Let $\mathcal{M} \left( \Sigma_{N}^{+} \times X \times \mathscr{K} \right)$ denote the space of all real-valued measurable functions such that for $\left( \mathfrak{B} \times \mu \right)$-almost every $\left( \omega, x \right) \in \Sigma_{N}^{+} \times X$ we have $\phi \left( \omega, x, \cdot \right) \in \mathcal{C} (\mathscr{K})$ and the map $(\omega,\, x) \longmapsto \left\| \phi \left( \omega, x, \cdot \right) \right\|_{\infty}$ is integrable with respect to the product measure $\left( \mathfrak{B} \times \mu \right)$. Define a norm $\| \cdot \|_{\mathcal{M}}$ on $\mathcal{M} \left( \Sigma_{N}^{+} \times X \times \mathscr{K} \right)$ given by $\displaystyle{ \left\| \phi \right\|_{\mathcal{M}} = \int \left\| \phi \left( \omega, x, \cdot \right) \right\|_{\infty} \mathrm{d} \left( \mathfrak{B} \times \mu \right)}$. 

In order to prove that $\mathcal{M} \left( \Sigma_{N}^{+} \times X \times \mathscr{K} \right)$ is a complete space with respect to $\left\| \cdot \right\|_{\mathcal{M}}$, we consider a Cauchy sequence of functions in $\mathcal{M} \left( \Sigma_{N}^{+} \times X \times \mathscr{K} \right)$, namely $\left\{ \phi_{n} \right\}_{n\, \in\, \mathbb{Z}_{+}}$, that contains a subsequence, say $\left\{ \phi_{n_{k}} \right\}_{k\, \in\, \mathbb{Z}_{+}}$ which satisfies $\displaystyle{\left\| \phi_{n_{k}} - \phi_{n_{k + 1}} \right\|_{\mathcal{M}} < \frac{1}{2^{k}}}$ for every $k \in \mathbb{Z}_{+}$. Then, 
\begin{eqnarray*} 
1\ \ =\ \ \sum_{k\, \ge\, 1} \frac{1}{2^{k}} & > & \sum_{k\, \ge\, 1} \left\| \phi_{n_{k}} - \phi_{n_{k + 1}} \right\|_{\mathcal{M}} \\ 
& = & \sum_{k\, \ge\, 1} \int \left\| \phi_{n_{k}} \left( \omega, x, \cdot \right) - \phi_{n_{k + 1}} \left( \omega, x, \cdot \right) \right\|_{\infty} \mathrm{d} \left( \mathfrak{B} \times \mu \right). 
\end{eqnarray*} 
Then, an application of a corollary of the Lebesgue dominated convergence theorem, as one may find in \cite{rudin:1987} (Theorem 1.38), says that $\displaystyle{\sum_{k\, \ge\, 1} \left\| \phi_{n_{k}} \left( \omega, x, \cdot \right) - \phi_{n_{k + 1}} \left( \omega, x, \cdot \right) \right\|_{\infty}}$ converges for $\left( \mathfrak{B} \times \mu \right)$-almost every $(\omega, x)$. Hence, the sequence $\left\{ \phi_{n_{k}} \left( \omega, x, \cdot \right) \right\}_{k\, \in\, \mathbb{Z}_{+}}$ is Cauchy in $\mathcal{C} (\mathscr{K})$, and hence converges to some limit, say $\phi_{0} \left( \omega, x, \cdot \right)$ in $\mathcal{C} (\mathscr{K})$. Note that $\phi_{0} \in \mathcal{M} \left( \Sigma_{N}^{+} \times X \times \mathscr{K} \right)$, provided the map $(\omega,\, x) \longmapsto \left\| \phi_{0} \left( \omega, x, \cdot \right) \right\|_{\infty}$ is integrable with respect to the product measure $\left( \mathfrak{B} \times \mu \right)$. We now prove the same. 

Consider the sequence $\left\{ \phi_{n_{k}} \left( \omega, x, y \right) \right\}_{k\, \in \mathbb{Z}_{+}}$ that converges to $\phi_{0} \left( \omega, x, y \right)$ uniformly in $\mathcal{C} (\mathscr{K})$ for $\left( \mathfrak{B} \times \mu \right)$-almost every $\left( \omega, x \right)$. Then, $\displaystyle{\lim_{k\, \to\, \infty} \left\| \phi_{n_{k}} \left( \omega, x, \cdot \right) \right\|_{\infty} = \left\| \phi_{0} \left( \omega, x, \cdot \right) \right\|_{\infty}}$. Then, by Fatou's lemma, as one may find in \cite{rudin:1987}, we have 
\[ \int \left\| \phi_{0} \left( \omega, x, \cdot \right) \right\|_{\infty} \mathrm{d} \left( \mathfrak{B} \times \mu \right)\ \ \le\ \ \liminf_{k\, \to\, \infty} \int \left\| \phi_{n_{k}} \left( \omega, x, \cdot \right) \right\|_{\infty} \mathrm{d} \left( \mathfrak{B} \times \mu \right)\ \ \le\ \ \liminf_{k\, \to\, \infty} \left\| \phi_{n_{k}} \right\|_{\mathcal{M}}\ \ <\ \ \infty, \]
and thus the map $(\omega,\, x) \longmapsto \left\| \phi_{0} \left( \omega, x, \cdot \right) \right\|_{\infty}$ is integrable with respect to the product measure $\left( \mathfrak{B} \times \mu \right)$. 

We still need to prove that the sequence $\displaystyle{\left\{ \phi_{n_{k}} \right\}}$ converges to $\phi_{0}$ in the $\mathcal{M}$-norm to conclude that $\mathcal{M} \left( \Sigma_{N}^{+} \times X \times \mathscr{K} \right)$ is complete with respect to the norm $\left\| \cdot \right\|_{\mathcal{M}}$. Towards that end, we note that 
\[ \left\| \phi_{n_{k}} \left( \omega, x, \cdot \right) - \phi_{n_{k + 1}} \left( \omega, x, \cdot \right) \right\|_{\infty}\ \ \le\ \ \sum_{l\, =\, n_{k}}^{n_{k + 1} - 1} \left\| \phi_{l + 1} \left( \omega, x, \cdot \right) - \phi_{l} \left( \omega, x, \cdot \right) \right\|_{\infty}. \] 
Fix $n_{k}$ and let $n_{k + 1} \to \infty$. Then, since $\phi_{n_{k}} \to \phi_{0}$ in the supremum norm, 
\begin{eqnarray*} 
\left\| \phi_{n_{k}} \left( \omega, x, \cdot \right) - \phi_{0} \left( \omega, x, \cdot \right) \right\|_{\infty} & \le & \sum_{l\, =\, n_{k}}^{\infty} \left\| \phi_{l + 1} \left( \omega, x, \cdot \right) - \phi_{l} \left( \omega, x, \cdot \right) \right\|_{\infty} \\ 
& \le & \sum_{l\, =\, n_{1}}^{\infty} \left\| \phi_{l + 1} \left( \omega, x, \cdot \right) - \phi_{l} \left( \omega, x, \cdot \right) \right\|_{\infty}. \nonumber 
\end{eqnarray*} 

Since, by assumption, $\displaystyle{\left\| \phi_{n_{k}} - \phi_{n_{k + 1}} \right\|_{\mathcal{M}} < \frac{1}{2^{k}}}$ for every $k \in \mathbb{Z}_{+}$, we have 
\[ \sum_{l\, =\, n_{1}}^{\infty} \left\| \phi_{l + 1} - \phi_{l} \right\|_{\mathcal{M}}\ \ =\ \ \sum_{l\, =\, n_{1}}^{\infty} \int \left\| \phi_{l + 1} - \phi_{l} \right\|_{\infty} \mathrm{d} \left( \mathfrak{B} \times \mu \right)\ \ <\ \ \infty. \] 
Thus, an application of Theorem 1.38 from \cite{rudin:1987} yields 
\[ \int \sum_{l\, =\, n_{1}}^{\infty} \left\| \phi_{l + 1} - \phi_{l} \right\|_{\infty} \mathrm{d} \left( \mathfrak{B} \times \mu \right)\ \ =\ \ \sum_{l\, =\, n_{1}}^{\infty} \int \left\| \phi_{l + 1} - \phi_{l} \right\|_{\infty} \mathrm{d} \left( \mathfrak{B} \times \mu \right)\ \ <\ \ \infty. \] 
Since 
\[ \int \sum_{l\, =\, n_{k}}^{\infty} \left\| \phi_{l + 1} - \phi_{l} \right\|_{\infty} \mathrm{d} \left( \mathfrak{B} \times \mu \right) \le \int \sum_{l\, =\, n_{1}}^{\infty} \left\| \phi_{l + 1} - \phi_{l} \right\|_{\infty} \mathrm{d} \left( \mathfrak{B} \times \mu \right), \] 
an application of the dominated convergence theorem, as one may find in \cite{rudin:1987} gives 
\begin{eqnarray*} 
\lim_{k\, \to\, \infty} \left\| \phi_{n_{k}} - \phi_{0} \right\|_{\mathcal{M}} & = & \lim_{k\, \to\, \infty} \int \left\| \phi_{n_{k}} - \phi_{0} \right\|_{\infty} \mathrm{d} \left( \mathfrak{B} \times \mu \right) \\ 
& \le & \lim_{k\, \to\, \infty} \int \sum_{l\, =\, n_{k}}^{\infty} \left\| \phi_{l + 1} - \phi_{l} \right\|_{\infty} \mathrm{d} \left( \mathfrak{B} \times \mu \right) \\ 
& = & \int \lim_{k\, \to\, \infty} \sum_{l\, =\, n_{k}}^{\infty} \left\| \phi_{l + 1} - \phi_{l} \right\|_{\infty} \mathrm{d} \left( \mathfrak{B} \times \mu \right) \\ 
& = & \int \lim_{k\, \to\, \infty} \left( \sum_{l\, =\, n_{1}}^{\infty} \left\| \phi_{l + 1} - \phi_{l} \right\|_{\infty} - \sum_{l\, =\, n_{1}}^{n_{k} - 1} \left\| \phi_{l + 1} - \phi_{l} \right\|_{\infty} \right) \mathrm{d} \left( \mathfrak{B} \times \mu \right) \\ 
& = & 0, 
\end{eqnarray*} 
thereby proving that $\mathcal{M} \left( \Sigma_{N}^{+} \times X \times \mathscr{K} \right)$ is complete with respect to the norm $\left\| \cdot \right\|_{\mathcal{M}}$. 

Let $\mathbb{M} \left( \mathfrak{B} \times \mu \right)$ be the set of all probability measures $\nu$ supported on $\Sigma_{N}^{+} \times X \times \mathscr{K}$ such that ${\rm Proj}_{*} \nu = \nu \circ {\rm Proj}^{-1} = \mathfrak{B} \times \mu$, where ${\rm Proj} : \Sigma_{N}^{+} \times X \times \mathscr{K} \longrightarrow \Sigma_{N}^{+} \times X$ is the projection map. Then, by an application of the Banach-Alaoglu theorem we get $\mathbb{M} \left( \mathfrak{B} \times \mu \right)$ as a compact metrisable subspace of the space of all probability measures supported on $\Sigma_{N}^{+} \times X \times \mathscr{K}$. 

We now state and prove a proposition that helps us in the sequel. 

\begin{proposition} 
\label{nustar}
Consider the dynamical system $\mathcal{G} : \Sigma_{N}^{+} \times X \times \mathscr{K} \circlearrowright$ given by $\mathcal{G} \left( \omega,\, x,\, v \right) = \left( \sigma \omega, T_{\pi_{1} \left( \omega \right)} x, G_{(\omega,\, x)} v \right)$, where $G_{(\omega,\, x)} : \mathscr{K} \longrightarrow \mathscr{K}$ is a continuous map for every $\left( \mathfrak{B} \times \mu \right)$-typical point $\left( \omega,\, x \right) \in \Sigma_{N}^{+} \times X$. Fix some $\left( \mathfrak{B} \times \mu \right)$-typical point $\left( \omega,\, x \right) \in \Sigma_{N}^{+} \times X$. Then, given any $\phi \in \mathcal{M} \left( \Sigma_{N}^{+} \times X \times \mathscr{K} \right)$, the real sequences 
\[ \left\{ \frac{1}{n} \inf_{v\, \in\, \mathscr{K}} \sum_{j\, =\, 0}^{n - 1} \phi \left( \mathcal{G}^{j} \left( \omega,\, x,\, v \right) \right) \right\}_{n\, \in\, \mathbb{Z}_{+}}\ \ \ \text{and}\ \ \ \left\{ \frac{1}{n} \sup_{v\, \in\, \mathscr{K}} \sum_{j\, =\, 0}^{n - 1} \phi \left( \mathcal{G}^{j} \left( \omega,\, x,\, v \right) \right) \right\}_{n\, \in\, \mathbb{Z}_{+}} \] 
converge to the limits $\phi_{*} \equiv \left( \phi_{*} \right)_{\left( \omega,\, x \right)}$ and $\phi^{*} \equiv \left( \phi^{*} \right)_{\left( \omega,\, x \right)}$. Further, there exist $\mathcal{G}$-invariant measures $\nu_{*}$ and $\nu^{*}$ such that 
\begin{eqnarray} 
\label{nuustar} 
\int \phi \mathrm{d} \nu_{*} & = & \int \phi_{*} \mathrm{d}\left( \mathfrak{B} \times \mu \right)\ \ =\ \ \int \lim_{n\, \to\, \infty} \frac{1}{n} \inf_{v\, \in\, \mathscr{K}} \sum_{j\, =\, 0}^{n - 1} \phi \left( \mathcal{G}^{j} \left( \omega,\, x,\, v \right) \right) \mathrm{d}\left( \mathfrak{B} \times \mu \right) \\ 
\label{nulstar} 
\int \phi \mathrm{d} \nu^{*} & = & \int \phi^{*} \mathrm{d}\left( \mathfrak{B} \times \mu \right)\ \ =\ \ \int \lim_{n\, \to\, \infty} \frac{1}{n} \sup_{v\, \in\, \mathscr{K}} \sum_{j\, =\, 0}^{n - 1} \phi \left( \mathcal{G}^{j} \left( \omega,\, x,\, v \right) \right) \mathrm{d}\left( \mathfrak{B} \times \mu \right). 
\end{eqnarray} 
\end{proposition} 

\begin{proof} 
Observe that the definition of $\mathcal{G}$ entails that its $n$-th iterate is given by 
\begin{eqnarray*} 
\mathcal{G}^{n} \left( \omega,\, x,\, v \right) & = & \left( \sigma^{n} \omega,\, T_{\pi_{n} \left( \omega \right)} x,\, \left( G_{\left( \sigma^{n - 1} \omega,\, T_{\pi_{n - 1} \left( \omega \right)} x \right)} \circ \cdots \circ G_{\left( \omega,\, x \right)} \right) v \right) \\ 
& = & \left( \sigma^{n} \omega,\, T_{\pi_{n} \left( \omega \right)} x,\, G_{\left( \omega,\, x \right)}^{n} v \right), 
\end{eqnarray*} 
where we define $G_{\left( \omega,\, x \right)}^{n} = G_{\left( \sigma^{n - 1} \omega,\, T_{\pi_{n - 1} \left( \omega \right)} x \right)} \circ \cdots \circ G_{\left( \omega,\, x \right)}$. Now, let 
\begin{equation} 
\label{eqn:phinstar} 
\left( \phi_{n} \right)_{*} \left( \omega,\, x \right)\ \ =\ \ \inf_{v\, \in\, \mathscr{K}} \sum_{j\, =\, 0}^{n - 1} \phi \left( \mathcal{G}^{j} \left( \omega,\, x,\, v \right) \right)\ \ \ \text{and}\ \ \ \left( \phi_{n} \right)^{*} \left( \omega,\, x \right)\ \ =\ \ \sup_{v\, \in\, \mathscr{K}} \sum_{j\, =\, 0}^{n - 1} \phi \left( \mathcal{G}^{j} \left( \omega,\, x,\, v \right) \right). 
\end{equation} 
We work with the sequence of functions $\left\{ \left( \phi_{n} \right)^{*} \right\}_{n\, \in\, \mathbb{Z}_{+}}$ and the arguments presented carry over to the sequence of functions $\left\{ \left( \phi_{n} \right)_{*} \right\}_{n\, \in\, \mathbb{Z}_{+}}$, as well. Note that $\left( \phi_{n} \right)^{*} \left( \omega,\, x \right)$ can be alternatively defined as 
\[ \left( \phi_{n} \right)^{*} \left( \omega,\, x \right)\ \ =\ \ \sup_{v\, \in\, \Delta} \sum_{j\, =\, 0}^{n - 1} \phi \left( \mathcal{G}^{j} \left( \omega,\, x,\, v \right) \right), \] 
where $\Delta$ is a countable dense subset of $\mathscr{K}$. Further, $\left\{ \left( \phi_{n} \right)^{*} \right\}_{n\, \in\, \mathbb{Z}_{+}}$ is a sequence of measurable functions defined on $\Sigma_{N}^{+} \times X$, that satisfies 
\[ \left( \phi_{n + p} \right)^{*} \left( \omega,\, x \right) \le \left( \phi_{n} \right)^{*} \left( \omega,\, x \right) + \left( \phi_{p} \right)^{*} \left( \sigma^{n} \omega,\, T_{\pi_{n} \left( \omega \right)} x \right). \] 

We now state a theorem due to the authors, as can be found in \cite{ts:pp} and make use of the same to prove the convergence of the sequences, mentioned in Proposition \ref{nustar}. Note that this theorem is the analogue of the Kingman ergodic theorem, in this setting.

\begin{theorem}\cite{ts:pp} 
\label{thm:kingmanomega}
Let $\left\{ T_{1}, T_{2}, \cdots, T_{N} \right\},\ 1 \le N < \infty$ be finitely many measure preserving maps defined on a compact probability measure space $(X, \mu)$. Let $\phi \in \mathcal{M} \left( \Sigma_{N}^{+} \times X \times \mathscr{K} \right)$. Then, the sequence $\displaystyle{\left\{\dfrac{\left( \phi_{n} \right)^{*}}{n}\right\}_{n\, \ge\, 1}}$ converges for $\left( \mathfrak{B} \times \mu \right)$-almost every $(\omega, x) \in \Sigma_{N}^{+} \times X$ to some measurable function, say $\phi^{*} : \Sigma_{N}^{+} \times X \longrightarrow [-\infty, \infty)$ that satisfies $\phi^{*} \left( \omega,\, x \right) = \phi^{*} \left( \sigma^{j} \omega, T_{\pi_{j} \left( \omega \right)} x \right)$ for any $j \in \mathbb{Z}_{+}$. Moreover, the positive part of the limit function $\phi^{*}$ is integrable and 
\[ \int \phi^{*} \mathrm{d}\left( \mathfrak{B} \times \mu \right)\ \ =\ \ \lim_{n\, \to\, \infty} \frac{1}{n} \int \left( \phi_{n} \right)^{*} \mathrm{d}\left( \mathfrak{B} \times \mu \right)\ \ =\ \ \inf_{n\, \ge\, 1} \frac{1}{n} \int \left( \phi_{n} \right)^{*} \mathrm{d}\left( \mathfrak{B} \times \mu \right)\ \ \in\ \ [-\infty, \infty). \] 
\end{theorem} 

Applying Theorem \ref{thm:kingmanomega} to $\left( \phi_{n} \right)^{*}$ as defined in Equation \eqref{eqn:phinstar}, we obtain the limit function $\phi^{*}$ for $\left( \mathfrak{B} \times \mu \right)$-almost every $(\omega, x) \in \Sigma_{N}^{+} \times X$. Analogously, we also obtain the limit function $\phi_{*}$, working with $\left( \phi_{n} \right)_{*}$ for $\left( \mathfrak{B} \times \mu \right)$-almost every $(\omega, x) \in \Sigma_{N}^{+} \times X$, thus proving the existence of the limits of the appropriate real sequences, as required in Proposition \ref{nustar}. 

In order to construct the $\mathcal{G}$-invariant measures $\nu_{*}$ and $\nu^{*}$ satisfying the Equations \eqref{nulstar} and \eqref{nuustar}, we consider the measurable set 
\begin{equation} 
\label{Deltan} 
\Delta_{n}\ =\ \left\{ \left( \omega,\, x,\, v \right) \in \Sigma_{N}^{+} \times X \times \mathscr{K} : \sum_{j\, =\, 0}^{n - 1} \phi \left( \mathcal{G}^{j} \left( \omega,\, x,\, v \right) \right) = \left( \phi_{n} \right)^{*} \left( \omega,\, x \right) \right\}. 
\end{equation} 
For some fixed $\left( \omega,\, x \right) \in \Sigma_{N}^{+} \times X$, we define the set $\Delta_{n} \left( \omega,\, x \right) = \left\{ v \in \mathscr{K} : \left( \omega,\, x,\, v \right) \in \Delta_{n} \right\}$. Since $\mathscr{K}$ is compact, we have that $\Delta_{n} \left( \omega,\, x \right)$ to be a nonempty compact subset of $\mathscr{K}$, thus making the map $\left( \omega,\, x \right) \longmapsto \Delta_{n} \left( \omega,\, x \right)$ as a measurable map from $\Sigma_{N}^{+} \times X \longrightarrow 2^{\mathscr{K}}_{{\rm cpt}}$. Thus, by Theorem \ref{equivalent}, there exists a measurable selection, say $w_{n} : \Sigma_{N}^{+} \times X \longrightarrow \mathscr{K}$ such that $w_{n} \left( \omega,\, x \right) \in \Delta_{n} \left( \omega,\, x \right) \in 2^{\mathscr{K}}_{{\rm cpt}}$. 

\section{Markov kernels} 
\label{markov} 

We begin this section by recalling the definition of a Markov kernel and a result pertaining to the same, as can be found in \cite{ka:2014}; expressing it in the case of our setting. 

\begin{definition} 
Consider the measurable spaces $\left( \Sigma_{N}^{+},\, \mathcal{B}_{\Sigma_{N}^{+}} \right),\ \left( X,\, \mathcal{B}_{X} \right)$ and $\left( \mathscr{K},\, \mathcal{B}_{\mathscr{K}} \right)$. A map $\kappa : \Sigma_{N}^{+} \times X \times \mathcal{B}_{\mathscr{K}} \longrightarrow [0, 1]$ is called a \emph{Markov kernel} if 
\begin{enumerate} 
\item the function $\left( \omega,\, x \right) \longmapsto \kappa \left( \omega,\, x,\, B_{\mathscr{K}} \right)$ is measurable for every $B_{\mathscr{K}} \in \mathcal{B}_{\mathscr{K}}$; 
\item the function $B_{\mathscr{K}} \longmapsto \kappa \left( \omega,\, x,\, B_{\mathscr{K}} \right)$ is a probability measure supported on $\mathscr{K}$ for any $\left( \omega,\, x \right)$. 
\end{enumerate} 
\end{definition} 

\begin{theorem}\cite{ka:2014} 
\label{markovthm}
Let $\kappa : \Sigma_{N}^{+} \times X \times \mathcal{B}_{\mathscr{K}} \longrightarrow [0, 1]$ be a Markov kernel. Then, there exists a unique measure given by $\gamma = \mathfrak{B} \times \mu \times \kappa$ supported on $\Sigma_{N}^{+} \times X \times \mathscr{K}$ such that 
\[ \gamma (B) = \left( \mathfrak{B} \times \mu \times \kappa \right) \left( B_{\Sigma_{N}^{+}} \times B_{X} \times B_{\mathscr{K}} \right)\ \ =\ \ \int\limits_{B_{\Sigma_{N}^{+}} \times B_{X}} \kappa \left( \omega,\, x,\, B_{\mathscr{K}} \right) \mathrm{d}\left( \mathfrak{B} \times \mu \right), \] 
for every measurable set $B = B_{\Sigma_{N}^{+}} \times B_{X} \times B_{\mathscr{K}}$ where $B_{\Sigma_{N}^{+}} \in \mathcal{B}_{\Sigma_{N}^{+}},\; B_{X} \in \mathcal{B}_{X}$ and $B_{\mathscr{K}} \in \mathcal{B}_{\mathscr{K}}$. 
\end{theorem} 

Define a sequence of Markov kernels $\left\{ \kappa_{n} \right\}$ as follows: For any set $B_{\mathscr{K}} \in \mathcal{B}_{\mathscr{K}}$, let 
\[ \kappa_{n} \left( \omega,\, x,\, B_{\mathscr{K}} \right)\ \ =\ \ \begin{cases} 1 & \text{if}\ v_{n} \left( \omega,\, x \right) \in B_{\mathscr{K}}; \\ 0 & \text{otherwise}. \end{cases} \] 
Then, by Theorem \ref{markovthm}, we obtain a sequence of measures, namely $\left\{ \gamma_{n} \right\}$ such that 
\[ \gamma_{n} (B)\ \ =\ \ \int\limits_{B_{\Sigma_{N}^{+}} \times B_{X}} \kappa_{n} \left( \omega,\, x,\, B_{\mathscr{K}} \right) \mathrm{d}\left( \mathfrak{B} \times \mu \right),\ \ \text{where}\ B = B_{\Sigma_{N}^{+}} \times B_{X} \times B_{\mathscr{K}}. \] 
Note that $\gamma_{n} \in \mathbb{M} \left( \mathfrak{B} \times \mu \right)$ since ${\rm Proj}_{*} \gamma_{n} = \mathfrak{B} \times \mu$. Also, the sequence of measures $\left\{ \nu_{n} \right\}$ supported on $\Sigma_{N}^{+} \times X \times \mathscr{K}$ and defined as \[ \nu_{n} (B)\ =\ \frac{1}{n} \sum_{j\, =\, 0}^{n - 1} \left( \mathcal{G}^{j}_{*} \gamma_{n} \right) (B)\ =\ \frac{1}{n} \sum_{j\, =\, 0}^{n - 1} \gamma_{n} \left( \mathcal{G}^{-j} (B) \right), \] 
satisfies ${\rm Proj}_{*} \nu_{n} = \mathfrak{B} \times \mu$ and thus, $\nu_{n} \in \mathbb{M} \left( \mathfrak{B} \times \mu \right)$. 

Since $\mathbb{M} \left( \mathfrak{B} \times \mu \right)$ is a compact space, there exists a subsequence $\left\{ \nu_{n_{k}} \right\}$ in $\mathbb{M} \left( \mathfrak{B} \times \mu \right)$ that converges to some measure in $\mathbb{M} \left( \mathfrak{B} \times \mu \right)$. We prove that this limiting measure is equal to $\nu^{*}$, as required in Proposition \ref{nustar}. We first prove that $\nu^{*}$ is $\mathcal{G}$-invariant, by taking an arbitrary function $\phi \in \mathcal{M} \left( \Sigma_{N}^{+} \times X \times \mathscr{K} \right)$. Then, 
\begin{eqnarray*} 
& & \left| \int \left( \phi \circ \mathcal{G} \right) \mathrm{d}\nu_{n_{k}} - \int \phi \mathrm{d}\nu_{n_{k}} \right| \\ 
& = & \frac{1}{n_{k}} \left| \int \left( \phi \circ \mathcal{G}^{n_{k}} \right) \mathrm{d}\gamma_{n_{k}} - \int \phi \mathrm{d}\gamma_{n_{k}} \right| \\ 
& = & \frac{1}{n_{k}} \Bigg| \int \phi \left( \sigma^{n_{k}} \omega,\, T_{\pi_{n_{k}} \left( \omega \right)} x,\, G^{n_{k}}_{\left( \omega,\, x \right)} v \right) \mathrm{d}\gamma_{n_{k}} - \int \phi \left( \omega,\,x,\, v \right) \mathrm{d}\gamma_{n_{k}} \Bigg| \\ 
& = & \frac{1}{n_{k}} \Bigg| \int \phi \left( \sigma^{n_{k}} \omega,\, T_{\pi_{n_{k}} \left( \omega \right)} x,\, w_{n_{k}} \left( \sigma_{n_{k}} \omega,\, T_{\pi_{n_{k}} \left( \omega \right)} x \right) \right) \mathrm{d}\left( \mathfrak{B} \times \mu \right) \\ 
& & \hspace{5cm} - \int \phi \left( \omega,\,x,\, w_{n_{k}} \left( \omega,\, x \right) \right) \mathrm{d}\left( \mathfrak{B} \times \mu \right) \Bigg| \\ 
& \le & \frac{1}{n_{k}} \int \sup_{v\, \in\, \mathscr{K}} \left| \phi \left( \sigma^{n_{k}} \omega,\, T_{\pi_{n_{k}} \left( \omega \right)} x,\, v \right) \right| \mathrm{d}\left( \mathfrak{B} \times \mu \right) + \frac{1}{n_{k}} \int \sup_{v\, \in\, \mathscr{K}} \left| \phi \left( \omega,\, x,\, v \right) \right| \mathrm{d}\left( \mathfrak{B} \times \mu \right) \\ 
& \le & \frac{2}{n_{k}} \left\| \phi \right\|_{\mathcal{M}}. 
\end{eqnarray*} 
As $k \to \infty$, we obtain the limiting measure $\nu_{*}$ to be $\mathcal{G}$-invariant. Further, 
\begin{eqnarray*} 
\int \phi \mathrm{d}\nu^{*}\ \ =\ \ \lim_{k\, \to\, \infty} \int \phi \mathrm{d}\nu_{n_{k}} & = & \lim_{k\, \to\, \infty} \int \phi \mathrm{d}\left( \frac{1}{n_{k}} \sum_{j\, =\, 0}^{n_{k} - 1} \mathcal{G}^{j}_{*} \gamma_{n_{k}} \right) \\ 
& = & \lim_{k\, \to\, \infty} \frac{1}{n_{k}} \int \sum_{j\, =\, 0}^{n_{k} - 1} \left( \phi \circ \mathcal{G}^{j} \right) \mathrm{d}\left( \mathfrak{B} \times \mu \right)\ \ =\ \ \int \phi^{*} \mathrm{d}\left( \mathfrak{B} \times \mu \right), 
\end{eqnarray*} 
thus proving Equation \eqref{nulstar}. Similarly, one can prove the existence of a limiting $\mathcal{G}$-invariant measure $\nu_{*}$ and the corresponding Equation \eqref{nuustar}, by considering the measurable set $\Theta_{n}$, as defined below, in place of $\Delta_{n}$, as defined in Equation \eqref{Deltan}. 
\[ \Theta_{n}\ =\ \left\{ \left( \omega,\, x,\, v \right) \in \Sigma_{N}^{+} \times X \times \mathscr{K} : \sum_{j\, =\, 0}^{n - 1} \phi \left( \mathcal{G}^{j} \left( \omega,\, x,\, v \right) \right) = \left( \phi_{n} \right)_{*} \left( \omega,\, x \right) \right\}, \] 
thereby, completing the proof of Proposition \ref{nustar}. 
\end{proof}

\begin{corollary} 
\label{cor:nustar}
Assume the hypothesis in Proposition \ref{nustar}. Then, for $\left( \mathfrak{B} \times \mu \right)$-almost every $\left( \omega,\, x \right) \in \Sigma_{N}^{+} \times X$, there exist functions $w_{*},\ w^{*} : \Sigma_{N}^{+} \times X \longrightarrow \mathscr{K}$ such that 
\begin{eqnarray*} 
\phi_{*} \left( \omega,\, x \right) & = & \lim_{n\, \to\, \infty} \frac{1}{n} \sum_{j\, =\, 0}^{n - 1} \left( \phi \circ \mathcal{G}^{j} \right) \left( \omega,\, x,\, w_{*} \left( \omega,\, x \right) \right); \\ 
\phi^{*} \left( \omega,\, x \right) & = & \lim_{n\, \to\, \infty} \frac{1}{n} \sum_{j\, =\, 0}^{n - 1} \left( \phi \circ \mathcal{G}^{j} \right) \left( \omega,\, x,\, w^{*} \left( \omega,\, x \right) \right), 
\end{eqnarray*} 
where $\phi_{*}$ and $\phi^{*}$ are functions as defined in Equations \eqref{nulstar} and \eqref{nuustar}, respectively. 
\end{corollary} 

\begin{proof} 
For the dynamical system on $\Sigma_{N}^{+} \times X$ given by $\left( \omega,\, x \right) \longmapsto \left( \sigma \omega,\, T_{\pi_{1} \left( \omega \right)} x \right)$ that preserves the measure $\mathfrak{B} \times \mu$ and an observable function $\psi \in \mathscr{L}^{1} \left( \mathfrak{B} \times \mu \right)$, one can obtain from the Birkhoff's ergodic theorem, as stated in Theorem \ref{thm:bet} that 
\[ \widetilde{\psi} \left( \omega,\, x \right)\ \ =\ \ \lim_{n\, \to\, \infty} \frac{1}{n} \sum_{j\, =\, 0}^{n - 1} \psi \left( \sigma^{j} \omega,\, T_{\pi_{j} \left( \omega \right)} x \right),\ \ \text{that satisfies}\ \ \int \widetilde{\psi} \mathrm{d}\left( \mathfrak{B} \times \mu \right) = \int \psi \mathrm{d}\left( \mathfrak{B} \times \mu \right). \] 
And for the dynamical system $\mathcal{G}$ defined on $\Sigma_{N}^{+} \times X \times \mathscr{K}$ that preserves the measure $\nu_{*}$, as obtained in Proposition \ref{nustar} and an integrable function $\phi \in \mathscr{L}^{1} \left( \nu_{*} \right)$, we have 
\[ \widetilde{\phi} \left( \omega,\, x,\, v \right)\ \ =\ \ \lim_{n\, \to\, \infty} \frac{1}{n} \sum_{j\, =\, 0}^{n - 1} \left( \phi \circ \mathcal{G}^{j} \right) \left( \omega,\, x,\, v \right),\ \ \text{that satisfies}\ \ \int \widetilde{\phi} \mathrm{d}\nu_{*} = \int \phi \mathrm{d}\nu_{*}. \] 
Owing to Proposition \ref{nustar}, we have $\displaystyle{\int \phi \mathrm{d}\nu_{*} = \int \phi_{*} \mathrm{d}\left( \mathfrak{B} \times \mu \right)}$. Thus, the measurable set $E = \left\{ \left( \omega,\, x,\, v \right) \in \Sigma_{N}^{+} \times X \times \mathscr{K} : \widetilde{\phi} \left( \omega,\, x,\, v \right) = \phi_{*} \left( \omega,\, x \right) \right\}$ is a set for which $\nu_{*} (E) = 1$. Thus, for $\left( \mathfrak{B} \times \mu \right)$-almost every $\left( \omega,\, x \right) \in \Sigma_{N}^{+} \times X$, there exists at least one point $v \in \mathscr{K}$ for which 
\[ \phi_{*} \left( \omega,\, x \right)\ =\ \lim_{n\, \to\, \infty} \frac{1}{n} \sum_{j\, =\, 0}^{n - 1} \left( \phi \circ \mathcal{G}^{j} \right) \left( \omega,\, x,\, v \right). \] 
We call that $v$ as $w_{*} \left( \omega,\, x \right)$. Similarly, considering the $\mathcal{G}$-invariant measure $\nu^{*}$, one arrives at the function $w^{*} \left( \omega,\, x \right)$. 
\end{proof} 

\section{Invariant measurable sub-bundle} 
\label{invariantmsb} 

In this section, we focus on stating and proving a proposition that enables us in the determination of the operator norm of $L_{\omega}^{n} (x)$, for $\left( \mathfrak{B} \times \mu \right)$-almost every $\left( \omega,\, x \right) \in \Sigma_{N}^{+} \times X$, which in turn helps us in the proof of part $C$ of Theorem \ref{thm:oselomega}. To do that, we start with the following definition. 

\begin{definition} 
Consider the map $H : \Sigma_{N}^{+} \times X \longrightarrow \mathbb{R}^{d}$, as advocated in Theorem \ref{equivalent:3}, that determines a vector subspace $V_{\left( \omega,\, x \right)} \subseteq \mathbb{R}^{d}$. $H$ is said to be a \emph{$\sigma$-invariant measurable sub-bundle} if $L(x) V_{\left( \omega,\, x \right)} = V_{\left( \sigma \omega,\, T_{\pi_{1} \left(\omega \right)} x \right)}$, where the map $L$ is as in Theorem \ref{thm:oselomega}.
\end{definition} 

\begin{proposition} 
\label{prop:usein10}
Consider the dynamical system $S : \Sigma_{N}^{+} \times X \times \mathbb{R}^{d} \circlearrowright$, as given in Theorem \ref{thm:oselomega}. Consider a $\sigma$-invariant measurable sub-bundle $H : \Sigma_{N}^{+} \times X \longrightarrow \mathbb{R}^{d}$. Then, for $\left( \mathfrak{B} \times \mu \right)$-almost every $\left( \omega,\, x \right) \in \Sigma_{N}^{+} \times X$, we have 
\begin{eqnarray*} 
\lim_{n\, \to\, \infty} \frac{1}{n} \log \left\| \left. L_{\omega}^{n} (x) \right|_{V_{\left( \omega,\, x \right)}} \right\|_{{\rm op}} & = & \max \left\{ \lambda \left( \omega,\, x,\, v \right)\ :\ v \in V_{\left( \omega,\, x \right)} \setminus \left\{ 0 \right\} \right\} \\ 
\lim_{n\, \to\, \infty} \frac{1}{n} \log \left\| \left. \left( L_{\omega}^{n} (x) \right)^{-1} \right|_{V_{\left( \omega,\, x \right)}} \right\|^{-1}_{{\rm op}} & = & \min \left\{ \lambda \left( \omega,\, x,\, v \right)\ :\ v \in V_{\left( \omega,\, x \right)} \setminus \left\{ 0 \right\} \right\}. 
\end{eqnarray*} 
\end{proposition} 

\begin{proof} 
Given that $H$ is a $\sigma$-invariant measurable sub-bundle that determines the vector subspace $V_{\left( \omega,\, x \right)}$, we have the set $X_{l} = \left\{ \left( \omega,\, x \right) \in \Sigma_{N}^{+} \times X : \dim \left(V_{(\omega,\, x)}\right) = l \right\}$, for some $1 \le l \le d$, as defined in Theorem \ref{equivalent:3} to be measurable. Thus, there exists a collection of measurable vector fields $\left\{ u_{1} (\omega,\, x), u_{2} (\omega,\, x), \cdots, u_{l} (\omega,\, x) \right\}$ that gives rise to an orthonormal basis (using the Gram-Schmidtt process) represented by $\left\{ v_{1} (\omega,\, x), v_{2} (\omega,\, x), \cdots, v_{l} (\omega,\, x) \right\}$ for the vector subspace $V_{\left( \omega,\, x \right)} \subseteq \mathbb{R}^{d}$, whose dimension is $l$. Thus, one can find an isometric isomorphism between $V_{\left( \omega,\, x \right)}$ and the Euclidean space $\mathbb{R}^{l}$. Further, denote the restriction of $L(x)$ on the $l$-dimensional vector space $V_{\left( \omega,\, x \right)}$ as $L_{l} (x) = \left. L(x) \right|_{V_{\left( \omega,\, x \right)}} \in {\rm GL}_{l} (\mathbb{R})$. 

Now, define a dynamical system 
\begin{equation} 
\label{essell} 
S_{l}\ \ :\ \ \Sigma_{N}^{+} \times X \times \mathbb{R}^{l} \circlearrowright\ \ \ \text{given by}\ \ \ S_{l} \left( \omega,\, x,\, v \right)\ =\ \left( \sigma \omega,\, T_{\pi_{1} \left( \omega \right)} x,\, L_{l} (x) v \right). 
\end{equation} 
Note that $\displaystyle{\left\| L_{l} (x)^{\pm 1} \right\|_{{\rm op}} \le \left\| L (x)^{\pm 1} \right\|_{{\rm op}}}$ and hence, $\log^{+} \left\| L_{l} (x)^{\pm 1} \right\| \in \mathscr{L}^{1} (\mu)$. Thus, Theorem \ref{thm:fkomega} asserts the existence of the limits 
\[ \eta^{\pm} \left( \omega,\, x \right)\ \ =\ \ \lim_{n\, \to\, \infty} \frac{1}{n} \log \left\| \prod_{j\, =\, 0}^{n - 1} \left( L_{l} \left( T_{\pi_{j} \left( \omega \right)} x \right) \right)^{\pm 1} \right\|_{{\rm op}}^{\pm 1}\ \ =\ \ \lim_{n\, \to\, \infty} \frac{1}{n} \log \left\| \prod_{j\, =\, 0}^{n - 1} \left. \left( L \left( T_{\pi_{j} \left( \omega \right)} x \right) \right)^{\pm 1} \right|_{V_{\left( \omega,\, x \right)}} \right\|_{{\rm op}}^{\pm 1}, \] 
for $\left( \mathfrak{B} \times \mu \right)$-almost every $\left( \omega,\, x \right) \in \Sigma_{N}^{+} \times X$. 

In order to apply Corollary \ref{cor:nustar} of Proposition \ref{nustar}, we projectivise the Euclidean space $\mathbb{R}^{l}$ and consider ${\rm P}\mathbb{R}^{l}$ that is equal to the set of all $1$-dimensional subspaces of $\mathbb{R}^{l}$ given by $\left\{ tv : t \in \mathbb{R} \right\}$ for $v \in \mathbb{R}^{l}_{*}$. Consequently, we redefine the dynamical system $S_{l}$ defined in Equation \eqref{essell} on $\Sigma_{N}^{+} \times X \times \mathbb{R}^{l}$ as $\mathcal{S}_{l}$ on $\Sigma_{N}^{+} \times X \times {\rm P}\mathbb{R}^{l}$ given by $\mathcal{S}_{l} \left( \omega,\, x,\, [v] \right) = \left( \sigma \omega,\, T_{\pi_{1} \left( \omega \right)} x,\, \left[ L_{l} (x) v \right] \right)$. Now, define the function $\phi : \Sigma_{N}^{+} \times X \times {\rm P}\mathbb{R}^{l} \longrightarrow \mathbb{R}$ by $\displaystyle{\phi \left( \omega,\, x,\, [v] \right) = \log \left( \frac{\left\| L_{l} (x) v \right\|}{\left\| v \right\|} \right)}$. Then, $\phi \in \mathcal{M} \left( \Sigma_{N}^{+} \times X \times {\rm P}\mathbb{R}^{l} \right)$, since $\phi \left( \omega,\, x,\, \cdot \right) \in \mathcal{C} \left( {\rm P}\mathbb{R}^{l} \right)$ for almost every $\left( \omega,\, x \right) \in \Sigma_{N}^{+} \times X$. Thus, applying Proposition \ref{nustar} to the dynamical system $\mathcal{S}_{l}$ on $\Sigma_{N}^{+} \times X \times {\rm P}\mathbb{R}^{l}$, we have 
\[ \phi^{*} \left( \omega,\, x \right) = \lim_{n\, \to\, \infty} \frac{1}{n} \sup_{v\, \in\, \mathbb{R}^{l}_{*}} \sum_{j\, =\, 0}^{n - 1} \phi \left( \mathcal{S}_{l}^{j} \left( \omega,\, x,\, [v] \right) \right) = \lim_{n\, \to\, \infty} \frac{1}{n} \log \left\| \prod_{j\, =\, 0}^{n - 1} \left( L_{l} \left( T_{\pi_{j} \left( \omega \right)} x \right) \right) \right\|_{{\rm op}} = \eta^{+} \left( \omega,\, x \right). \] 
Analogously, we also have $\phi_{*} \left( \omega,\, x \right) = \eta^{-} \left( \omega,\, x \right)$. 

Since $\mathbb{R}^{l}$ is isometrically isomorphic to $V_{\left( \omega,\, x \right)}$, we have for any vector $v \in \mathbb{R}^{l}_{*}$, 
\[ \limsup_{n\, \to\, \infty} \frac{1}{n} \sum_{j\, =\, 0}^{n - 1} \phi \left( \mathcal{S}_{l}^{j} \left( \omega,\, x,\, [v] \right) \right) = \limsup_{n\, \to\, \infty} \frac{1}{n} \log \left( \frac{\left\| \prod\limits_{j\, =\, 0}^{n - 1} \left( L_{l} \left( T_{\pi_{j} \left( \omega \right)} x \right) \right) v \right\|}{\left\| v \right\|} \right) = \lambda \left( \omega,\, x,\, v \right). \]

Now, Corollary \ref{cor:nustar} asserts the existence of the vectors $w_{*} \left( \omega,\, x \right)$ and $w^{*} \left( \omega,\, x \right)$ in $\mathbb{R}^{l}_{*}$ such that 
\[ \lambda \left( \omega,\, x,\, w^{*} \left( \omega,\, x \right) \right) = \phi^{*} \left( \omega,\, x \right) = \eta^{+} \left( \omega,\, x \right)\ \ \ \text{and}\ \ \ \lambda \left( \omega,\, x,\, w_{*} \left( \omega,\, x \right) \right) = \phi_{*} \left( \omega,\, x \right) = \eta^{-} \left( \omega,\, x \right). \]

Analogous to Lemma \ref{lambdapmbound}, in the current scenario we have $\eta^{-} \left( \omega,\, x \right) \le \lambda \left( \omega,\, x,\, v \right) \le \eta^{+} \left( \omega,\, x \right)$ for all $v \in \mathbb{R}^{l}$. Hence, 
\begin{displaymath} 
\begin{array}{r c l c l} 
\lambda \left( \omega,\, x,\, w^{*} \left( \omega,\, x \right) \right) & = & \eta^{+} \left( \omega,\, x \right) & = & \max \left\{ \lambda \left( \omega,\, x,\, v \right) : v \in \mathbb{R}^{l}_{*} \right\} \vspace{+5pt} \\ 
\lambda \left( \omega,\, x,\, w_{*} \left( \omega,\, x \right) \right) & = & \eta^{-} \left( \omega,\, x \right) & = & \min \left\{ \lambda \left( \omega,\, x,\, v \right) : v \in \mathbb{R}^{l}_{*} \right\}. 
\end{array} 
\end{displaymath} 
\end{proof}

\section{The action of $L(x)$ on $V_{\left( \omega,\, x \right)}^{\perp} \oplus V_{\left( \omega,\, x \right)}$} 
\label{action} 

Let $H$ be a $\sigma$-invariant measurable sub-bundle. Consider a $\left( \mathfrak{B} \times \mu \right)$-typical point $\left( \omega,\, x \right) \in \Sigma_{N}^{+} \times X$. Then, $V_{\left( \omega,\, x \right)}^{\perp} \oplus V_{\left( \omega,\, x \right)} = \mathbb{R}^{d} = V_{\left( \sigma \omega,\, T_{\pi_{1} \left( \omega \right)} x \right)}^{\perp} \oplus V_{\left( \sigma \omega,\, T_{\pi_{1} \left( \omega \right)} x \right)}$ provide for two orthogonal decompositions of $\mathbb{R}^{d}$. Then, we can express the matrix $L(x)$ as 
\begin{eqnarray} 
\label{eqn:decompL}
L (x)\ \ =\ \ \begin{pmatrix} L_{11} (x) & L_{12} (x) \\ L_{21} (x) & L_{22} (x) \end{pmatrix},\ \ \ \text{where}\ \ \ L_{11} (x) & : & V_{\left( \omega,\, x \right)}^{\perp} \longrightarrow V_{\left( \sigma \omega,\, T_{\pi_{1} \left( \omega \right)} x \right)}^{\perp}, \nonumber \\ 
L_{12} (x) & : & V_{\left( \omega,\, x \right)} \longrightarrow V_{\left( \sigma \omega,\, T_{\pi_{1} \left( \omega \right)} x \right)}^{\perp}, \nonumber \\ 
L_{21} (x) & : & V_{\left( \omega,\, x \right)}^{\perp} \longrightarrow V_{\left( \sigma \omega,\, T_{\pi_{1} \left( \omega \right)} x \right)}\ \ \text{and} \nonumber \\ 
L_{22} (x) & : & V_{\left( \omega,\, x \right)} \longrightarrow V_{\left( \sigma \omega,\, T_{\pi_{1} \left( \omega \right)} x \right)}. 
\end{eqnarray} 
Since $H$ is a $\sigma$-invariant measurable sub-bundle, we have $L_{12} (x) \equiv 0$. Further, since $\log^{+} \left\| L(x)^{\pm 1} \right\| \in \mathscr{L}^{1} \left( \mu \right)$, we have $\log^{+} \left\| L_{11}(x)^{\pm 1} \right\|, \log^{+} \left\| L_{21}(x)^{\pm 1} \right\|, \log^{+} \left\| L_{22}(x)^{\pm 1} \right\| \in \mathscr{L}^{1} \left( \mu \right)$. 

\begin{lemma} 
\label{lem:twoparts} 
For $\left( \mathfrak{B} \times \mu \right)$-almost every $\left( \omega,\, x \right) \in \Sigma_{N}^{+} \times X$, we have the following. 
\begin{enumerate} 
\item For every $u \in V_{\left( \omega,\, x \right)}^{\perp} \setminus \{ 0 \}$ and $v \in V_{\left( \omega,\, x \right)}$, we have 
\[ \limsup_{n\, \to\, \infty} \frac{1}{n} \log \left\| \prod\limits_{j\, =\, 0}^{n - 1} \left( L_{11} \left( T_{\pi_{j} \left( \omega \right)} x \right) \right) u \right\|\ \ =\ \ \limsup_{n\, \to\, \infty} \frac{1}{n} \log \left\| \prod\limits_{j\, =\, 0}^{n - 1} \left( L \left( T_{\pi_{j} \left( \omega \right)} x \right) \right) \left( u + v \right) \right\|. \] 
\item Suppose $\displaystyle{\lim_{n\, \to\, \infty} \frac{1}{n} \log \left\| \prod\limits_{j\, =\, 0}^{n - 1} \left( L_{11} \left( T_{\pi_{j} \left( \omega \right)} x \right) \right) u \right\|}$ exists for some vector $u \in V_{\left( \omega,\, x \right)}^{\perp} \setminus \{ 0 \}$, then so does $\displaystyle{\lim_{n\, \to\, \infty} \frac{1}{n} \log \left\| \prod\limits_{j\, =\, 0}^{n - 1} \left( L \left( T_{\pi_{j} \left( \omega \right)} x \right) \right) \left( u + v \right)\right\|}$ for every $v \in V_{\left( \omega,\, x \right)} \setminus \{ 0 \}$. Moreover, in this case, 
\[ \lim_{n\, \to\, \infty} \frac{1}{n} \log \left\| \prod\limits_{j\, =\, 0}^{n - 1} \left( L_{11} \left( T_{\pi_{j} \left( \omega \right)} x \right) \right) u \right\|\ \ =\ \ \lim_{n\, \to\, \infty} \frac{1}{n} \log \left\| \prod\limits_{j\, =\, 0}^{n - 1} \left( L \left( T_{\pi_{j} \left( \omega \right)} x \right) \right) \left( u + v \right) \right\|. \] 
\end{enumerate} 
\end{lemma} 

\begin{proof} 
\begin{enumerate} 
\item For any $u \in V_{\left( \omega,\, x \right)}^{\perp} \setminus \{ 0 \}$ and $v \in V_{\left( \omega,\, x \right)}$, we know from Equation \eqref{lamda(v1+v2)} 
\begin{equation} 
\label{ninepointone} 
\lambda \left( \omega,\, x,\, (u + v) \right)\ \ \le\ \ \max \left\{ \lambda \left( \omega,\, x,\, u \right),\ \lambda \left( \omega,\, x,\, v \right) \right\}. 
\end{equation} 
Choose $\left( \mathfrak{B} \times \mu \right)$-integrable functions $\alpha$ and $\beta$ that satisfies 
\begin{enumerate} 
\item $\alpha < \beta$ on $\Sigma_{N}^{+} \times X$, 
\item $\alpha \left( \sigma^{j} \omega,\, T_{\pi_{j} \left( \omega \right)} x \right) = \alpha \left( \omega,\, x \right)$ and $\beta \left( \sigma^{j} \omega,\, T_{\pi_{j} \left( \omega \right)} x \right) = \beta \left( \omega,\, x \right)$ for all $j \in \mathbb{Z}_{+}$ and 
\item $\lambda \left( \omega,\, x,\, v \right) \le \alpha \left( \omega,\, x \right)$ for every $v \in V_{\left( \omega,\, x \right)} \setminus \{ 0 \}$ and $\lambda \left( \omega,\, x,\, u \right) \ge \beta \left( \omega,\, x \right)$ for every $u \in \mathbb{R}^{d} \setminus V_{\left( \omega,\, x \right)}$. 
\end{enumerate} 
Note that such a choice of functions is possible, since $\eta^{-}$ and $\eta^{+}$ are two such functions that satisfy the above requirements. Hence, for any $v \in V_{\left( \omega,\, x \right)} \setminus \{ 0 \}$ and $u \in \mathbb{R}^{d} \setminus V_{\left( \omega,\, x \right)}$, we have 
\begin{equation} 
\label{eqn:ablcompare} 
\lambda \left( \omega,\, x,\, v \right)\ \le\ \alpha \left( \omega,\, x \right)\ <\ \beta \left( \omega,\, x \right)\ \le\ \lambda \left( \omega,\, x,\, u \right), 
\end{equation} 
prompting that the maximum in Equation \eqref{ninepointone} is attained by the first term. Thus, $\lambda \left( \omega,\, x,\, (u + v) \right) \le \lambda \left( \omega,\, x,\, u \right)$. We now prove the other way inequality and then proceed to conclude the first part of the lemma. Consider 
\begin{equation} 
\label{ninepointthree} 
\lambda \left( \omega,\, x,\, u \right)\ \ \le\ \ \max \left\{ \lambda \left( \omega,\, x,\, \left( u + v \right) \right),\ \lambda \left( \omega,\, x,\, v \right) \right\}. 
\end{equation} 
Suppose we prove that the maximum occurs in the first quantity in Equation \eqref{ninepointthree}, we can assert that $\lambda \left( \omega,\, x,\, \left( u + v \right) \right) = \lambda \left( \omega,\, x,\, u \right)$. Towards that end, we first observe that taking $u \in V_{\left( \omega,\, x \right)}^{\perp} \setminus \{ 0 \}$ and $v \in V_{\left( \omega,\, x \right)} \setminus \{ 0 \}$ ensures that $u + v \notin V_{\left( \omega,\, x \right)}$. Thus, $u + v \in \mathbb{R}^{d} \setminus V_{\left( \omega,\, x \right)}$. Now, employing Equation \eqref{eqn:ablcompare}, we observe that the maximum among the two quantities in Equation \eqref{ninepointthree} is attained by the first term, proving that 
\begin{equation} 
\label{luvlu} 
\lambda \left( \omega,\, x,\, \left( u + v \right) \right)\ \ =\ \ \lambda \left( \omega,\, x,\, u \right),\ \ \ \ \text{for}\ u \in V_{\left( \omega,\, x \right)}^{\perp} \setminus \{ 0 \}\ \text{and}\ v \in V_{\left( \omega,\, x \right)}. 
\end{equation} 
Thus, it is sufficient for us to  with $u \in V_{\left( \omega,\, x \right)}^{\perp} \setminus \{ 0 \}$ and $v \equiv 0$. 

Since the matrix $\displaystyle{L(x) = \begin{pmatrix} L_{11}(x) & 0 \\ L_{21}(x) & L_{22}(x) \end{pmatrix}}$, one can evaluate using induction that 
\[ \prod\limits_{j\, =\, 0}^{n - 1} \left( L \left( T_{\pi_{j} \left( \omega \right)} x \right) \right) = \begin{pmatrix} \prod\limits_{j\, =\, 0}^{n - 1} \left( L_{11} \left( T_{\pi_{j} \left( \omega \right)} x \right) \right) & 0 \\ \left( \prod\limits_{j\, =\, 0}^{n - 1} \left( L \left( T_{\pi_{j} \left( \omega \right)} x \right) \right) \right)_{21} & \prod\limits_{j\, =\, 0}^{n - 1} \left( L_{22} \left( T_{\pi_{j} \left( \omega \right)} x \right) \right) \end{pmatrix}, \] 
where 
\begin{eqnarray*} 
& & \left( \prod\limits_{j\, =\, 0}^{n - 1} \left( L \left( T_{\pi_{j} \left( \omega \right)} x \right) \right) \right)_{21} \\ 
& = & \left( \prod\limits_{j\, =\, 1}^{n - 1} \left( L_{22} \left( T_{\pi_{j} \left( \omega \right)} x \right) \right) \right) L_{21} (x) \\ 
& & +\ \left( \prod\limits_{j\, =\, 2}^{n - 1} \left( L_{22} \left( T_{\pi_{j} \left( \omega \right)} x \right) \right) \right) L_{21} \left( T_{\pi_{1} \left( \omega \right)} x \right) L_{11} (x) \\ 
& & +\ \left( \prod\limits_{j\, =\, 3}^{n - 1} \left( L_{22} \left( T_{\pi_{j} \left( \omega \right)} x \right) \right) \right) L_{21} \left( T_{\pi_{2} \left( \omega \right)} x \right) \left( \prod_{k\, =\, 0}^{1} L_{11} \left( T_{\pi_{k} \left( \omega \right)} x \right) \right) \\ 
& & +\ \cdots \\ 
& & +\ \left( L_{22} \left( T_{\pi_{n - 1} \left( \omega \right)} x \right) \right) L_{21} \left( T_{\pi_{n - 2} \left( \omega \right)} x \right) \left( \prod_{k\, =\, 0}^{n - 3} L_{11} \left( T_{\pi_{k} \left( \omega \right)} x \right) \right) \\ 
& & +\ L_{21} \left( \left( T_{\pi_{n - 1} \left( \omega \right)} x \right) \right) \left( \prod\limits_{j\, =\, 0}^{n - 2} \left( L_{11} \left( T_{\pi_{j} \left( \omega \right)} x \right) \right) \right) \\ 
& = & \sum_{j\, =\, 0}^{n - 1} \left[ \left( \prod_{k\, =\, j + 1}^{n - j - 1} L_{22} \left( T_{\pi_{k} \left( \omega \right)} x \right) \right) L_{21} \left( T_{\pi_{j} \left( \omega \right)} x \right) \left( \prod_{k\, =\, 0}^{j} L_{11} \left( T_{\pi_{k} \left( \omega \right)} x \right) \right) \right]. 
\end{eqnarray*} 

Given any $u \in V_{\left( \omega,\, x \right)}^{\perp} \setminus \{ 0 \}$, in order to estimate an upper bound for $\displaystyle{\left\| \prod\limits_{j\, =\, 0}^{n - 1} \left( L \left( T_{\pi_{j} \left( \omega \right)} x \right) \right) u \right\|}$, we compute an upper bound for $\displaystyle{\left\| \prod\limits_{j\, =\, 0}^{n - 1} \left( L_{11} \left( T_{\pi_{j} \left( \omega \right)} x \right) \right) u \right\|,\ \left\| \prod_{k\, =\, j + 1}^{n - j - 1} L_{22} \left( T_{\pi_{k} \left( \omega \right)} x \right) \right\|_{{\rm op}}}$ and $\displaystyle{\left\| L_{21} \left( T_{\pi_{j} \left( \omega \right)} x \right) \right\|_{{\rm op}}}$ for $0 \le j \le n - 1$. We now state our next lemma that captures these estimates, whose proof is deferred to the end of the current section. 

\begin{lemma} 
\label{lem:estimates}
Given any $\epsilon > 0$, there exist $\left( \mathfrak{B} \times \mu \right)$-measurable functions on $\Sigma_{N}^{+} \times X$, namely $A_{\epsilon}, B_{\epsilon}$ and $C_{\epsilon}$ such that for any $n \in \mathbb{Z}_{+}$, we have 
\begin{enumerate} 
\item $\displaystyle{\left\| \prod\limits_{j\, =\, 0}^{n - 1} \left( L_{11} \left( T_{\pi_{j} \left( \omega \right)} x \right) \right) u \right\| \le A_{\epsilon} \left( \omega,\, x \right) e^{n \gamma \left( \omega,\, x \right) + n \epsilon}}$, where 
\begin{equation} 
\label{gamma} 
\gamma \left( \omega,\, x \right)\ \ =\ \ \max \left\{ \alpha \left( \omega,\, x \right),\ \limsup_{n\, \to\, \infty} \frac{1}{n} \log \left( \left\| \prod\limits_{j\, =\, 0}^{n - 1} \left( L_{11} \left( T_{\pi_{j} \left( \omega \right)} x \right) \right) u \right\| \right) \right\}, 
\end{equation} 
\item $\displaystyle{\left\| \prod_{k\, =\, j + 1}^{n - j - 1} L_{22} \left( T_{\pi_{k} \left( \omega \right)} x \right) \right\|_{{\rm op}} \le B_{\epsilon} \left( \omega,\, x \right) e^{(n - j - 1) \alpha \left( \omega,\, x \right) + n \epsilon}}$, for $0 \le j \le n - 1$ and  
\item $\displaystyle{\left\| L_{21} \left( T_{\pi_{n} \left( \omega \right)} x \right) \right\|_{{\rm op}} \le C_{\epsilon} \left( \omega,\, x \right) e^{n \epsilon}}$. 
\end{enumerate} 
\end{lemma} 

We further state a lemma for positive real sequences, as may be found in \cite{mv:2014}. 

\begin{lemma} 
\label{lem:seq}
Suppose $\left\{ a_{n} \right\}$ and $\left\{ b_{n} \right\}$ are positive real sequences, then 
\begin{eqnarray*} 
\limsup_{n\, \to\, \infty} \frac{1}{n} \log \left( \sqrt{a_{n}^{2} + b_{n}^{2}} \right) & = & \max \left\{ \limsup_{n\, \to\, \infty} \frac{1}{n} \log a_{n}, \limsup_{n\, \to\, \infty} \frac{1}{n} \log b_{n} \right\}, \\ 
\liminf_{n\, \to\, \infty} \frac{1}{n} \log \left( \sqrt{a_{n}^{2} + b_{n}^{2}} \right) & \ge & \max \left\{ \liminf_{n\, \to\, \infty} \frac{1}{n} \log a_{n}, \liminf_{n\, \to\, \infty} \frac{1}{n} \log b_{n} \right\}. 
\end{eqnarray*} 
\end{lemma} 

Thus, 
\begin{eqnarray} 
& & \lambda \left( \omega,\, x,\, u \right) \nonumber \\ 
& = & \limsup_{n\, \to\, \infty} \frac{1}{n} \log \left( \left\| \prod\limits_{j\, =\, 0}^{n - 1} \left( L \left( T_{\pi_{j} \left( \omega \right)} x \right) \right) u \right\| \right) \nonumber \\ 
& = & \limsup_{n\, \to\, \infty} \frac{1}{n} \log \left( \left\| \prod\limits_{j\, =\, 0}^{n - 1} \left( L_{11} \left( T_{\pi_{j} \left( \omega \right)} x \right) \right) u \right\|^{2} \right. \nonumber \\ 
& & + \left. \left\| \sum_{j\, =\, 0}^{n - 1} \left[ \left( \prod_{k\, =\, j + 1}^{n - j - 1} L_{22} \left( T_{\pi_{k} \left( \omega \right)} x \right) \right) L_{21} \left( T_{\pi_{j} \left( \omega \right)} x \right) \left( \prod_{k\, =\, 0}^{j} L_{11} \left( T_{\pi_{k} \left( \omega \right)} x \right) u \right) \right] \right\|^{2} \right)^{\frac{1}{2}} \nonumber \\ 
& = & \max \left\{ \limsup_{n\, \to\, \infty} \frac{1}{n} \log \left( \left\| \prod\limits_{j\, =\, 0}^{n - 1} \left( L_{11} \left( T_{\pi_{j} \left( \omega \right)} x \right) \right) u \right\| \right), \right. \nonumber \\ 
& & \limsup_{n\, \to\, \infty} \frac{1}{n} \log \left( \left\| \sum_{j\, =\, 0}^{n - 1} \left[ \left( \prod_{k\, =\, j + 1}^{n - j - 1} L_{22} \left( T_{\pi_{k} \left( \omega \right)} x \right) \right) L_{21} \left( T_{\pi_{j} \left( \omega \right)} x \right) \right. \right. \right. \nonumber \\ 
& & \left. \left. \left. \left. \hspace{+8cm} \left( \prod_{k\, =\, 0}^{j} L_{11} \left( T_{\pi_{k} \left( \omega \right)} x \right) u \right) \right] \right\| \right) \right\} \nonumber \\ 
& \le & \max \left\{ \gamma \left( \omega,\, x \right), \gamma \left( \omega,\, x \right) + 3 \epsilon \right\}, \hspace{+2cm} \text{where}\ \gamma\ \text{is as defined in Equation \eqref{gamma}} \nonumber \\ 
\label{ninefour}
& & 
\end{eqnarray} 
since 
\begin{eqnarray*} 
& & \left\| \sum_{j\, =\, 0}^{n - 1} \left[ \left( \prod_{k\, =\, j + 1}^{n - j - 1} L_{22} \left( T_{\pi_{k} \left( \omega \right)} x \right) \right) L_{21} \left( T_{\pi_{j} \left( \omega \right)} x \right) \left( \prod_{k\, =\, 0}^{j} L_{11} \left( T_{\pi_{k} \left( \omega \right)} x \right) u \right) \right] \right\| \\ 
& \le & \sum_{j\, =\, 0}^{n - 1} A_{\epsilon} \left( \omega,\, x \right) B_{\epsilon} \left( \omega,\, x \right) C_{\epsilon} \left( \omega,\, x \right) e^{(n - 1) \gamma (x) + (n + 2j) \epsilon} \\ 
& \le & n D_{\epsilon} \left( \omega,\, x \right) e^{n \gamma \left( \omega,\, x \right) + 3n \epsilon}, 
\end{eqnarray*} 
for an appropriately defined $\left( \mathfrak{B} \times \mu \right)$-measurable function $D_{\epsilon}$ on $\Sigma_{N}^{+} \times X$. Therefore, using Equation \eqref{ninefour}, we have 
\begin{equation} 
\label{ninesix} 
\limsup_{n\, \to\, \infty} \frac{1}{n} \log \left( \left\| \prod\limits_{j\, =\, 0}^{n - 1} \left( L_{11} \left( T_{\pi_{j} \left( \omega \right)} x \right) \right) u \right\| \right)\ \le\ \lambda \left( \omega,\, x,\, u \right)\ \le\ \gamma \left( \omega,\, x \right) + 3 \epsilon. 
\end{equation} 
Since $\epsilon$ is an arbitrarily small positive quantity, we obtain $\lambda \left( \omega,\, x,\, u \right) \le \gamma \left( \omega,\, x \right)$. Moreover, by our choice of the functions $\alpha$, that satisfies Equation \eqref{eqn:ablcompare} and $\gamma$, as defined in Equation \eqref{gamma}, it is clear that
\begin{equation} 
\label{nineseven} 
\gamma \left( \omega,\, x \right)\ \ =\ \ \limsup_{n\, \to\, \infty} \frac{1}{n} \log \left( \left\| \prod\limits_{j\, =\, 0}^{n - 1} \left( L_{11} \left( T_{\pi_{j} \left( \omega \right)} x \right) \right) u \right\| \right). 
\end{equation} 
Thus, plugging in Equation \eqref{nineseven} in Equation \eqref{ninesix} along with the arbitrariness of $\epsilon$, we obtain 
\[ \gamma \left( \omega,\, x \right)\ \ =\ \ \limsup_{n\, \to\, \infty} \frac{1}{n} \log \left( \left\| \prod\limits_{j\, =\, 0}^{n - 1} \left( L_{11} \left( T_{\pi_{j} \left( \omega \right)} x \right) \right) u \right\| \right)\ \ \le\ \ \lambda \left( \omega,\, x,\, u \right)\ \ \le\ \ \gamma \left( \omega,\, x \right). \]
Finally, using Equation \eqref{luvlu}, we have 
\[ \limsup_{n\, \to\, \infty} \frac{1}{n} \log \left\| \prod\limits_{j\, =\, 0}^{n - 1} \left( L_{11} \left( T_{\pi_{j} \left( \omega \right)} x \right) \right) u \right\|\ \ =\ \ \lambda \left( \omega,\, x,\, u \right)\ \ =\ \ \lambda \left( \omega,\, x,\, \left( u + v \right) \right), \] 
whenever $u \in V_{\left( \omega,\, x \right)}^{\perp} \setminus \{ 0 \}$ and $v \in V_{\left( \omega,\, x \right)}$. 

\item Let $u \in V_{\left( \omega,\, x \right)}^{\perp} \setminus \{ 0 \}$ and $v \in V_{\left( \omega,\, x \right)}$. Then, 
\begin{eqnarray*} 
& & \liminf_{n\, \to\, \infty} \frac{1}{n} \log \left( \left\| \prod\limits_{j\, =\, 0}^{n - 1} \left( L \left( T_{\pi_{j} \left( \omega \right)} x \right) \right) \left( u + v \right) \right\| \right) \\ 
& = & \liminf_{n\, \to\, \infty} \frac{1}{n} \log \left( \left\| \prod\limits_{j\, =\, 0}^{n - 1} \left( L_{11} \left( T_{\pi_{j} \left( \omega \right)} x \right) \right) u \right\|^{2} \right. \\ 
& & + \left\| \sum_{j\, =\, 0}^{n - 1} \left[ \left( \prod_{k\, =\, j + 1}^{n - j - 1} L_{22} \left( T_{\pi_{k} \left( \omega \right)} x \right) \right) L_{21} \left( T_{\pi_{j} \left( \omega \right)} x \right) \left( \prod_{k\, =\, 0}^{j} L_{11} \left( T_{\pi_{k} \left( \omega \right)} x \right) u \right) \right] \right. \\ 
& & \left. \left. \hspace{+8cm} + \prod\limits_{j\, =\, 0}^{n - 1} \left( L_{22} \left( T_{\pi_{j} \left( \omega \right)} x \right) \right) v \right\|^{2} \right)^{\frac{1}{2}} 
\end{eqnarray*} 

\begin{eqnarray*} 
& \ge & \max \left\{ \liminf_{n\, \to\, \infty} \frac{1}{n} \log \left\| \prod\limits_{j\, =\, 0}^{n - 1} \left( L_{11} \left( T_{\pi_{j} \left( \omega \right)} x \right) \right) u \right\|, \right. \\ 
& & \liminf_{n\, \to\, \infty} \frac{1}{n} \log \left\| \sum_{j\, =\, 0}^{n - 1} \left[ \left( \prod_{k\, =\, j + 1}^{n - j - 1} L_{22} \left( T_{\pi_{k} \left( \omega \right)} x \right) \right) L_{21} \left( T_{\pi_{j} \left( \omega \right)} x \right) \left( \prod_{k\, =\, 0}^{j} L_{11} \left( T_{\pi_{k} \left( \omega \right)} x \right) u \right) \right] \right. \\ 
& & \left. \left. \hspace{+9.3cm} + \prod\limits_{j\, =\, 0}^{n - 1} \left( L_{22} \left( T_{\pi_{j} \left( \omega \right)} x \right) \right) v \right\| \right\} \\ 
& & \hspace{+10.5cm} \text{(using Lemma \ref{lem:seq})} \\ 
& \ge & \liminf_{n\, \to\, \infty} \frac{1}{n} \log \left\| \prod\limits_{j\, =\, 0}^{n - 1} \left( L_{11} \left( T_{\pi_{j} \left( \omega \right)} x \right) \right) u \right\| \\ 
& = & \limsup_{n\, \to\, \infty} \frac{1}{n} \log \left\| \prod\limits_{j\, =\, 0}^{n - 1} \left( L_{11} \left( T_{\pi_{j} \left( \omega \right)} x \right) \right) u \right\| \hspace{+4.6cm} (\text{by hypothesis}) \\ 
& = & \limsup_{n\, \to\, \infty} \frac{1}{n} \log \left\| \prod\limits_{j\, =\, 0}^{n - 1} \left( L \left( T_{\pi_{j} \left( \omega \right)} x \right) \right) \left( u + v \right) \right\| \hspace{+1.8cm} (\text{by part (1) of Lemma \ref{lem:twoparts}}), 
\end{eqnarray*} 
thus implying in this case that 
\[ \lim_{n\, \to\, \infty} \frac{1}{n} \log \left\| \prod\limits_{j\, =\, 0}^{n - 1} \left( L_{11} \left( T_{\pi_{j} \left( \omega \right)} x \right) \right) u \right\|\ \ =\ \ \lim_{n\, \to\, \infty} \frac{1}{n} \log \left\| \prod\limits_{j\, =\, 0}^{n - 1} \left( L \left( T_{\pi_{j} \left( \omega \right)} x \right) \right) \left( u + v \right) \right\|. \] 
\end{enumerate} 
\end{proof} 

We now prove Lemma \ref{lem:estimates}. 

\begin{proof} \big(of Lemma \ref{lem:estimates}\big) 
\begin{enumerate} 
\item[$(a)$] From the definition of the function $\gamma$, as written in Equation \eqref{gamma}, we have that 
\[ \limsup_{n\, \to\, \infty} \frac{1}{n} \log \left( \left\| \prod\limits_{j\, =\, 0}^{n - 1} \left( L_{11} \left( T_{\pi_{j} \left( \omega \right)} x \right) \right) u \right\| \right)\ \ \le\ \ \gamma \left( \omega,\, x \right). \] 
Thus, for any given $\epsilon > 0$, there exists a threshold, say $N_{1} \in \mathbb{Z}_{+}$ such that for every $n > N_{1}$, we have $\left\| \prod\limits_{j\, =\, 0}^{n - 1} \left( L_{11} \left( T_{\pi_{j} \left( \omega \right)} x \right) \right) u \right\| \le e^{n \gamma \left( \omega,\, x \right) + n \epsilon}$. Now defining 
\[ A_{\epsilon} = \max \left\{ 1,\; \max_{1\, \le\, k\, \le\, N_{1}} \left\{ \frac{\left\| \prod\limits_{j\, =\, 0}^{k - 1} \left( L_{11} \left( T_{\pi_{j} \left( \omega \right)} x \right) \right) u \right\|}{e^{k \gamma \left( \omega,\, x \right) + k \epsilon}} \right\} \right\}, \] 
completes the proof. 

\item[$(b)$] In order to prove this part, we define a function as follows: For any fixed $\mathfrak{p} \in \mathbb{Z}_{+} \cup \left\{ 0 \right\}$ and $\epsilon > 0$, consider 
\begin{eqnarray} 
& & \rho^{\epsilon} \left( \sigma^{\mathfrak{p}} \omega,\, x \right) \nonumber \\ 
& = & \sup_{\mathfrak{q}\, \in\, \mathbb{Z}_{+} \cup \{ 0 \}} \left\| L_{22} \left( T_{\omega_{\mathfrak{p} + \mathfrak{q} - 1}} \circ \cdots \circ T_{\omega_{\mathfrak{p} + 1}} x \right) \cdots L_{22} \left( T_{\omega_{\mathfrak{p} + 1}} x \right) L_{22} (x) \right\|_{{\rm op}} e^{- \mathfrak{q} \left( \alpha \left( \sigma^{\mathfrak{p}} \omega,\, x \right) + \epsilon \right)}. \nonumber \\ 
\label{gsmn} 
& & 
\end{eqnarray} 
The matrices in the product that appears in the definition of $\rho^{\epsilon}$ in Equation \eqref{gsmn} in the degenerate case of the supremum being achieved at $\mathfrak{q} = 0$ for $\mathfrak{p} = 0$ must be understood, in accordance with Equation \eqref{eqn:s^0}, as the identity matrix $I_{d} \in {\rm GL}_{d} \left( \mathbb{R} \right)$. Thus, if the supremum is achieved in the definition of $\rho^{\epsilon}$ at $\mathfrak{q} = 0$, then $\rho^{\epsilon} \left( \omega,\, x \right) \equiv 1$. Moreover, for $\left( \mathfrak{B} \times \mu \right)$-almost every $\left( \omega,\, x \right) \in \Sigma_{N}^{+} \times X$, we have $1 \le \rho^{\epsilon} \left( \sigma^{\mathfrak{p}} \omega,\, x \right) < \infty$, for any $\mathfrak{p} \in \mathbb{Z}_{+} \cup \{ 0 \}$. 

Note that by definition, we have 
\begin{eqnarray*} 
& & \rho^{\epsilon} \left( \sigma \omega,\, T_{\omega_{1}} x \right) \\ 
& = & \sup_{\mathfrak{q}\, \in\, \mathbb{Z}_{+} \cup \{ 0 \}} \left\| L_{22} \left( T_{\pi_{\mathfrak{q}} \left( \omega \right)} x \right) \cdots L_{22} \left( T_{\pi_{1} \left( \omega \right)} x \right) \right\|_{{\rm op}} e^{- \mathfrak{q} \left( \alpha \left( \sigma \omega,\, T_{\pi_{1} \left( \omega \right)} x \right) + \epsilon \right)} \\ 
& = & \sup_{\mathfrak{q}\, \in\, \mathbb{Z}_{+} \cup \{ 0 \}} \left\| L_{22} \left( T_{\pi_{\mathfrak{q}} \left( \omega \right)} x \right) \cdots L_{22} \left( T_{\pi_{1} \left( \omega \right)} x \right) \left[ L_{22} (x) L_{22} (x)^{-1} \right] \right\|_{{\rm op}} e^{- \mathfrak{q} \left( \alpha \left( \omega,\, x \right) + \epsilon \right)} \\ 
& \le & \rho^{\epsilon} \left( \omega,\, x \right) e^{\left( \alpha \left( \omega,\, x \right) + \epsilon \right)} \left\| \left[ L_{22} (x) \right]^{-1} \right\|_{{\rm op}}. 
\end{eqnarray*} 
Thus, 
\begin{equation} 
\label{gepsilonub} 
\log \left( \rho^{\epsilon} \left( \sigma \omega,\, T_{\omega_{1}} x \right) \right) - \log \left( \rho^{\epsilon} \left( \omega,\, x \right) \right)\ \ \le\ \ \log^{+} \left( \left\| \left[ L_{22} (x) \right]^{-1} \right\|_{{\rm op}} \right) + \alpha \left( \omega,\, x \right) + \epsilon. 
\end{equation} 
Also, 
\begin{eqnarray*} 
\rho^{\epsilon} \left( \sigma \omega,\, T_{\omega_{1}} x \right) & = & \sup_{\mathfrak{q}\, \in\, \mathbb{Z}_{+} \cup \{ 0 \}} \left\| L_{22} \left( T_{\pi_{\mathfrak{q}} \left( \omega \right)} x \right) \cdots L_{22} \left( T_{\pi_{1} \left( \omega \right)} x \right) \right\|_{{\rm op}} \frac{\left\| L_{22} (x) \right\|_{{\rm op}}}{\left\| L_{22} (x) \right\|_{{\rm op}}} e^{- \mathfrak{q} \left( \alpha \left( \omega,\, x \right) + \epsilon \right)} \\ 
& \ge & \rho^{\epsilon} \left( \omega,\, x \right) e^{\left( \alpha \left( \omega,\, x \right) + \epsilon \right)} \left\| L_{22} (x) \right\|^{-1}_{{\rm op}}, 
\end{eqnarray*} 
that gives 
\begin{equation} 
\label{gepsilonlb} 
\log \left( \rho^{\epsilon} \left( \sigma \omega,\, T_{\omega_{1}} x \right) \right) - \log \left( \rho^{\epsilon} \left( \omega,\, x \right) \right)\ \ \ge\ \ \log \left( \left\| L_{22} (x) \right\|^{-1}_{{\rm op}} \right) + \alpha \left( \omega,\, x \right) + \epsilon. 
\end{equation} 

Hence, independent of the case whether the supremum is achieved at $\mathfrak{q} = 0$ or $\mathfrak{q} \in \mathbb{Z}_{+}$, we have from Equations \eqref{gepsilonub} and \eqref{gepsilonlb} that 
\begin{eqnarray*} 
\min \left\{ 0,\; - \log^{+} \left\| L_{22} (x) \right\|_{{\rm op}} + \alpha \left( \omega,\, x \right) + \epsilon \right\} & \le & \log \left( \rho^{\epsilon} \left( \sigma \omega,\, T_{\omega_{1}} x \right) \right) - \log \left( \rho^{\epsilon} \left( \omega,\, x \right) \right) \\ 
& \le & \log^{+} \left( \left\| \left[ L_{22} (x) \right]^{-1} \right\|_{{\rm op}} \right) + \alpha \left( \omega,\, x \right) + \epsilon, 
\end{eqnarray*} 
thus, making the function $\log \left( \rho^{\epsilon} \left( \sigma \omega,\, T_{\omega_{1}} x \right) \right) - \log \left( \rho^{\epsilon} \left( \omega,\, x \right) \right)$ to be $\left( \mathfrak{B} \times \mu \right)$-integrable. Analogously, one may prove that $\log \left( \rho^{\epsilon} \left( \sigma^{\mathfrak{p}} \omega,\, T_{\pi_{\mathfrak{p}} \left( \omega \right)} x \right) \right) - \log \left( \rho^{\epsilon} \left( \sigma^{\mathfrak{p} - 1} \omega,\, T_{\pi_{\mathfrak{p} - 1} \left( \omega \right)} x \right) \right)$ is integrable, for any $\mathfrak{p} \in \mathbb{Z}_{+}$. 

Applying the assertion of Theorem \ref{thm:ketomega} to the sequence 
\[ \left\{ \frac{1}{\mathfrak{p}} \sum_{j\, =\, 0}^{\mathfrak{p} - 1} \left[ \log \left( \rho^{\epsilon} \left( \sigma^{j + 1} \omega,\, T_{\pi_{j + 1} \left( \omega \right)} x \right) \right) - \log \left( \rho^{\epsilon} \left( \sigma^{j} \omega,\, T_{\pi_{j} \left( \omega \right)} x \right) \right) \right] \right\}_{\mathfrak{p}\, \in\, \mathbb{Z}_{+}}, \] 
we note that $\displaystyle{\lim\limits_{\mathfrak{p}\, \to\, \infty} \frac{1}{\mathfrak{p}} \sum_{j\, =\, 0}^{\mathfrak{p} - 1} \left[ \log \left( \rho^{\epsilon} \left( \sigma^{j + 1} \omega,\, T_{\pi_{j + 1} \left( \omega \right)} x \right) \right) - \log \left( \rho^{\epsilon} \left( \sigma^{j} \omega,\, T_{\pi_{j} \left( \omega \right)} x \right) \right) \right] \to \mathfrak{g}}$ for $\left( \mathfrak{B} \times \mu \right)$-almost every $\left( \omega,\, x \right) \in \Sigma_{N}^{+} \times X$, where for every $\mathfrak{q} \in \mathbb{Z}_{+}$, the limit function $\mathfrak{g}$ satisfies $\mathfrak{g} \left( \omega,\, x \right) = \mathfrak{g} \left( \sigma^{\mathfrak{q}} \omega,\, T_{\pi_{\mathfrak{q}} \left( \omega \right)} x \right)$, for $\left( \mathfrak{B} \times \mu \right)$-almost every $\left( \omega,\, x \right) \in \Sigma_{N}^{+} \times X$. Moreover, 
\begin{eqnarray*} 
& & \int \mathfrak{g} \mathrm{d} \left( \mathfrak{B} \times \mu \right) \\ 
& = & \lim_{\mathfrak{p}\, \to\, \infty} \frac{1}{\mathfrak{p}} \sum_{j\, =\, 0}^{\mathfrak{p} - 1} \left[ \int \log \left( \rho^{\epsilon} \left( \sigma^{j + 1} \omega,\, T_{\pi_{j + 1} \left( \omega \right)} x \right) \right) \mathrm{d}\mu - \int \log \left( \rho^{\epsilon} \left( \sigma^{j} \omega,\, T_{\pi_{j} \left( \omega \right)} x \right) \right) \mathrm{d} \left( \mathfrak{B} \times \mu \right) \right] \\ 
& = & 0. 
\end{eqnarray*} 

Furthermore, writing 
\begin{eqnarray*} 
& & \log \left( \rho^{\epsilon} \left( \sigma^{\mathfrak{p}} \omega,\, T_{\pi_{\mathfrak{p}} \left( \omega \right)} x \right) \right) \\ 
& = & \log \rho^{\epsilon} \left( \omega,\, x \right) + \sum_{j\, =\, 0}^{\mathfrak{p} - 1} \left[ \log \left( \rho^{\epsilon} \left( \sigma^{j + 1} \omega,\, T_{\pi_{j + 1} \left( \omega \right)} x \right) \right) - \log \left( \rho^{\epsilon} \left( \sigma^{j} \omega,\, T_{\pi_{j} \left( \omega \right)} x \right) \right) \right], 
\end{eqnarray*} 
we note that 
\begin{eqnarray*} 
& & \lim_{\mathfrak{p}\, \to\, \infty} \frac{1}{\mathfrak{p}} \log \left( \rho^{\epsilon} \left( \sigma^{\mathfrak{p}} \omega,\, T_{\pi_{\mathfrak{p}} \left( \omega \right)} x \right) \right) \\ 
& = & \lim_{\mathfrak{p}\, \to\, \infty} \frac{1}{\mathfrak{p}} \left[ \log \rho^{\epsilon} \left( \omega,\, x \right) + \sum_{j\, =\, 0}^{\mathfrak{p} - 1} \left[ \log \left( \rho^{\epsilon} \left( \sigma^{j + 1} \omega,\, T_{\pi_{j + 1} \left( \omega \right)} x \right) \right) - \log \left( \rho^{\epsilon} \left( \sigma^{j} \omega,\, T_{\pi_{j} \left( \omega \right)} x \right) \right) \right] \right] \\ 
& = & 0. 
\end{eqnarray*} 
 
Hence, for any given $\epsilon > 0$, there exists a threshold, say $N_{2} \in \mathbb{Z}_{+}$ such that for every $\mathfrak{p} > N_{2}$, we have $\displaystyle{\rho^{\epsilon} \left( \sigma^{\mathfrak{p}} \omega,\, T_{\pi_{\mathfrak{p}} \left( \omega \right)} x \right) \le e^{\mathfrak{p} \epsilon}}$. As earlier, defining 
\[ B_{\epsilon}\ \ =\ \ \max \left\{ 1, \max_{1\, \le\, k\, \le\, N_{2}} \left\{ \frac{\rho^{\epsilon} \left( \sigma^{k} \omega,\, T_{\pi_{k} \left( \omega \right)} x \right)}{e^{k \epsilon}} \right\} \right\}, \] 
we obtain $\rho^{\epsilon} \left( \sigma^{\mathfrak{p}} \omega,\, T_{\pi_{\mathfrak{p}} \left( \omega \right)} x \right) \le B_{\epsilon} e^{\mathfrak{p} \epsilon}$ for every $\mathfrak{p} \in \mathbb{Z}_{+}$. Now, using the definition of the function $\rho^{\epsilon} \left( \sigma^{\mathfrak{p}} \omega,\, T_{\pi_{\mathfrak{p}} \left( \omega \right)} x \right)$, from Equation \eqref{gsmn}, we have for every $\mathfrak{p} \in \mathbb{Z}_{+}$ and $\mathfrak{q} \in \mathbb{Z}_{+} \cup \{ 0 \}$, 
\[ \left\| L_{22} \circ T_{\pi_{\mathfrak{p} + \mathfrak{q} - 1} \left( \omega \right)} (x) \cdots L_{22} \circ T_{\pi_{\mathfrak{p}} \left( \omega \right)} (x) \right\|_{{\rm op}}\ \ \le\ \ B_{\epsilon} e^{(\mathfrak{p} + \mathfrak{q}) \epsilon + \mathfrak{q} \left( \alpha \left( \omega,\, x \right) \right)}. \] 
The particular case of $\mathfrak{p} = j + 1$ and $\mathfrak{q} = n - j - 1$ clinches the proof. 

\item[$(c)$] In order to prove this part, we need to obtain the following: Given $\epsilon > 0$, there exists a threshold, say $N_{3} \in \mathbb{Z}_{+}$ such that for every $n > N_{3}$, we have 
\[ \frac{1}{n} \log \left\| \left( L_{21} \left( T_{\pi_{n} \left( \omega \right)} x \right) \right) \right\|_{{\rm op}}\ \ \le\ \ \epsilon\ \ \text{for}\ \left( \mathfrak{B} \times \mu \right)\text{-almost every}\ \left( \omega,\, x \right) \in \Sigma_{N}^{+} \times X. \] 
Once we obtain the above mentioned threshold $N_{3}$, defining 
\[ C_{\epsilon} = \max \left\{ 1,\; \max_{1\, \le\, k\, \le\, N_{3}} \left\{ \frac{\left\| \left( L_{21} \left( T_{\pi_{k} \left( \omega \right)} x \right) \right) \right\|_{{\rm op}}}{e^{k \epsilon}} \right\} \right\}, \] 
completes the proof. Towards that end, consider the sequence of sets 
\[ \mathcal{A}_{n}\ \ =\ \ \left\{ \left( \omega,\, x \right) \in \Sigma_{N}^{+} \times X\ :\ \left| \log \left\| \left( L_{21} \left( T_{\pi_{n} \left( \omega \right)} x \right) \right) \right\|_{{\rm op}} \right|\ \ge\ n \epsilon \right\}. \] 
Then, 
\begin{eqnarray*} 
& & \sum_{n\, \ge\, 1} \left( \mathfrak{B} \times \mu \right) \left( \mathcal{A}_{n} \right) \\ 
& = & \sum_{n\, \ge\, 1} \left( \mathfrak{B} \times \mu \right) \left( \left\{ \left( \omega,\, x \right) \in \Sigma_{N}^{+} \times X\ :\ \left| \log \left\| \left( L_{21} \left( T_{\pi_{n} \left( \omega \right)} x \right) \right) \right\|_{{\rm op}} \right|\ \ge\ n \epsilon \right\} \right) \\ 
& = & \sum_{n\, \ge\, 1} \sum_{k\, \ge\, n} \mu \left( \left\{ x \in X\ :\ k \le \frac{\left| \log \left\| \left( L_{21} \left( x \right) \right) \right\|_{{\rm op}} \right|}{\epsilon} < k + 1 \right\} \right) \\ 
& = & \sum_{k\, \ge\, 1} k \mu \left( \left\{ x \in X\ :\ k \le \frac{\left| \log \left\| \left( L_{21} \left( x \right) \right) \right\|_{{\rm op}} \right|}{\epsilon} < k + 1 \right\} \right) \\ 
& < & \infty, 
\end{eqnarray*}
since $\displaystyle{\log \left\| \left( L \left( x \right) \right) \right\|_{{\rm op}}}$ is $\mu$-integrable, and thus, $\displaystyle{\log \left\| \left( L_{21} \left( x \right) \right) \right\|_{{\rm op}}}$ is $\mu$-integrable. 

Then, by an application of the Borel-Cantelli lemma, as one may find in \cite{parthasarathy:1977}, we have $\displaystyle{\mu \left( \limsup_{n\, \to\, \infty} \mathcal{A}_{n} \right) = 0}$ and hence, 
\[ \lim_{n\, \to\, \infty} \frac{1}{n} \log \left\| \left( L_{21} \left( T_{\pi_{n} \left( \omega \right)} x \right) \right) \right\|_{{\rm op}}\ \ =\ \ 0\ \ \text{for}\ \left( \mathfrak{B} \times \mu \right)\text{-almost every}\ \left( \omega,\, x \right) \in \Sigma_{N}^{+} \times X, \] 
proving the existence of the threshold $N_{3}$ for any given $\epsilon > 0$, as required. 
\end{enumerate} 
\end{proof} 

\section{Proof of Part $C$ of Theorem \ref{thm:oselomega}} 
\label{partcomega} 

Having gathered sufficient tools, we now proceed to prove part $C$ of Theorem \ref{thm:oselomega} in this section. 

\begin{proof}\big(of part $C$ of Theorem \ref{thm:oselomega}\big) 
We start with the base case, namely $k = 1$, by considering the vector subspace $V^{1}_{\left( \omega,\, x \right)} \subset \mathbb{R}^{d}$. We know from parts $A$ and $B$ of Theorem \ref{thm:oselomega} that the map $\left( \omega,\, x \right) \longmapsto V^{1}_{\left( \omega,\, x \right)}$ is measurable and that $V^{1}_{\left( \omega,\, x \right)}$ is an $\sigma$-invariant sub-bundle. Thus, 
\begin{eqnarray*} 
\lim_{n\, \to\, \infty} \frac{1}{n} \log \left\| \prod\limits_{j\, =\, 0}^{n - 1} \left. \left( L \left( T_{\pi_{j} \left( \omega \right)} x \right) \right)^{-1} \right|_{V^{1}_{\left( \omega,\, x \right)}} \right\|_{{\rm op}}^{-1} & = &  \min \left\{ \lambda \left( \omega,\, x,\, v \right)\ :\ v \in V^{1}_{\left( \omega,\, x \right)} \setminus V^{0}_{\left( \omega,\, x \right)} \right\} \\ 
& & \hspace{+4.4cm} \text{(from Proposition \ref{prop:usein10})} \\ 
& = & \lambda_{1} \left( \omega,\, x \right) \hspace{+2.8cm} \text{(from Equation \eqref{eqn:lam})} \\ 
& = & \max \left\{ \lambda \left( \omega,\, x,\, v \right)\ :\ v \in V^{1}_{\left( \omega,\, x \right)} \setminus V^{0}_{\left( \omega,\, x \right)} \right\}. \\ 
& & \hspace{+4.5cm} \text{(from Equation \eqref{eqn:lam})}
\end{eqnarray*} 
An application of Lemma \ref{lambdapmbound} then yields 
\[ \lim_{n\, \to\, \infty} \frac{1}{n} \log \left\| \prod\limits_{j\, =\, 0}^{n - 1} L \left( T_{\pi_{j} \left( \omega \right)} x \right) v \right\|\ \ =\ \ \lambda_{1} \left( \omega,\, x \right)\ \ \ \ \forall v \in V^{1}_{\left( \omega,\, x \right)} \setminus V^{0}_{\left( \omega,\, x \right)}. \] 

We now write the proof for the remainder of the cases, namely $2 \le k \le r$. Define $\left( \mathfrak{B} \times \mu \right)$-measurable functions $\alpha^{(k)}$ and $\beta^{(k)}$ on $\Sigma_{N}^{+} \times X$ given by $\alpha^{(k)} \left( \omega,\, x \right) \equiv \lambda_{k - 1} \left( \omega,\, x \right)$ and $\beta^{(k)} \left( \omega,\, x \right) \equiv \lambda_{k} \left( \omega,\, x \right)$. Then, the functions $\alpha^{(k)}$ and $\beta^{(k)}$ satisfy the conditions mentioned in the proof of part $1$ of Lemma \ref{lem:twoparts}, namely 
\begin{enumerate} 
\item[(a)] $\alpha^{(k)}$ and $\beta^{(k)}$ are $\left( \mathfrak{B} \times \mu \right)$-integrable functions on $\Sigma_{N}^{+} \times X$, 
\item[(b)] $\alpha^{(k)} < \beta^{(k)}$ on $\Sigma_{N}^{+} \times X$, 
\item[(c)] $\alpha^{(k)} \left( \sigma^{j} \omega,\, T_{\pi_{j} \left( \omega \right)} x \right) = \alpha^{(k)} \left( \omega,\, x \right)$ and $\beta^{(k)} \left( \sigma^{j} \omega,\, T_{\pi_{j} \left( \omega \right)} x \right) = \beta^{(k)} \left( \omega,\, x \right)$ for all $j \in \mathbb{Z}_{+}$ and 
\item[(d)] $\lambda \left( \omega,\, x,\, v \right) \le \alpha^{(k)} \left( \omega,\, x \right)$ for every $v \in V^{k - 1}_{\left( \omega,\, x \right)} \setminus V^{0}_{\left( \omega,\, x \right)}$ and $\lambda \left( \omega,\, x,\, u \right) \ge \beta^{(k)} \left( \omega,\, x \right)$ for every $u \in \mathbb{R}^{d} \setminus V^{k - 1}_{\left( \omega,\, x \right)}$. 
\end{enumerate} 
Suppose $\dim \left( V^{k - 1}_{\left( \omega,\, x \right)} \right) = d_{k - 1}$, then one may consider $\left( V^{k - 1}_{\left( \omega,\, x \right)} \right)^{\perp}$ to be isometrically isomorphic to $\mathbb{R}^{d - d_{k - 1}}$. Thus, the decomposition of $L(x)$, as one may infer from Equation \eqref{eqn:decompL}, looks like 
\[ L (x)\ \ =\ \ \begin{pmatrix} L_{11}^{(k - 1)} (x) & 0 \\ L_{21}^{(k - 1)} (x) & L_{22}^{(k - 1)} (x) \end{pmatrix},\ \ \ \text{where}\ \ L_{11}^{(k - 1)} (x) \in {\rm GL}_{d - d_{k - 1}} \left( \mathbb{R} \right). \] 
Now, define vector subspaces $W^{k - 1}_{\left( \omega,\, x \right)} = V^{k - 1}_{\left( \omega,\, x \right)} \cap \left( V^{k - 1}_{\left( \omega,\, x \right)} \right)^{\perp} = V^{0}_{\left( \omega,\, x \right)}$ and $W^{k}_{\left( \omega,\, x \right)} = V^{k}_{\left( \omega,\, x \right)} \cap \left( V^{k - 1}_{\left( \omega,\, x \right)} \right)^{\perp}$. Then, one can see that $W^{k}_{\left( \omega,\, x \right)} \setminus W^{k - 1}_{\left( \omega,\, x \right)} = \left( V^{k}_{\left( \omega,\, x \right)} \setminus V^{k - 1}_{\left( \omega,\, x \right)} \right) \cap \left( V^{k - 1}_{\left( \omega,\, x \right)} \right)^{\perp}$. Since the maps $\left( \omega,\, x \right) \longmapsto V^{i}_{\left( \omega,\, x \right)}$ is measurable with $V^{i}_{\left( \omega,\, x \right)}$ being $\sigma$-invariant sub-bundles for $ i = k - 1$ and $k$, we also have the map $\left( \omega,\, x \right) \longmapsto W^{k}_{\left( \omega,\, x \right)}$ to be measurable with $W^{k}_{\left( \omega,\, x \right)}$ being an $\sigma$-invariant sub-bundle. Hence, using the idea employed in the base case, however, for the vector space $W^{k}_{\left( \omega,\, x \right)}$, instead of $V^{1}_{\left( \omega,\, x \right)}$ and for the map $L_{11}^{(k - 1)} (x)$ instead of $L(x)$, we obtain 
\begin{eqnarray*} 
\lim_{n\, \to\, \infty} \frac{1}{n} \log \left\| \prod\limits_{j\, =\, 0}^{n - 1} \left. \left( L_{11}^{(k - 1)} \left( T_{\pi_{j} \left( \omega \right)} x \right) \right)^{-1} \right|_{W^{k}_{\left( \omega,\, x \right)}} \right\|_{{\rm op}}^{-1} & = &  \min \left\{ \lambda \left( \omega,\, x,\, u \right)\ :\ u \in W^{k}_{\left( \omega,\, x \right)} \setminus W^{k - 1}_{\left( \omega,\, x \right)} \right\} \\ 
& & \hspace{+3.1cm} \text{(from Proposition \ref{prop:usein10})} \\ 
& = & \lambda_{k} \left( \omega,\, x \right) \hspace{+1.6cm} \text{(from Equation \eqref{eqn:lam})} \\ 
& = & \max \left\{ \lambda \left( \omega,\, x,\, u \right)\ :\ u \in W^{k}_{\left( \omega,\, x \right)} \setminus W^{k - 1}_{\left( \omega,\, x \right)} \right\} \\ 
& & \hspace{+3.2cm} \text{(from Equation \eqref{eqn:lam})}. 
\end{eqnarray*} 
As earlier, an application of Lemma \ref{lambdapmbound} then yields 
\[ \lim_{n\, \to\, \infty} \frac{1}{n} \log \left\| \prod\limits_{j\, =\, 0}^{n - 1} L_{11}^{(k - 1)} \left( T_{\pi_{j} \left( \omega \right)} x \right) u \right\|\ \ =\ \ \lambda_{k} \left( \omega,\, x \right)\ \ \ \ \forall u \in W^{k}_{\left( \omega,\, x \right)} \setminus W^{k - 1}_{\left( \omega,\, x \right)}. \] 
Since every vector $w \in V^{k}_{\left( \omega,\, x \right)} \setminus V^{k - 1}_{\left( \omega,\, x \right)}$ can be expressed as $w = u + v$ where $u \in W^{k}_{\left( \omega,\, x \right)} \setminus W^{k - 1}_{\left( \omega,\, x \right)} \subseteq \left( V^{k - 1}_{\left( \omega,\, x \right)} \right)^{\perp}$ and $v \in V^{k - 1}_{\left( \omega,\, x \right)}$, we make use of Lemma \ref{lem:twoparts}, to write 
\[ \lim_{n\, \to\, \infty} \frac{1}{n} \log \left\| \prod\limits_{j\, =\, 0}^{n - 1} L \left( T_{\pi_{j} \left( \omega \right)} x \right) w \right\|\ \ =\ \ \lim_{n\, \to\, \infty} \frac{1}{n} \log \left\| \prod\limits_{j\, =\, 0}^{n - 1} L_{11}^{(k - 1)} \left( T_{\pi_{j} \left( \omega \right)} x \right) u \right\|\ \ =\ \ \lambda_{k} \left( \omega,\, x \right). \] 
Thus, the proof of part $C$ of Theorem \ref{thm:oselomega} is complete. 
\end{proof} 

\section{Proof of Theorem \ref{thm:oselnoomega}: Parts $A$ and $B$} 
\label{sec:potnoomegaab} 

The ideas to prove Theorem \ref{thm:oselnoomega} follow the lines of the proof of Theorem \ref{thm:oselomega} and can appear like a repetition of sorts from Section \ref{sec:afsl} onwards, for a layman. However, upon closer inspection, one may observe that in the arguments presented henceforth, the quantities defined and results obtained pertain to the evolution of orbits in the dynamical system by simultaneously employing the considered collection of maps $\left\{ T_{1}, T_{2}, \cdots, T_{N} \right\}$ and are therefore, independent of $\omega \in \Sigma_{N}^{+}$. Furthermore, the evolution of every orbit is set-valued here, as explained in Equation \eqref{actionofR}. 

We now recall the definition of the quantity $\lambda \left( \omega,\, x,\, v \right)$, as defined in Equation \eqref{eqn:loxv}, for some fixed $\omega \in \Sigma_{N}^{+}$, for $\mu$-almost every $x \in X$ and $v \in \mathbb{R}^{d}$. 
\[ \lambda \left( \omega,\, x,\, v \right)\ \ =\ \ \limsup_{n\, \to\, \infty} \frac{1}{n} \log \left\| \prod_{j\, =\, 0}^{n - 1} \left( L \circ T_{\pi_{j} \left( \omega \right)} x \right) v \right\|. \] 
As earlier, let $\omega^{n} = \pi_{n} \left( \omega \right) \in \Sigma_{N}^{n}$ satisfy the condition that the sequence of infinite concatenation of $\omega^{n}$ with itself denoted by $\displaystyle{\left\{ \overline{\omega^{n}} \right\}_{n\, \in\, \mathbb{Z}_{+}}}$ converges to the considered $\omega$ in $\Sigma_{N}^{+}$. Then, it is trivial to observe that $\displaystyle{\frac{1}{n} \log \left\| \prod_{j\, =\, 0}^{n - 1} \left( L \circ T_{\pi_{j} \left( \omega \right)} x \right) v \right\| = \frac{1}{n} \log \left\| \prod_{j\, =\, 0}^{n - 1} \left( L \circ T_{\pi_{j} \left( \omega^{n} \right)} x \right) v \right\|}$. Thus, defining 
\[ \lambda_{n} \left( \omega^{n},\, x,\, v \right)\ \ =\ \ \frac{1}{n} \log \left\| \prod_{j\, =\, 0}^{n - 1} \left( L \circ T_{\pi_{j} \left( \omega^{n} \right)} x \right) v \right\|, \] 
we note that $\lambda \left( \omega,\, x,\, v \right) = \limsup\limits_{n\, \to\, \infty} \lambda_{n} \left( \omega^{n},\, x,\, v \right)$. 

For any $E \subsetneq X$ with $\# E < \infty$ and for any $F \subsetneq \mathbb{R}^{d}$ with $\# F < \infty$, we define a quantity called $\Lambda \left( E,\, F \right)$, given by 
\begin{equation} 
\label{eqn:Lxv} 
\Lambda \left( E,\, F \right)\ \ =\ \ \begin{cases} \limsup\limits_{n\, \to\, \infty} \dfrac{1}{N^{n}} \dfrac{1}{\# E} \mathlarger{\sum}\limits_{\omega^{n}\, \in\, \Sigma_{N}^{n}}\ \mathlarger{\sum}\limits_{x\, \in\, E}\ \lambda_{n} \left( \omega^{n},\, x,\, \left( \sum\limits_{v\, \in\, F} v \right) \right) & \text{when}\ \sum\limits_{v\, \in\, F} v \ne 0 \\ 
- \infty & \text{otherwise}. \end{cases} 
\end{equation} 

Suppose $F_{1}$ and $F_{2}$ are finite subsets of $\mathbb{R}^{d}_{*}$, we define the set 
\[ F_{1} + F_{2}\ \ =\ \ \left\{ v_{1} + v_{2} : v_{1} \in F_{1},\ v_{2} \in F_{2} \right\}\ \ \ \ \text{and}\ \ \ \ c F_{1}\ \ =\ \ \left\{ c v_{1} : v_{1} \in F_{1} \right\}\ \ \ \ \text{for some}\ c \in \mathbb{R}. \] 
Then, one can easily verify that 
\begin{eqnarray} 
\label{Lamdav1v2} 
\Lambda \left( E,\, F_{1} + F_{2} \right) & \le & \max \left\{ \Lambda \left( E,\, F_{1} \right),\; \Lambda \left( E,\, F_{2} \right) \right\}\ \ \ \ \text{and} \nonumber \\ 
\Lambda \left( E,\, c F_{1} \right) & = & \Lambda \left( E,\, F_{1} \right)\ \text{for every}\ c \ne 0 \in \mathbb{R}. 
\end{eqnarray} 

Since for any vector $v \in \mathbb{R}^{d}$, we have 
\[ \| v \|\ \ \le\ \ \prod_{\omega^{n}\, \in\, \Sigma_{N}^{n}} \left[ \left\| \left( \prod_{j\, =\, 0}^{n - 1} \left( L \left( T_{\pi_{j} \left( \omega^{n} \right)} x \right) \right) \right)^{-1} \right\|_{{\rm op}} \left\| \prod_{j\, =\, 0}^{n - 1} L \left( T_{\pi_{j} \left( \omega^{n} \right)} x \right) \right\|_{{\rm op}} \right] \| v \|, \]
one obtains $\Lambda_{-} (x) \le \Lambda \left( E,\, F \right) \le \Lambda_{+} (x)$, where $E = \left\{ x \right\}$ for some $\mu$-typical point $x \in X$ and $F = \left\{ v \right\} \subsetneq \mathbb{R}^{d}_{*}$. Here $\Lambda_{-} (x)$ and $\Lambda_{+} (x)$ are the extremal Lyapunov exponents, as defined in Theorem \ref{thm:fkomegafree}. 

For some $\mu$-typical point $x \in X$, fix $E_{0} = \{ x \}$ and define a function $g_{0} : \mathbb{R}^{d} \longrightarrow \mathbb{R} \cup \{ - \infty \}$ given by $g_{0} (v) = \Lambda \left( E_{0},\, F_{-1} \right)$, where $F_{-1} = \{ v \}$. Then, one may note that $g_{0}$ satisfies the hypothesis of Lemma \ref{fonrd} for $\mu$-almost every $x \in X$ and hence, $g_{0}$ attains at most $d$ distinct values in $\mathbb{R}$. Suppose the cardinality of the set $U_{E_{0}} = \left\{ \Lambda \left( E_{0},\, F_{-1} \right)\ :\ F_{-1} = \{ v \} \subsetneq \mathbb{R}^{d}_{*} \right\}$ is $r = r \left( E_{0} \right)$, we arrange the elements in $U_{E_{0}}$ in increasing order as $\Lambda_{1} \left( E_{0} \right) < \Lambda_{2} \left( E_{0} \right) < \cdots < \Lambda_{r} \left( E_{0} \right)$ and define for every $1 \le k \le r$, the set 
\begin{equation} 
\label{nestedvs1} 
V_{E_{0}}^{k}\ \ =\ \ \left\{ v \in \mathbb{R}^{d}_{*}\ :\ \Lambda \left( E_{0},\, \{ v \} \right)\ \le\ \Lambda_{k} \left( E_{0} \right) \right\}\ \cup\ \{ 0 \}. 
\end{equation} 
Then, owing to the two Equations in \eqref{Lamdav1v2}, we infer that $V_{E_{0}}^{k}$ is a nested collection of vector subspaces of $\mathbb{R}^{d}$. Hence, defining $V_{E_{0}}^{0}$ to be the trivial subspace of $\mathbb{R}^{d}$, we have 
\[ \left\{ 0 \right\}\ \ =\ \ V_{E_{0}}^{0}\ \ \subsetneq\ \ V_{E_{0}}^{1}\ \ \subsetneq\ \ V_{E_{0}}^{2}\ \ \subsetneq\ \ \cdots\ \ \subsetneq\ \ V_{E_{0}}^{k}\ \ \equiv\ \ \mathbb{R}^{d}. \]
Then, our definition of the collection of vector subspaces $\left\{ V^{k}_{E_{0}} \right\}$ for $0 \le k \le r$ and our ordering of the numbers $\left\{ \Lambda_{1} \left( E_{0} \right), \cdots, \Lambda_{r} \left( E_{0} \right) \right\}$ ensures that 
\begin{equation} 
\label{eqn:Lam} 
\Lambda \left( E_{0},\, \{ v \} \right)\ \ =\ \ \Lambda_{k} \left( E_{0} \right)\ \ \ \ \forall v \in V^{k}_{E_{0}} \setminus V^{k -1}_{E_{0}}, 
\end{equation} 
as otherwise, it would result in an increase in the cardinality of $U_{E_{0}}$. 

Given some $\mu$-typical point $x \in X$, fix $E_{0} = \{ x \}$ and let $\left\{ E_{n} \right\}_{n\, \in\, \mathbb{Z}_{+}}$ be the sequence of subsets of $X$, as given in Theorem \ref{thm:oselnoomega}, defined by $\displaystyle{E_{n} = \bigcup\limits_{\omega^{n}\, \in\, \Sigma_{N}^{n}} T_{\omega^{n}} x}$. Note that every set $E_{n}$ in the sequence can be rewritten as $\displaystyle{E_{n} = \bigcup\limits_{\omega_{n}\, =\, 1}^{N} \bigcup_{y\, \in\, E_{n - 1}} T_{\omega_{n}} y}$. Pertaining to this $E_{0}$, given some vector $v \in \mathbb{R}^{d}$, consider $F_{-1} = \{ v \}$ and define a sequence of subsets of $\mathbb{R}^{d}$ denoted by $\left\{ F_{n} \right\}_{n\, \in\, \mathbb{Z}_{+}\; \cup\; \{0 \}}$ given by $\displaystyle{F_{n} = \bigcup\limits_{y\, \in\, E_{n}} \bigcup\limits_{w\, \in\, F_{n - 1}} L(y) w}$, where the vectors are repeated in each of the sets $F_{n}$, if necessary. Since the sets $E_{0}$ and $F_{-1}$ are of finite cardinality, we note that every set in the sequence $\left\{ E_{n} \right\}_{n\, \in\, \mathbb{Z}_{+}}$ and $\left\{ F_{n} \right\}_{n\, \in\, \mathbb{Z}_{+}\; \cup\; \{ 0 \}}$ are also of finite cardinality. We fix this notation for the sequence of subsets $\left\{ E_{n} \right\}$ of $X$ and $\left\{ F_{n} \right\}$ of $\mathbb{R}^{d}$ constructed from $E_{0} = \{ x \}$ and $F_{-1} = \{ v \}$, for the remainder of this manuscript. With this notation of sets, one can see that 
\begin{equation}
\label{Lambdaenfn}
\Lambda \left( E_{n},\, F_{n - 1} \right)\ \ \equiv\ \ \Lambda \left( E_{0},\, F_{-1} \right)\ \ \ \ \text{for every}\ \ n \in \mathbb{Z}_{+}. 
\end{equation} 

We now define a function $g_{n} : \mathbb{R}^{d} \longrightarrow \mathbb{R} \cup \{ - \infty \}$ given by $g_{n} (v) = \Lambda \left( E_{n},\, F_{n - 1} \right)$, where $F_{n - 1}$ is obtained by taking $F_{-1} = \{ v \}$. Then, it is true that $g_{n}$ satisfies the hypothesis of Lemma \ref{fonrd} for $\mu$-almost every $x \in X$ with $E_{0} = \{ x \}$ and hence, $g_{n}$ attains at most $d$ distinct values in $\mathbb{R}$. Thus, the set $\displaystyle{U_{E_{n}} = \left\{ \Lambda \left( E_{n},\, F_{n - 1} \right) \right\} = U_{E_{0}}}$, by Equation \eqref{Lambdaenfn}. Arranging the members of $U_{E_{n}}$ in ascending order, we have 
\begin{equation} 
\label{LamdakEnandrEn} 
\Lambda_{k} \left( E_{n} \right)\ \ \equiv\ \ \Lambda_{k} \left( E_{0} \right)\ \ \ \ \text{for every}\ \ n \in \mathbb{Z}_{+}\ \ \ \ \ \text{and hence,}\ \ \ \ \ r \left( E_{n} \right)\ \ \equiv\ \ r \left( E_{0} \right)\ \ \ \ \forall n \in \mathbb{Z}_{+}, 
\end{equation} 
where $r \left( E_{n} \right)$ denotes the cardinality of the set $U_{E_{n}}$. 

Using Equations \eqref{Lambdaenfn} and \eqref{LamdakEnandrEn} in the expression of the set $\displaystyle{\sum_{\omega^{l}\, \in\, \Sigma_{N}^{l}} L \left( T_{\omega^{l}} x \right) V_{E_{l}}^{k}}$ and noting that $\displaystyle{\sum\limits_{w\, \in\, F_{l}} w = \sum_{\omega^{l}\, \in\, \Sigma_{N}^{l}} L \left( T_{\omega^{l}} x \right) u}$, where $u \in F_{l - 1}$, we obtain 
\[ \sum_{\omega^{l}\, \in\, \Sigma_{N}^{l}} L \left( T_{\omega^{l}} x \right) V_{E_{l}}^{k}\ \ =\ \ V_{E_{l + 1}}^{k}, \] 
thereby completing the proof of all statements in part $A$ of Theorem \ref{thm:oselnoomega}. 

In order to prove part $B$ of Theorem \ref{thm:oselnoomega}, we now state a theorem analogous to Theorem \ref{equivalent:3} by getting rid of $\omega$ and the space $\Sigma_{N}^{+}$. Alternatively, this can be obtained as a direct consequence of Theorem \ref{equivalent}, as one may find in \cite{pw:1993}. 

\begin{theorem} 
\label{equivalent:4} 
Let $X$ be a complete probability space. Let ${\rm Gr}(d)$ denote the Grassmannian on $\mathbb{R}^{d}$. Let $\mathcal{H} : X \longrightarrow {\rm Gr}(d)$. Then, the following statements are equivalent. 
\begin{enumerate} 
\item The map $\mathcal{H}$ given by $\mathcal{H} (x) = V_{E_{0}}$ where $E_{0} = \{ x \}$ is measurable. 
\item The set $\Gamma_{\mathcal{H}} = \left\{ \left( x,\, v \right) \in X \times \mathbb{R}^{d} : v \in V_{E_{0}}\ \text{where}\ E_{0} = \{ x \} \right\}$ belongs to the product sigma-algebra $\mathcal{B}_{X} \times \mathcal{B}_{\mathbb{R}^{d}}$ of the product space $X \times \mathbb{R}^{d}$. 
\item For every $1 \le l \le d$, the set $\mathcal{X}_{l} = \left\{ x \in X : \dim \left(V_{\{ x \}}\right) = l \right\}$ belongs to $\mathcal{B}_{X}$. Moreover, there exists measurable vector fields $u_{i} : \mathcal{X}_{l} \longrightarrow \mathbb{R}^{d}$ for every $1 \le i \le l$ such that the collection $\left\{ u_{1} (x), u_{2} (x), \cdots, u_{l} (x) \right\}$ is a basis of the vector space $\mathcal{H} (x)$ for every $x \in \mathcal{X}_{l}$. 
\end{enumerate} 
\end{theorem} 

We now prove that the maps $x \longmapsto r \left( E_{0} \right),\ x \longmapsto \Lambda_{k} \left( E_{0} \right)$ and $x \longmapsto V_{E_{0}}^{k}$ are measurable for every $1 \le k \le r$ where $E_{0} = \left\{ x \right\}$, by induction. Consider the case $k = r$. Then, for $- \infty < a \le 1$ we have $r^{-1} \left( [a,\, \infty) \right) = X$. Moreover, for $\mu$-almost every $x \in X$, we know from Equation \eqref{nestedvs1} that $V_{E_{0}}^{r} = \mathbb{R}^{d}$. Suppose $\mathcal{X}_{l} = \left\{ x \in X : \dim \left( V_{E_{0}}^{r} \right) = l \right\}$, then, $\mathcal{X}_{l} = \emptyset$ for all $1 \le l < d$. Let $\left\{ e_{1}, e_{2}, \cdots, e_{d} \right\}$ be a basis of $\mathbb{R}^{d}$. Define a vector field $v_{i} : \mathcal{X}_{d} \longrightarrow \mathbb{R}^{d}$ as $v_{i} (x) = e_{i}$. Then, by Theorem \ref{equivalent:4}, the map $x \longrightarrow V_{E_{0}}^{r}$ is measurable. 

Suppose that the basis vectors of $\mathbb{R}^{d}$ are written in such an order that the collection of vectors $\left\{ e_{1},\, e_{2}, \cdots, e_{k} \right\}$ spans the vector subspace $V_{E_{0}}^{k}$. Then, for any $v \in V_{E_{0}}^{k} \setminus V_{E_{0}}^{k - 1}$, we have $\Lambda \left( E_{0},\, \left\{ v \right\} \right) = \Lambda_{k} \left( E_{0} \right)$. This implies $\displaystyle{\max \left\{ \Lambda \left( E_{0},\, \left\{ e_{i} \right\} \right) : 1 \le i \le k \right\} = \Lambda_{k} \left( E_{0} \right)}$ for every $1 \le k \le r$. Thus, the map $x \longmapsto \Lambda_{k} \left( E_{0} \right)$ is measurable for every $1 \le k \le r$, where $E_{0} = \left\{ x \right\}$. 

Further, the set $\left\{ (x,\, v) \in X \times \mathbb{R}^{d} : \Lambda \left( \left\{ x \right\},\, \left\{ v \right\} \right) < \Lambda_{k} \left( \left\{ x \right\} \right) \right\}$ is a measurable subset of $X \times \mathbb{R}^{d}$. Then, using Lemma \ref{lem:projmeas}, we have that \[ \big\{ x \in X : \Lambda \left( \left\{ x \right\},\, \left\{ v \right\} \right) < \Lambda_{r} \left( \left\{ x \right\} \right) \big\} = \big\{ x \in X : r \left( \left\{ x \right\} \right) \ge 2 \big\} \] 
is a measurable subset of $X$. Moreover, 
\begin{eqnarray*} 
& & \left\{ (x,\, v) \in X \times \mathbb{R}^{d}\ :\ v \in V_{\left\{ x \right\}}^{r - 1} \right\} \\ 
& = & \left\{ (x,\, v) \in X \times \mathbb{R}^{d}\ :\ \Lambda \left( \left\{ x \right\},\, \left\{ v \right\} \right) \le \Lambda_{r - 1} \left( \left\{ x \right\} \right) \right\}\ \cup\ \left( X \times \{ 0 \} \right) 
\end{eqnarray*} 
is a measurable subset of $X \times \mathbb{R}^{d}$. Thus, applying Lemma \ref{equivalent:4}, we obtain the map $x \longmapsto V_{E_{0}}^{r - 1}$ is measurable, where $E_{0} = \left\{ x \right\}$. The proof of part $B$ of Theorem \ref{thm:oselnoomega} is then concluded by appealing to an inductive argument on $r$, that decreases over the set of positive integers. 

\section{Proof of Part $C$ of Theorem \ref{thm:oselnoomega}} 
\label{sec:potnoomegac} 

Since the ideas of the proof of part $C$ of Theorem \ref{thm:oselnoomega} follows the same as the proof of part $C$ of Theorem \ref{thm:oselomega}, we compress the arguments made in Sections \ref{markov}, \ref{invariantmsb}, \ref{action} and \ref{partcomega} as necessary and complete the proof in this section. 

Denote by $\mathcal{M}' \left( X \times \mathscr{K} \right)$ the space of all real-valued measurable functions such that for $\mu$-almost every $x \in X$, we have $\Phi \left( x, \cdot \right) \in \mathcal{C} (\mathscr{K})$ and the map $x \longmapsto \left\| \Phi \left( x, \cdot \right) \right\|_{\infty}$ is $\mu$-integrable, where we recall from Section \ref{markov} that $\mathscr{K}$ is a compact metric space. Define a norm $\| \cdot \|_{\mathcal{M}'}$ on $\mathcal{M}' \left( X \times \mathscr{K} \right)$ given by $\displaystyle{ \left\| \Phi \right\|_{\mathcal{M'}} = \int \left\| \Phi \left( x, \cdot \right) \right\|_{\infty} \mathrm{d} \mu}$. Then, one can verify that $\mathcal{M}' \left( X \times \mathscr{K} \right)$ is complete with respect to $\left\| \cdot \right\|_{\mathcal{M}'}$. Let $\mathbb{M}' \left( \mu \right)$ be the set of all probability measures $\nu'$ supported on $X \times \mathscr{K}$ such that ${\rm proj}_{*} \nu' = \nu' \circ {\rm proj}^{-1} = \mu$, where ${\rm proj} : X \times \mathscr{K} \longrightarrow X$ is the canonical projection map. Then, by the Banach-Alaoglu theorem, we obtain $\mathbb{M}' \left( \mu \right)$ to be a compact metrizable space.

\begin{proposition}
\label{Ginvariantmeasure} 
Consider a simultaneous dynamical system $R$ defined using finitely many maps $T_{1}, T_{2}, \cdots, T_{N}$ on subsets of $X \times \mathscr{K}$ given by $\displaystyle{R \left( E,\, F \right) = \bigcup_{\omega_{1}\, =\, 1}^{N} \bigcup_{x\, \in\, E} \bigcup_{v\, \in\, F} \big\{ \left( T_{\omega_{1}} x,\, \mathcal{R}_{x} v \right) \big\}}$, where $\mathcal{R}_{x} : \mathcal{C} \left( \mathscr{K} \right) \longrightarrow \mathcal{C} \left( \mathscr{K} \right)$ is a continuous function for $\mu$-almost every $x \in X$. Fix a $\mu$-typical point $x\in X$ and consider $E_{0} = \{ x \}$. Let $\{ E_{n} \}$ be a sequence of subsets of $X$, as defined in Theorem \ref{thm:oselnoomega}. Then, given any $\Phi \in \mathcal{M}' \left( X \times \mathscr{K} \right)$, the real sequences 
\[ \left\{ \frac{1}{n} \frac{1}{N^{n}} \left( \Phi_{n} \right)_{*} (x) \right\}_{n\, \in\, \mathbb{Z}_{+}}\ \ \ \ \text{and}\ \ \ \ \left\{ \frac{1}{n} \frac{1}{N^{n}} \left( \Phi_{n} \right)^{*} (x) \right\}_{n\, \in\, \mathbb{Z}_{+}} \] 
converge, where 
\begin{eqnarray} 
\label{Phinl*} 
\left( \Phi_{n} \right)_{*} (x) & = & \inf_{ v\, \in\, \mathcal{K}}\ \sum_{\omega^{n}\, \in\, \Sigma_{N}^{n}}\ \sum_{j\, =\, 0}^{n - 1}\ \Phi \left( T_{\pi_{j} \left( \omega^{n} \right)} x,\, \mathcal{R}_{T_{\pi_{j - 1} \left( \omega^{n} \right)} x} \circ \cdots \circ \mathcal{R}_{T_{\pi_{0} \left( \omega^{n} \right)} x} v \right)\ \ \text{and} \nonumber \\ 
& & \\ 
\label{Phinu*} 
\left( \Phi_{n} \right)^{*} (x) & = & \sup_{ v\, \in\, \mathcal{K}}\ \sum_{\omega^{n}\, \in\, \Sigma_{N}^{n}}\ \sum_{j\, =\, 0}^{n - 1}\ \Phi \left( T_{\pi_{j} \left( \omega^{n} \right)} x,\, \mathcal{R}_{T_{\pi_{j - 1} \left( \omega^{n} \right)} x} \circ \cdots \circ \mathcal{R}_{T_{\pi_{0} \left( \omega^{n} \right)} x} v \right). \nonumber \\ 
& & 
\end{eqnarray} 
Moreover, there exist probability measures $(\nu')_{*}$ and $(\nu')^{*}$ supported on $X \times \mathscr{K}$ such that 
\begin{eqnarray}
\label{Phimeasureu}
\int \Phi \mathrm{d} (\nu')_{*} & = & \int \Phi_{*} \mathrm{d} \mu\ \ =\ \ \int \lim_{n\, \to\, \infty} \frac{1}{n} \frac{1}{N^{n}} \left( \Phi_{n} \right)_{*} \mathrm{d}\mu \\ 
\label{Phimeasurel}
\int \Phi \mathrm{d} (\nu')^{*} & = & \int \Phi^{*} \mathrm{d} \mu\ \ =\ \ \int \lim_{n\, \to\, \infty} \frac{1}{n} \frac{1}{N^{n}} \left( \Phi_{n} \right)^{*} \mathrm{d}\mu. 
\end{eqnarray}
\end{proposition}

Since the ideas to prove Proposition \ref{Ginvariantmeasure} are similar to that of the proof of Proposition \ref{nustar}, we shall only state the important steps in the proof. 

\begin{proof} 
We start by observing that 
\begin{eqnarray*} 
& & \left( \Phi_{n + p} \right)^{*} (x) \nonumber \\ 
& = & \sup_{v\, \in\, \mathcal{K}}\ \sum_{\omega^{n + p}\, \in\, \Sigma_{N}^{n + p}} \sum_{j\, =\, 0}^{n + p - 1}\ \Phi \left( T_{\pi_{j} \left( \omega^{n + p} \right)} x,\, \mathcal{R}_{T_{\pi_{j - 1} \left( \omega^{n + p} \right)} x} \circ \cdots \circ \mathcal{R}_{T_{\pi_{0} \left( \omega^{n + p} \right)} x} v \right) \nonumber \\ 
& \le & \sup_{v\, \in\, \mathcal{K}}\ \sum_{\omega^{n + p}\, \in\, \Sigma_{N}^{n + p}}\ \sum_{j\, =\, 0}^{n - 1}\ \Phi \left( T_{\pi_{j} \left( \omega^{n} \right)}x,\, \mathcal{R}_{T_{\pi_{j - 1} \left( \omega^{n + p} \right)} x} \circ \cdots \circ \mathcal{R}_{T_{\pi_{0} \left( \omega^{n + p} \right)} x} v \right) \nonumber \\ 
& & \hspace{+2cm} +\ \sup_{v\, \in\, \mathcal{K}}\ \sum_{\omega^{n + p}\, \in\, \Sigma_{N}^{n + p}}\ \sum_{j\, =\, n}^{n + p - 1}\ \Phi \left( T_{\pi_{j} \left( \omega^{n + p} \right)} x, \mathcal{R}_{T_{\pi_{j - 1} \left( \omega^{n + p} \right)} x} \circ \cdots \circ \mathcal{R}_{T_{\pi_{0} \left( \omega^{n + p} \right)} x} v \right) \nonumber \\ 
\end{eqnarray*} 

\begin{eqnarray} 
\label{Nsubadd} 
& \le & \sup_{v\, \in\, \mathcal{K}}\ \sum_{\omega^{n}\, \in\, \Sigma_{N}^{n}} \sum_{j\, =\, 0}^{n - 1}\ \Phi \left( T_{\pi_{j} \left( \omega^{n} \right)} x,\, \mathcal{R}_{T_{\pi_{j-1} \left( \omega^{n} \right)} x} \circ \cdots \circ \mathcal{R}_{T_{\pi_{0} \left( \omega^{n} \right)} x} v \right) \nonumber \\ 
& & +\ \sum_{\omega^{n}\, \in\, \Sigma_{N}^{n}}\ \sup_{v\, \in\, \mathcal{K}}\ \sum_{\omega^{p}\, \in\, \Sigma_{N}^{p}}\ \sum_{j\, =\, 0}^{p - 1}\ \Phi \left( T_{\pi_{j} \left( \omega^{p} \right)} \left( T_{\omega^{n}} x \right),\, \mathcal{R}_{T_{\pi_{j - 1} \left( \omega^{p} \right)} \left( T_{\omega^{n}} x \right)} \circ \cdots \circ \mathcal{R}_{T_{\pi_{0} \left( \omega^{p} \right)} \left( T_{\omega^{n}} x \right)} v \right) \nonumber \\ 
& = & \left( \Phi_{n} \right)^{*} (x) + \sum_{\omega^{n}\, \in\, \Sigma_{N}^{n}} \left( \Phi_{p} \right)^{*} \left( T_{\omega^{n}} x \right). 
\end{eqnarray}  

We now state a theorem, as in \cite{ts:pp} due to the authors, that plays a vital role in proving the convergence of the real sequences mentioned in Proposition \ref{Ginvariantmeasure}. 

\begin{theorem}\cite{ts:pp} 
\label{kingmanomegafreel} 
For a sequence of functions, say $\left\{ (\Phi_{n})^{*} \right\}_{n\, \ge\, 1}$ that satisfies Equation \eqref{Nsubadd}, suppose the positive part of $(\Phi_{1})^{*}$, denoted by $\left[ \left( \Phi_{1} \right)^{*} \right]^{+} \in \mathscr{L}^{1} (\mu)$. Then, the sequence $\displaystyle{\left\{ \dfrac{1}{n} \dfrac{1}{N^{n}} \left( \Phi_{n} \right)^{*} \right\}_{n\, \ge\, 1}}$ converges to some measurable function $\Phi^{*} : X \longrightarrow [-\infty, \infty)\ $ for $\mu$-almost every $x \in X$. Moreover, the positive part of the limit function is integrable and 
\[ \int \Phi^{*} \mathrm{d}\mu\ \ =\ \ \lim_{n\, \to\, \infty} \frac{1}{n} \frac{1}{N^{n}} \int \left( \Phi_{n} \right)^{*} \mathrm{d}\mu\ \ =\ \ \inf_{n\, \ge\, 1} \frac{1}{n} \frac{1}{N^{n}} \int \left( \Phi_{n} \right)^{*} \mathrm{d}\mu\ \ \in\ \ [-\infty, \infty). \] 
\end{theorem} 

An application of Theorem \ref{kingmanomegafreel} to the sequence of functions $\left\{ (\Phi_{n})^{*} \right\}_{n\, \ge\, 1}$, as defined in Equation \eqref{Phinu*} then assures the existence of the limit function $\Phi^{*}$ for $\mu$-almost every $x \in X$. Similarly, one works with the sequence of functions $\left\{ (\Phi_{n})_{*} \right\}_{n\, \ge\, 1}$, as defined in Equation \eqref{Phinl*} to obtain the limit function $\Phi_{*}$ for $\mu$-almost every $x \in X$. 

Considering 
\begin{eqnarray*} 
\Delta_{n}' & = & \Bigg\{ \left( x,\, v \right) \in X \times \mathscr{K} : \left( \Phi_{n} \right)^{*} (x) \\ 
& & \hspace{+3cm} =\ \sum_{\omega^{n}\, \in\, \Sigma_{N}^{n}}\ \sum_{j\, =\, 0}^{n - 1}\ \Phi \left( T_{\pi_{j} \left( \omega^{n} \right)} x,\, \mathcal{R}_{T_{\pi_{j - 1} \left( \omega^{n} \right)} x} \circ \cdots \circ \mathcal{R}_{T_{\pi_{0} \left( \omega^{n} \right)} x} v \right) \Bigg\} \\ 
\Delta_{n}' (x) & = & \big\{ v \in \mathscr{K} : \left( x,\, v \right) \in \Delta_{n}' \big\}\ \ \ \text{for each}\ x \in X, 
\end{eqnarray*} 
we get a measurable map from $X \longrightarrow 2^{\mathscr{K}}_{{\rm cpt}}$, that in turn leads us to a measurable selection, say $u_{n}' : X \longrightarrow \mathscr{K}$ by Theorem \ref{equivalent}, that satisfies $u_{n}' (x) \in \Delta_{n}'(x)$. 

For $B = B_{X} \times B_{\mathscr{K}} \in \mathcal{B}_{X} \times \mathcal{B}_{\mathscr{K}}$, define  a sequence of measures $\left\{ \gamma_{n}' \right\}_{n\, \in\, \mathbb{Z}_{+}}$ supported on $X \times \mathscr{K}$ given by 
\[ \gamma_{n}' (B)\ \ =\ \ \int_{B_{X}} \kappa_{n}' \left( x,\, B_{\mathscr{K}} \right) \mathrm{d} \mu\ \ \ \text{where}\ \ \ \kappa_{n}' \left( x,\, B_{\mathscr{K}} \right)\ \ =\ \ \begin{cases} 1 & \text{if}\ u_{n}' (x) \in B_{\mathscr{K}}; \\ 0 & \text{otherwise}. \end{cases} \] 

Suppose for a fixed $\omega^{n} \in \Sigma_{N}^{n}$, we define the action of a map $\mathcal{G}_{\omega^{n}} : X \times \mathscr{K} \longrightarrow X \times \mathscr{K}$ in such a fashion that for any $0 \le j \le n$, we have $\mathcal{G}_{\omega^{n}}^{j} \left( x,\, v \right) = \left( T_{\pi_{j} \left( \omega^{n} \right)} x,\, \mathcal{R}_{T_{\pi_{j - 1} \left( \omega^{n} \right)} x} \circ \cdots \circ \mathcal{R}_{T_{\pi_{0} \left( \omega^{n} \right)} x} v \right)$, then, we consider the sequence of measures $\nu_{n}'$ on $X \times \mathscr{K}$ defined by
\[ \nu_{n}' (B)\ \ =\ \ \frac{1}{n} \frac{1}{N^{n}} \sum_{\omega^{n}\, \in\, \Sigma_{N}^{n}}\ \sum_{j\, =\, 0}^{n - 1}\ \left( \mathcal{G}_{\omega^{n}} \right)_{*}^{j} \gamma_{n}' (B)\ \ =\ \ \frac{1}{n} \frac{1}{N^{n}} \sum_{\omega^{n}\, \in\, \Sigma_{N}^{n}}\ \sum_{j\, =\, 0}^{n - 1}\ \gamma_{n}' \left( \mathcal{G}_{\omega^{n}}^{-j} (B) \right). \] 

Clearly, $\left\{ \gamma_{n}' \right\}_{n\, \in\, \mathbb{Z}_{+}}$ and thus, $\left\{ \nu_{n}' \right\}_{n\, \in\, \mathbb{Z}_{+}}$ are sequences of measures in $\mathbb{M}'(\mu)$. Since $\mathbb{M}'(\mu)$ is a compact metric space, there exists a subsequence $\left\{ \nu_{n_{k}}' \right\}$ that converges to some some measure $\left( \nu' \right)^{*} \in  \mathbb{M}'(\mu)$. Furthermore, applying Theorem \ref{kingmanomegafreel}, we get $\displaystyle{\int \Phi \mathrm{d} \left( \nu' \right)^{*} = \int \Phi^{*} \mathrm{d} \mu}$. Analogously, one can prove $\displaystyle{\int \Phi \mathrm{d} \left( \nu' \right)_{*} = \int \Phi_{*} \mathrm{d} \mu}$. 
\end{proof}

Hence, for $\mu$-almost every $x \in X$, there exist functions $\left(u'\right)_{*}, \left(u'\right)^{*} : X \longrightarrow \mathscr{K}$ such that 
\begin{eqnarray*} 
\Phi_{*} \left( x \right) & = & \lim_{n\, \to\, \infty} \frac{1}{n} \frac{1}{N^{n}} \sum_{\omega^{n}\, \in\, \Sigma_{N}^{n}}\ \sum_{j\, =\, 0}^{n - 1}\ \Phi \left( T_{\pi_{j} \left( \omega^{n} \right)} x,\, \mathcal{R}_{T_{\pi_{j - 1} \left( \omega^{n} \right)} x} \circ \cdots \circ \mathcal{R}_{T_{\pi_{0} \left( \omega^{n} \right)} x} \left( u' \right)_{*} (x) \right); \\ 
\Phi^{*} \left( x \right) & = & \lim_{n\, \to\, \infty} \frac{1}{n} \frac{1}{N^{n}} \sum_{\omega^{n}\, \in\, \Sigma_{N}^{n}}\ \sum_{j\, =\, 0}^{n - 1}\ \Phi \left( T_{\pi_{j} \left( \omega^{n} \right)} x,\, \mathcal{R}_{T_{\pi_{j - 1} \left( \omega^{n} \right)} x} \circ \cdots \circ \mathcal{R}_{T_{\pi_{0} \left( \omega^{n} \right)} x} \left( u' \right)^{*} (x) \right). 
\end{eqnarray*} 

Since the remainder of the current section is only, but an argument analogous to the ones made in Sections \ref{invariantmsb} and \ref{action}, we merely state the necessary definitions and results. Readers may find it interesting to prove the same, using the same ideas. 

\begin{definition}
We say that the map $\mathcal{H} : X \longrightarrow \mathbb{R}^{d}$ given by $\mathcal{H} (x) = V_{E_{0}}$ where $E_{0} = \{ x \}$ is a \emph{measurable sub-bundle} if one of the statements in Theorem \ref{equivalent:4} holds. Further, a measurable sub-bundle is called an \emph{$\omega$-free measurable sub-bundle} if $L(x) V_{E_{0}} = V_{E_{1}}$, where $E_{0} =\{ x \}$ and $E_{1} = \bigcup\limits_{\omega^{1}\, \in\, \Sigma_{N}^{1}} T_{\omega^{1}} x$, for $\mu$- almost every $x \in X$. 
\end{definition}

\begin{proposition} 
\label{operatoronvx}
For the simultaneous dynamical system defined on $X \times \mathbb{R}^{d}$, as given in Equation \eqref{actionofR} and considered in Theorem \ref{thm:oselnoomega}, let $\mathcal{H} : X \longrightarrow \mathbb{R}^{d}$ be an $\omega$-free measurable sub-bundle. Then, for $\mu$-almost every $x \in X$, we have 
\begin{eqnarray*} 
\lim_{n\, \to\, \infty} \frac{1}{n} \frac{1}{N^{n}} \sum_{\omega^{n}\, \in\, \Sigma_{N}^{n}} \log \left\| \left. L_{\omega^{n}}^{n} (x) \right|_{V_{E_{0}}} \right\|_{{\rm op}} & = & \max \left\{ \Lambda \left( E_{0},\, \{ v \} \right)\ :\ v \in V_{E_{0}} \setminus \left\{ 0 \right\} \right\} \\ 
\lim_{n\, \to\, \infty} \frac{1}{n} \frac{1}{N^{n}} \sum_{\omega^{n}\, \in\, \Sigma_{N}^{n}} \log \left\| \left. \left( L_{\omega^{n}}^{n} (x) \right)^{-1} \right|_{V_{E_{0}}} \right\|^{-1}_{{\rm op}} & = & \min \left\{ \Lambda \left( E_{0},\, \{ v \} \right)\ :\ v \in V_{E_{0}} \setminus \left\{ 0 \right\} \right\},
\end{eqnarray*} 
where $E_{0} = \{ x \}$ and $L_{\omega^{n}}^{n} (x) = \prod \limits_{j\, =\, 0}^{n - 1} L \left( T_{\pi_{j} \left( \omega^{n} \right)} (x) \right)$. 
\end{proposition}

Given a $\mu$-typical point $x \in X$, consider $E_{0} = \{ x \}$. Let $\mathcal{H}$ be an $\omega$-free measurable sub-bundle. Consider a linear map 
\[ L(x)\ \ :\ \ \mathbb{R}^{d}\ =\ V_{E_{0}}^{\perp} \oplus V_{E_{0}}\ \longrightarrow\ \mathbb{R}^{d} = V_{E_{1}}^{\perp} \oplus V_{E_{1}} \]
given by 
\begin{eqnarray} 
\label{eqn:decompLomegafee}
L (x)\ \ =\ \ \begin{pmatrix} L_{11} (x) & 0 \\ L_{21} (x) & L_{22} (x) \end{pmatrix},\ \ \ \text{where}\ \ \ L_{11} (x) & : & V_{E_{0}}^{\perp} \longrightarrow V_{E_{1}}^{\perp}, \nonumber \\ 
L_{21} (x) & : & V_{E_{0}}^{\perp} \longrightarrow V_{E_{1}}\ \ \text{and} \nonumber \\ 
L_{22} (x) & : & V_{E_{0}} \longrightarrow V_{E_{1}}. 
\end{eqnarray} 
Moreover, we have $\log^{+} \left\| L_{11}(x)^{\pm} \right\|_{{\rm op}},\ \log^{+} \left\| L_{21}(x)^{\pm} \right\|_{{\rm op}}$ and $\log^{+} \left\| L_{22}(x)^{\pm} \right\|_{{\rm op}}$ to be $\mu$-integrable, since $\log^{+} \left\| L(x)^{\pm} \right\|_{{\rm op}}$ is $\mu$-integrable, by our hypothesis in Theorem \ref{thm:oselnoomega}. Then, for $\mu$-almost every $x \in X$, taking $E_{0} = \{ x \}$, we have 
\begin{enumerate} 
\item For every $u \in V_{E_{0}}^{\perp} \setminus \{ 0 \}$ and $v \in V_{E_{0}}$, we have 
\begin{eqnarray*} 
& & \limsup_{n\, \to\, \infty} \frac{1}{n} \frac{1}{N^{n}} \sum_{\omega^{n}\, \in\, \Sigma_{N}^{n}} \log \left\| \prod\limits_{j\, =\, 0}^{n - 1} \left( L_{11} \left( T_{\pi_{j} \left( \omega^{n} \right)} x \right) \right) u \right\| \\ 
& = & \limsup_{n\, \to\, \infty} \frac{1}{n} \frac{1}{N^{n}} \sum_{\omega^{n}\, \in\, \Sigma_{N}^{n}} \log \left\| \prod\limits_{j\, =\, 0}^{n - 1} \left( L \left( T_{\pi_{j} \left( \omega^{n} \right)} x \right) \right) \left( u + v \right) \right\|. 
\end{eqnarray*} 
\item Suppose $\displaystyle{\lim_{n\, \to\, \infty} \frac{1}{n} \frac{1}{N^{n}} \sum_{\omega^{n}\, \in\, \Sigma_{N}^{n}} \log \left\| \prod\limits_{j\, =\, 0}^{n - 1} \left( L_{11} \left( T_{\pi_{j} \left( \omega^{n} \right)} x \right) \right) u \right\|}$ exists for some vector $u \in V_{E_{0}}^{\perp} \setminus \{ 0 \}$, then so does $\displaystyle{\lim_{n\, \to\, \infty} \frac{1}{n} \frac{1}{N^{n}} \sum_{\omega^{n}\, \in\, \Sigma_{N}^{n}} \log \left\| \prod\limits_{j\, =\, 0}^{n - 1} \left( L \left( T_{\pi_{j} \left( \omega^{n} \right)} x \right) \right) \left( u + v \right) \right\|}$ for every $v \in V_{E_{0}} \setminus \{ 0 \}$. Moreover, the limits are equal. 
\end{enumerate} 

One can then follow the ideas mentioned in Section \ref{partcomega}, use the results we have thus far proved in the case of the set-valued dynamics $R$ to complete the proof of part $C$ of Theorem \ref{thm:oselnoomega}.

\end{document}